\documentclass[preprint,11pt]{elsarticle}
\usepackage{amsmath}
\usepackage{fullpage, amsthm,amsfonts,url,amssymb,stmaryrd, mathrsfs}
\usepackage{mathrsfs}
\usepackage{amssymb,amscd}
\usepackage{color}
\usepackage{chemarrow}
\usepackage{pictexwd,dcpic}
\usepackage[textwidth=18cm,centering]{geometry}
\usepackage{extarrows}
\usepackage[colorlinks]{hyperref}

\makeatletter
\def\ps@pprintTitle{%
 \let\@oddhead\@empty
 \let\@evenhead\@empty
 \def\@oddfoot{\centerline{\thepage}}%
 \let\@evenfoot\@oddfoot}
\makeatother

\newtheorem{theorem}{Theorem}[section]
\newtheorem{corollary}[theorem]{Corollary}
\newtheorem{lemma}[theorem]{Lemma}
\newtheorem{proposition}[theorem]{Proposition}

\newtheorem{remark}[theorem]{Remark}

\newcommand{\hooklongrightarrow}{\lhook\joinrel\longrightarrow}
\newcommand{\twoheadlongrightarrow}{\relbar\joinrel\twoheadrightarrow}

\newcommand{\ra}{\rightarrow}
\newcommand{\lra}{\longrightarrow}

\newcommand{\ul}{\underline}

\newcommand{\bA}{\mathbb A}

\newcommand{\Q}{\mathbb Q}

\newcommand{\bT}{\mathbb T}
\newcommand{\Z}{\mathbb Z}

\newcommand{\bU}{\mathbb U}
\newcommand{\bI}{\mathbb I}

\newcommand{\cL}{\mathcal L}
\newcommand{\co}{\mathcal O}

\newcommand{\cR}{\mathcal R}

\newcommand{\cC}{\mathcal C}
\newcommand{\cS}{\mathcal S}
\newcommand{\cD}{\mathcal D}

\newcommand{\cT}{\mathcal T}
\newcommand{\cM}{\mathcal M}
\newcommand{\cF}{\mathcal F}

\newcommand{\cE}{\mathcal E}

\newcommand{\fn}{\mathfrak n}

\newcommand{\fm}{\mathfrak{m}}
\newcommand{\ub}{\mathfrak b}
\newcommand{\fl}{\mathfrak l}
\newcommand{\fp}{\mathfrak p}

\newcommand{\ug}{\mathfrak g}

\newcommand{\fX}{\mathfrak X}

\newcommand{\ft}{\mathfrak t}
\newcommand{\fa}{\mathfrak a}

\newcommand{\fI}{\mathfrak I}
\newcommand{\fd}{\mathfrak d}

\newcommand{\sE}{\mathscr E}

\newcommand{\sW}{\mathscr W}

\DeclareMathOperator{\pol}{\mathrm pol}

\DeclareMathOperator{\tr}{\mathrm tr}
\DeclareMathOperator{\GL}{\mathrm GL}
\DeclareMathOperator{\gr}{\mathrm gr}

\DeclareMathOperator{\Fil}{\mathrm Fil}
\DeclareMathOperator{\Res}{\mathrm Res}

\DeclareMathOperator{\Gal}{\mathrm Gal}
\DeclareMathOperator{\Hom}{\mathrm Hom}

\DeclareMathOperator{\End}{\mathrm End}

\DeclareMathOperator{\rig}{\mathrm rig}
\DeclareMathOperator{\an}{\mathrm an}
\DeclareMathOperator{\Spec}{\mathrm Spec}

\DeclareMathOperator{\Ind}{\mathrm Ind}
\DeclareMathOperator{\unr}{\mathrm unr}

\DeclareMathOperator{\lp}{\mathrm lp}
\DeclareMathOperator{\Ker}{\mathrm Ker}
\DeclareMathOperator{\pr}{\mathrm pr}

\DeclareMathOperator{\Ext}{\mathrm Ext}
\DeclareMathOperator{\usl}{\mathfrak sl}
\DeclareMathOperator{\Spm}{\mathrm Spm}
\DeclareMathOperator{\Spf}{\mathrm Spf}
\DeclareMathOperator{\Ima}{\mathrm Im}

\DeclareMathOperator{\id}{\mathrm id}

\DeclareMathOperator{\dett}{\mathrm det}
\DeclareMathOperator{\alg}{\mathrm alg}

\DeclareMathOperator{\soc}{\mathrm soc}

\DeclareMathOperator{\Supp}{\mathrm Supp}

\DeclareMathOperator{\Ad}{\mathrm Ad}

\DeclareMathOperator{\Rep}{\mathrm Rep}

\DeclareMathOperator{\st}{\mathrm st}
\DeclareMathOperator{\St}{\mathrm St}

\DeclareMathOperator{\WD}{\mathrm WD}
\DeclareMathOperator{\HT}{\mathrm HT}

\DeclareMathOperator{\rk}{\mathrm rk}

\DeclareMathOperator{\wt}{\mathrm wt}

\DeclareMathOperator{\tri}{\mathrm tri}

\DeclareMathOperator{\sm}{\mathrm sm}

\DeclareMathOperator{\univ}{\mathrm univ}
\DeclareMathOperator{\loc}{\mathrm loc}
\DeclareMathOperator{\hH}{\mathrm H}

\DeclareMathOperator{\val}{\mathrm val}
\begin{document}
  \title{Simple $\cL$-invariants for $\GL_n$}
  \author{Yiwen Ding}
\ead{yiwen.ding@bicmr.pku.edu.cn}
\address{B.I.C.M.R., Peking University,
No.5 Yiheyuan Road Haidian District,
Beijing, P.R. China 100871}
\begin{abstract}Let $L$ be a finite extension of $\Q_p$, and $\rho_L$ be an $n$-dimensional  semi-stable non crystalline $p$-adic representation of $\Gal_L$ with full monodromy rank. Via a study of  Breuil's (simple) $\cL$-invariants, we attach to $\rho_L$ a locally $\Q_p$-analytic representation $\Pi(\rho_L)$ of $\GL_n(L)$, which carries the exact information of the Fontaine-Mazur simple $\cL$-invariants of $\rho_L$. When $\rho_L$ comes from an automorphic representation of $G(\bA_{F^+})$ (for a unitary group $G$ over a totally real filed $F^+$ which is compact at infinite places and $\GL_n$ at $p$-adic places), we prove under mild hypothesis that $\Pi(\rho_L)$ is a subrerpresentation of the associated Hecke-isotypic subspaces of the Banach spaces of $p$-adic automorphic forms on $G(\bA_{F^+})$. In other words, we prove the equality of Breuil's simple $\cL$-invariants and Fontaine-Mazur simple $\cL$-invariants.
\end{abstract}\maketitle
\tableofcontents
\section{Introduction and notation}\addtocontents{toc}{\protect\setcounter{tocdepth}{1}}
\noindent This paper seeks to explore new cases  in  $p$-adic Langlands program (see \cite{Br0} for a survey), where a  basic problem  (and a starting point) is to find the Galois information, that is lost when passing from an $n$-dimensional de Rham $p$-adic Galois representation $\rho_L: \Gal_L\ra \GL_n(E)$ to its associated Weil-Deligne representation $\WD(\rho_L)$, on the automorphic side, e.g. in the Banach representations or locally $\Q_p$-analytic representations of $\GL_n(L)$. Here $L$ is a finite extension of $\Q_p$ of degree $d_L$, and $E$ is the coefficient field of the representations that we consider in this paper, which is also finite over $\Q_p$ and sufficiently large containing all the embeddings of $L$ in $\overline{\Q_p}$.
In this paper, we consider the case where $\rho_L$ is non-critical  semi-stable non-crystalline with the monodromy operator $N$ on $D_{\st}(\rho_L)$ satisfying $N^{n-1}\neq 0$.

\subsection*{Fontaine-Mazur simple $\cL$-invariants}
\noindent When passing from $\rho_L$ (under the above condition) to its associated Weil-Deligne representation, the lost information (beside of the Hodge-Tate weights) can be concretely described via the so-called \emph{Fontaine-Mazur $\cL$-invariants}. We study \emph{some of} the Fontaine-Mazur $\cL$-invariants in this paper, which we call Fontaine-Mazur \emph{simple} $\cL$-invariants. Let  $D:=D_{\rig}(\rho_L)$, which is trianguline and admits a unique triangulation: $(\delta_1,\cdots, \delta_n)$  (where the $\delta_i$'s are continuous characters of $L^{\times}$ in $E^{\times}$). The simple $\cL$-invariants are those characterizing the consecutive extensions $D_i^{i+1}$ of $\cR_E(\delta_{i+1})$ by $\cR_E(\delta_i)$ for $i\in \Delta=\{1,\cdots, n-1\}$ (which we identify with the set of simple roots of $\GL_n$) inside $D$.
\\

\noindent In fact, since $N^{n-1}\neq 0$ on $D_{\st}(\rho_L)$, we have that $D_i^{i+1}$ is semi-stable non-crystalline, and $\delta_i\delta_{i+1}^{-1}=\varepsilon\prod_{\sigma:L\hookrightarrow E} \sigma^{k_{\sigma}}$ where $k_{\sigma}\in \Z_{\geq 0}$ for all $\sigma: L\hookrightarrow E$, and $\varepsilon$ denotes the cyclotomic character. We have a natural perfect pairing (induced from the Tate pairing, cf. \S~\ref{sec: lgln-LFM1}\big):
\begin{equation}\label{equ: lgln-perf0}
  \Ext^1_{(\varphi,\Gamma)}\big(\cR_E(\delta_{i+1}), \cR_E(\delta_i)\big) \times \Hom(L^{\times}, E) \xlongrightarrow{\langle \cdot, \cdot \rangle} E\ \big(\cong \hH^2_{(\varphi,\Gamma)}\big(\cR_E(\delta_i\delta_{i+1}^{-1})\big)\big),
\end{equation}
where $\Hom(L^{\times}, E) $ denotes the $(d_L+1)$-dimensional $E$-vector space of $E$-valued additive characters on $L^{\times}$ and is naturally isomorphic to $\Ext^ 1_{(\varphi,\Gamma)}(\cR_E(\delta_{i+1}), \cR_E(\delta_{i+1}))$. We let
\begin{equation*}
 \label{equ: lgln-LFM}\cL(\rho_L)_i:=\cL(D)_i \subset \Hom(L^{\times},E)
\end{equation*}
be the $E$-vector subspace orthogonal to $[D_i^{i+1}]$ via the perfect pairing (\ref{equ: lgln-perf0}), which is thus $d_L$-dimensional and is determined by $D_i^{i+1}$ (and vice versa). Put $\cL(\rho_L):=\cL(D):=\{\cL(D)_i\}_{i\in \Delta}$.

\subsection*{Breuil's simple $\cL$-invariants}
\noindent First introduce some notation. Let $\alpha\in E^{\times}$ be such that $\{\alpha, q_L \alpha, \cdots, q_L^{n-1} \alpha\}$ are the eigenvalues of $\varphi^{f_L}$ on $D_{\st}(\rho_L)$, where $f_L$ denotes the unramified degree of $L$ over $\Q_p$, and $q_L:=p^{f_L}$. Let $\ul{\lambda}$ be the dominant weight of $T(L)$ \big(with respect to the Borel subgroup $B$ of upper triangular matrices\big) associated to the Hodge-Tate weights $\HT(\rho_L)$ of $\rho_L$. Let $\St_n^{\infty}(\ul{\lambda})$ (resp. $\St_n^{\an}(\ul{\lambda})$) denote the locally algebraic (resp. locally $\Q_p$-analytic) Steinberg representation of weight $\ul{\lambda}$, and let $v_{\overline{P}}^{\infty}(\ul{\lambda})$ be the locally algebraic generalized Steinberg representation of $\GL_n(L)$ with respect to $\overline{P}$ (see \S \ref{sec: lgln-LB-not} for details) for a parabolic subgroup $\overline{P}\supset \overline{B}$ (where $\overline{B}$ is the Borel subgroup of lower triangular matrices). Let $\St_n^{\infty}(\alpha, \ul{\lambda}):=\St_n^{\infty}(\ul{\lambda})\otimes_E \unr(\alpha)\circ \dett$, $\Sigma(\alpha, \ul{\lambda}):=\St_n^{\an}(\ul{\lambda})\otimes_E \unr(\alpha)\circ \dett$,  and $v_{\overline{P}}^{\infty}(\alpha, \ul{\lambda}):=v_{\overline{P}}^{\infty}(\ul{\lambda})\otimes_E \unr(\alpha)\circ \dett$, where $\unr(\alpha)$ denotes the unramified character of $L^{\times}$ sending uniformizers to $\alpha$. For $i\in \Delta$, denote by $\overline{P}_i\supsetneq \overline{B}$ the minimal parabolic subgroup with respect to $i$ (i.e. $i$ is the unique simple root of the Levi subgroup of $\overline{P}_i$ containing the torus subgroup $T$). The following theorem plays a crucial role in finding the counter part of the Fontaine-Mazur simple $\cL$-invariants in the locally analytic representations of $\GL_n(L)$ (which we refer to as \emph{Breuil's simple $\cL$-invariants}).

\begin{theorem}[cf. Theorem \ref{cor: lgln-key}]For $i\in \Delta$, there exists a natural isomorphism of $E$-vector spaces
\begin{equation}\label{equ: lgln-LB}
  \Hom(L^{\times},E)\xlongrightarrow{\sim} \Ext^1_{\GL_n(L)}\big(v_{\overline{P}_i}^{\infty}(\alpha, \ul{\lambda}), \St_n^{\an}(\alpha, \ul{\lambda})\big) ,
\end{equation}
where the $\Ext^1$ is taken in the abelian category of admissible locally $\Q_p$-analytic representations of $\GL_n(L)$.
\end{theorem}
\noindent We refer to Remark \ref{rem: lgln-ext3} (ii) for an explicit description of the isomorphism. Let  $V$  be an $E$-vector subspace of $\Hom(L^{\times},E)$ of dimension $d$, $i\in \Delta$.  Let $\{\psi_1,\cdots, \psi_d\}$ be a basis of $V$, and we put
\begin{equation*}
\tilde{\Sigma}_i(\alpha, \ul{\lambda}, V):=\tilde{\Sigma}_i(\alpha, \ul{\lambda}, \psi_1)\oplus_{\St_n^{\an}(\alpha, \ul{\lambda})}\tilde{\Sigma}_i(\alpha, \ul{\lambda}, \psi_2)\oplus_{\St_n^{\an}(\alpha, \ul{\lambda})}\cdots \oplus_{\St_n^{\an}(\alpha, \ul{\lambda})}\tilde{\Sigma}_i(\alpha, \ul{\lambda}, \psi_d),
\end{equation*}
where $\tilde{\Sigma}_i(\alpha, \ul{\lambda}, \psi_j)$ denotes the image of $\psi_j$ via (\ref{equ: lgln-LB}). Thus  $\tilde{\Sigma}_i(\alpha, \ul{\lambda}, V)$ is isomorphic to  an extension of $v_{\overline{P}_i}^{\infty}(\alpha, \ul{\lambda})^{\oplus d_L}$ by $\St_n^{\an}(\alpha, \ul{\lambda})$, and is independent of the choice of the basis of $V$. We put
\begin{equation*}
  \tilde{\Sigma}(\alpha, \ul{\lambda}, \cL(\rho_L)):=\oplus^{i\in \Delta}_{\St_n^{\an}(\alpha, \ul{\lambda})} \tilde{\Sigma}_i(\alpha, \ul{\lambda}, \cL(\rho_L)_i).
\end{equation*}
This is a locally $\Q_p$-analytic representation of $\GL_n(L)$, and is isomorphic to an extension of $\oplus_{i\in \Delta} v_{\overline{P}_i}^{\infty}(\alpha, \ul{\lambda})^{d_L}$ by $\St_n^{\an}(\alpha,\ul{\lambda})$. By (\ref{equ: lgln-LB}), $\tilde{\Sigma}(\alpha, \ul{\lambda}, \cL(\rho_L)_{\Delta})$ carries the exact information of $\{\alpha, \ul{\lambda}, \cL(\rho_L)\}$, i.e. the information on the Weil-Deligne representation associated to $\rho_L$, the Hodge-Tate weights of $\rho_L$, and the simple $\cL$-invariants of $\rho_L$.

\subsection*{Local-global compatibility}\noindent
The precedent discussion is purely \emph{local}. Suppose now $\rho_L$ appears on the patched eigenvariety \cite{BHS1} (hence we also assume the so-called Taylor-Wiles hypothesis, see \S~\ref{sec: lgln-pbr} for details). In this case,  one can associate to $\rho_L$  an admissible unitary Banach representation $\widehat{\Pi}(\rho_L)$ of $\GL_n(L)$ as in \cite{CEGGPS} (where the method is \emph{global}), which is expected to be the right representation (up to multiplicities) corresponding to $\rho_L$. We have (e.g. by global triangulation theory, see Proposition \ref{lg-ind})
\begin{equation}\label{equ: lgln-llc}
  \St_n^{\infty}(\alpha, \ul{\lambda})\hooklongrightarrow \widehat{\Pi}(\rho_L).
\end{equation}
The following theorem is the main result of the paper (see Theorem \ref{thm: lgln-main}, Remark \ref{rem: lgln-lg} (i))
\begin{theorem}\label{thm: lgln-lg}Keep the assumption, the injection (\ref{equ: lgln-llc}) extends uniquely to a morphism
\begin{equation}\label{equ: lgln-exiL}
  \tilde{\Sigma}(\alpha, \ul{\lambda}, \cL(\rho_L))\lra \widehat{\Pi}(\rho_L).
\end{equation}
Moreover, for $i\in \Delta$, $\psi\in \Hom(L^{\times}, E)$, (\ref{equ: lgln-llc}) can extend to a morphism
\begin{equation*}
\tilde{\Sigma}_i(\alpha, \ul{\lambda}, \psi) \lra \widehat{\Pi}(\rho_L)
\end{equation*}
if and only if $\psi \in \cL(\rho_L)_i$.
\end{theorem}
\noindent
 This theorem implies that $\cL(\rho_L)$ can be explicitly read out from $\widehat{\Pi}(\rho_L)$. We remark that the map (\ref{equ: lgln-exiL}) is injective if $\soc_{\GL_n(L)} \Sigma(\alpha, \ul{\lambda})$ is exactly $\St_n^{\infty}(\alpha, \ul{\lambda})$, which is actually true for $\GL_2(L)$ and $\GL_3(\Q_p)$. The author does not know whether it is true in general. However, one can show $\tilde{\Sigma}(\alpha, \ul{\lambda}, \cL(\rho_L))$ actually comes by push-forward from an extension $\Sigma(\alpha, \ul{\lambda}, \cL(\rho_L))$ of $\oplus_{i\in \Delta}v_{\overline{P}_i}^{\infty}(\alpha, \ul{\lambda})^{\oplus d_L}$ by a certain $\Sigma(\alpha, \ul{\lambda})\subset \St_n^{\an}(\alpha, \ul{\lambda})$. And we have $\soc_{\GL_n(L)} \Sigma(\alpha, \ul{\lambda}, \cL(\rho_L)) \cong \St_n^{\infty}(\alpha, \ul{\lambda})$. We refer to (\ref{stru}) for the structure of $\Sigma(\alpha, \ul{\lambda}, \cL(\rho_L))$, and to Theorem \ref{thm: lgln-main} for a similar statement for $\Sigma(\alpha, \ul{\lambda}, \cL(\rho_L))$.
\\

\noindent A similar result of Theorem \ref{thm: lgln-lg} in modular curve case was proved by Breuil (\cite{Br10}) using modular symbols, and played a crucial rule in Emerton's proof of Mazur-Tate-Teitelbaum exceptional zero conjecture (cf. \cite{Em05}). One can expect the results (or the methods) of this article have applications in the study of $p$-adic $L$-functions.
\subsection*{Sketch of the proof}
\noindent We give a sketch of the proof of Theorem \ref{thm: lgln-lg}. It is sufficient to prove the second part of the theorem. Let $i\in \Delta$, and $\psi\in \cL(D)_i$, the uniqueness of the $\cL$-invariants (i.e. the ``only if" part) follows from the existence (i.e. the ``if" part) and the global triangulation theory. We sketch how to obtain a non-zero map $\tilde{\Sigma}_i(\alpha,\ul{\lambda}, \psi)\ra \widehat{\Pi}(\rho_L)$ if $0\neq \psi\in \cL(D)_i$.
\\

\noindent We first prove a result on triangulation deformations of $\rho_L$. Let $\fX_{\rho_L}^{\tri}$ denote the rigid space over $E$ parameterizing trianguline deformations of $\rho_L$ (note also that since $N^{n-1}\neq 0$, $\rho_L$ admits a unique triangulation). We have a natural tangent map $d\kappa: \fX_{\rho_L}^{\tri}(E[\epsilon]/\epsilon^2) \ra \Hom(T(L), E)$. For $i\in \Delta$, denote by $\kappa_i: \Hom(T(L), E)\ra \Hom(L^{\times}, E)$, $(\psi_j)_{j=1,\cdots, n}\mapsto \psi_i-\psi_{i+1}$. The following proposition is of interest in it own right (cf. \S~\ref{sec: triDef}).
\begin{proposition}\label{thm: lgln-tridef}
  $\fX_{\rho_L}^{\tri}$ is formally smooth of dimension $1+\frac{n(n+1)}{2} d_L$, and the composition $\kappa_i\circ d \kappa$ factors through a surjective map $\fX_{\rho_L}^{\tri}(E[\epsilon]/\epsilon^2)\twoheadrightarrow \cL(\rho_L)_i$.
\end{proposition}
\noindent We unwind a little  on the construction of $\widehat{\Pi}(\rho_L)$. As in \cite{CEGGPS}, we have a continuous Banach representation $\Pi_{\infty}$ of $\GL_n(L)$ which is  equipped with a continuous action of certain patched Galois deformation ring $R_{\infty}$ commuting with the $\GL_n(L)$-action. By \cite{BHS1}, one can then construct the so-called patched eigenvariety $X_p(\overline{\rho})$ (where $\overline{\rho}$ is a certain global residue Galois representation) from the Jacquet-Emerton module $J_B(\Pi_{\infty}^{R_{\infty}-\an})$ (where ``$(\cdot)^{R_{\infty}-\an}$" denotes the locally $R_{\infty}$-analytic vectors, cf. \cite[\S~3.1]{BHS1}). We have a natural embedding (by construction) $X_p(\overline{\rho}) \hookrightarrow \widehat{T} \times (\Spf R_{\infty})^{\rig}$, and hence a natural morphism $\omega: X_p(\overline{\rho}) \ra \widehat{T}$, where $\widehat{T}$ denotes the rigid space over $E$ parametrizing continuous characters of $T(L)$.
Assume that we can associate to $\rho_L$ a classical point $x=(\chi, \fm_x)\in X_p(\overline{\rho})$ with $\fm_x$ the corresponding maximal ideal of $R_{\infty}[1/p]$ (and $\chi$ a locally algebraic character of $T(\Q_p)$ attached to $\rho_L$, see \S~\ref{sec:lg} for details), and we put $\widehat{\Pi}(\rho_L):=\Pi_{\infty}[\fm_x]$, the $E$-vector subspace of vectors annihilated by $\fm_x$. Consider the tangent map of $\omega$ at $x$:
\begin{equation*}
  d \omega_x: \cT_{X_p(\overline{\rho}), x} \lra \cT_{\widehat{T}, \omega(x)}\cong \Hom(T(L),E).
\end{equation*}
By a study of the geometry of $X_p(\overline{\rho})$ at $x$ (similarly as in \cite[\S~4]{BHS2}, but much easier), we can deduce from  Proposition \ref{thm: lgln-tridef}:
\begin{corollary}\label{cor: lgln-keyL} The composition $\kappa_i \circ d\omega_x$ factors through a projection of $E$-vector spaces $\cT_{X_p(\overline{\rho}), x} \twoheadrightarrow \cL(\rho_L)_i$.
\end{corollary}
\noindent By the corollary, there exists $\Psi\in \Ima d\omega_x$ with $\kappa_i(\Psi)=\psi$. Together with the construction of $X_p(\overline{\rho})$, one obtains an injection of $T(L)$-representations
\begin{equation}\label{equ: lgln-LinJB}
  \chi \otimes_E 1_{\Psi} \hooklongrightarrow J_B(\Pi_{\infty}^{R_{\infty}-\an})\{\fm_x\}
\end{equation}
where $\{\fm_x\}$ denotes the generalized eigenspace, i.e. the subspace of the vectors annihilated by a certain power of $\fm_x$, and where $1_{\Psi}$ denotes the extension of two trivial characters attached to $\Psi$. The non-critical assumption on $\rho_L$ then implies that (\ref{equ: lgln-LinJB}) is \emph{balanced} in the sense of \cite[Def. 0.8]{Em2}. We can then apply Emerton's adjunction formula (cf. \cite[Thm. 0.13]{Em2}, and we adopt the notation of \emph{loc. cit.}) to get a non-zero morphism
\begin{equation}\label{equ: lgln-Adj0}
  I_{\overline{B}}^{\GL_n}\big(\chi\delta_B^{-1}  \otimes_E 1_{\Psi}\big) \lra \Pi_{\infty}^{R_{\infty}-\an}\{\fm_x\}.
\end{equation}
By some locally analytic representation theory (and unwinding the bijection in (\ref{equ: lgln-LB})), one can actually show that (\ref{equ: lgln-Adj0}) induces a  non-zero morphism
\begin{equation*}
  \tilde{\Sigma}_i(\alpha, \ul{\lambda}, \psi) \lra  \Pi_{\infty}^{R_{\infty}-\an}[\fm_x].
\end{equation*}

\noindent We finally remark that the global context that we are working in is not important for the above arguments. Actually a key point above in finding the Fontaine-Mazur $\cL$-invariants in the space of $p$-adic automorphic representations is Corollary \ref{cor: lgln-keyL}, which one can expect to hold in more general setting (e.g. with $X_p(\overline{\rho})$ replaced by the classical eigenvarieties, see Remark \ref{rem: eig}, \ref{rem: lgln-lg} (ii)). We refer to the body of the text for more detailed and more precise statements.

\subsection*{Acknowledgement}
\noindent I would like to thank Christophe Breuil, Florian Herzig, Ruochuan Liu and Benjamin Schraen for answering my questions  or helpful remarks. This work was supported by EPSRC grant EP/L025485/1 and by Grant No. 8102600240 from B.I.C.M.R.. I also thank the referee for the careful reading and helpful comments.

\subsection{Notation}
\label{sec: lgln-LB-not}
\noindent Let $L$ (resp. $E$) be a finite extension of $\Q_p$ with $\co_L$ (resp. $\co_E$) its ring of integers and $\varpi_L$ (resp. $\varpi_E$) a uniformizer. Suppose $E$ is sufficiently large containing all the embeddings of $L$ in $\overline{\Q_p}$. Put
\begin{equation*}
  \Sigma_L:=\{\sigma: L\hookrightarrow \overline{\Q_p}\} =\{\sigma: L\hookrightarrow E\}.
\end{equation*}
Let $\val_L(\cdot)$ (resp. $\val_p$) be the $p$-adic valuation on $\overline{\Q_p}$ normalized by sending uniformizers of $\co_L$ (resp. of $\Z_p$) to $1$. Let $d_L:=[L:\Q_p]=|\Sigma_L|$, $e_L:=\val_L(\varpi_L)$ and $f_L:=d_L/e_L$. We have $q_L:=p^{f_L}=|\co_L/\varpi_L|$.
\\

\noindent
For a Lie algebra $\ug$ over $L$, $\sigma\in \Sigma_L$, let $\ug_{\sigma}:=\ug\otimes_{L,\sigma} E$ (which is  a Lie algebra over $E$). For $J\subseteq \Sigma_L$, let $\ug_J:=\prod_{\sigma\in J} \ug_{\sigma}$. In particular, we have $\ug_{\Sigma_L}\cong \ug\otimes_{\Q_p} E$.
\\

\noindent
Let $\Delta$ be the set of simple roots of $\GL_n$ (with respect to the Borel subgroup $B$ of upper triangular matrices), and we identify the set $\Delta$ with $\{1,\cdots, n-1\}$ such that $i\in \{1,\cdots, n-1\}$ corresponds to the simple root $(x_1,\cdots, x_n)\in \ft \mapsto x_i-x_{i+1}$, where $\ft$ denotes the Lie algebra of the torus $T$ of diagonal matrices.
To $I\subset \Delta$, we can associate a parabolic subgroup $P_I$ of $\GL_n$ containing $B$ such that $I$ is equal to the set of simple roots of $L_I$, the unique Levi subgroup of $P_I$ containing $T$ \big(thus $P_{\Delta}=\GL_n$, $P_{\emptyset}=B$\big).
Let $N_I$ be the nilpotent radical of $P_I$,  $Z_I$ the center of $L_I$ and $D_I$ the derived subgroup of $L_I$. Let $\overline{P}_I$ be the parabolic subgroup opposite to $P_I$, and $\overline{N}_I$ be the nilpotent radical of $\overline{P}_I$. Note that $L_I$ is also the unique Levi subgroup of $\overline{P}_I$ containing $T$. Let $N:=N_{\emptyset}$, $\overline{N}:=\overline{N}_{\emptyset}$. Let $\ug$, $\fp_I$, $\fn_I$, $\fl_I$, $\overline{\fp}_I$, $\overline{\fn}_I$, $\ft$ be the Lie algebra over $L$ of $\GL_n$, $P_I$, $N_I$, $L_I$, $\overline{P}_I$, $\overline{N}_I$, $T$ respectively, and let $\ub:=\fp_{\emptyset}$, $\overline{\ub}:=\overline{\fp}_{\emptyset}$,  $\fn:=\fn_{\emptyset}$ and $\overline{\fn}:=\overline{\fn}_{\emptyset}$.
\\

\noindent
Let $\ul{\lambda}:=(\lambda_{1,\sigma}, \cdots, \lambda_{n,\sigma})_{\sigma\in \Sigma_L}$ be a weight of $\ft_{\Sigma_L}$ . For $I\subseteq \Delta=\{1,\cdots, n-1\}$, we call that $\ul{\lambda}$ is $I$-dominant with respect to $B(L)$ (resp. with respect to $\overline{B}(L)$) if $\lambda_{i,\sigma}\geq \lambda_{i+1,\sigma}$ \big(resp. $\lambda_{i,\sigma} \leq \lambda_{i+1,\sigma}$\big) for all $i\in I$ and $\sigma\in \Sigma_L$. We denote by $X_{I}^+$ (resp. $X_{I}^-$) the set of $I$-dominant integral weights  of $\ft_{S}$ with respect to $B(L)$ (resp. with respect to $\overline{B}(L)$). Note that $\ul{\lambda}_S\in X_{I}^+$ if and only if $-\ul{\lambda}\in X_{I}^-$. For $\ul{\lambda}\in X_I^+$, there exists a unique irreducible algebraic representation, denoted by $L(\ul{\lambda})_I$, of $\Res_{\Q_p}^L L_I$ with highest weight $\ul{\lambda}$, i.e. $L(\ul{\lambda})_I^{N_I(L)}$ is one dimensional on which $\ft_{\Sigma_L}$ acts via $\ul{\lambda}$. Denote $\chi_{\ul{\lambda}}:=L(\ul{\lambda})_{\emptyset}$. If $\ul{\lambda}\in X_{\Delta}^+$, let $L(\ul{\lambda}):=L(\ul{\lambda})_{\Delta}$.  Let $I\subseteq J\subseteq \Delta$, and suppose $\ul{\lambda}\in X_J^+$, then we have $L(\ul{\lambda})_J^{N_I(L) \cap L_J(L)} \cong L(\ul{\lambda})_I$.
\\

\noindent Let $\ul{\lambda}$ be an integral weight, denote by $M(\ul{\lambda}):=\text{U}(\ug_{\Sigma_L})\otimes_{\text{U}(\ub_{\Sigma_L})} \ul{\lambda}$ (resp. $\overline{M}(\ul{\lambda}):=\text{U}(\ug_{\Sigma_L})\otimes_{\text{U}(\overline{\ub}_{\Sigma_L})} \ul{\lambda}$), and let $L(\ul{\lambda})$ (resp. $\overline{L}(\ul{\lambda})$) be the unique simple quotient of $M(\ul{\lambda})$ (resp. of $\overline{M}(\ul{\lambda})$). Actually, when $\ul{\lambda}\in X_{\Delta}^+$, $L(\ul{\lambda})$ is finite dimensional and  isomorphic to the algebraic representation $L(\ul{\lambda})$ introduced above (hence there is no conflict of notation). If moreover $\ul{\lambda}\in X_{\Delta}^+$ (i.e. $-\ul{\lambda}\in X_{\Delta}^-$) we have
\begin{equation*}
  \overline{L}(-\ul{\lambda})\cong L(\ul{\lambda})^{\vee}.
\end{equation*}

%\noindent
%.
%\\

\noindent Denote by $\sW$ ($\cong S_n$) the Weyl group of $\GL_n$, $s_i$ the simple reflection corresponding to $i\in \Delta$. Let  $\sW_{\Sigma_L}\cong S_n^{d_L}$ be the Weyl group of $\Res^L_{\Q_p} \GL_n$. For $\sigma\in L$, let $\sW_{\sigma}$ be the $\sigma$-th factor of $\sW_{\Sigma_L}$, which is isomorphic to $\sW$, and let $s_{i,\sigma}\in \sW_{\sigma}$ be the simple reflection corresponding to $i$. For $S\subseteq \Sigma_L$, denote by $\sW_S:=\prod_{\sigma\in S} \sW_{\sigma}$.
\\

\noindent
Let $A$ be an affinoid $E$-algebra. A locally $\Q_p$-analytic character $\delta: L^{\times} \ra A^{\times}$ induces a $\Q_p$-linear map
\begin{equation*}
    L \lra A, \ x\mapsto \frac{d}{dt}\delta(\exp(tx))|_{t=0},
\end{equation*}
and hence induces an $E$-linear map
\begin{equation*}L\otimes_{\Q_p} E\cong \prod_{\sigma\in \Sigma_L} E \lra A.
\end{equation*}
Thus there exist $\wt(\delta)_{\sigma}\in A$ for all $\sigma\in \Sigma_L$, called the $\sigma$-weight of $\delta$, such that the above map is given by $(a_{\sigma})_{\sigma\in \Sigma_L}\mapsto \sum_{\sigma\in \Sigma_L} a_\sigma \wt(\delta)_{\sigma}$. And we call  $\wt(\delta):=(\wt(\delta)_{\sigma})_{\sigma\in \Sigma_L}$ the \emph{weight} of $\delta$. For a weight $\ul{\lambda}$ of $\ft_{\Sigma_L}$ over $E$, we see $\chi_{\ul{\lambda}}$ is just the algebraic character of $T(L)$ of weight $\ul{\lambda}$.
\\

\noindent For a topological commutative group $M$, we denote by $\Hom(M,E)$ the $E$-vector space of  continuous additive $E$-valued characters on $M$. If $M$ is a locally $L$-analytic group, denote by $\Hom_{\sigma}(M,E)$ the $E$-vector space of locally $\sigma$-analytic characters on $M$ (i.e. the continuous characters, which are moreover locally $\sigma$-analytic as $E$-valued functions on $M$).
\\

\noindent We normalize the local class field theory by sending a uniformizer  to a (lift of the) geometric Frobenius. In this way, we obtain a bijection
\begin{equation*}
  \Hom(L^{\times}, E) \cong \Hom(\Gal_L,E).
\end{equation*}
Denote by $\varepsilon$ the cyclotomic character of $\Gal_L$ over $E$.
\\

\noindent
If $A$ is a commutative ring, $M$ an $A$ module and $I$ an ideal of $A$, we denote by $M[I]\subseteq M$ the $A$-submodule of elements killed by $I$ and by $M\{I\}:=\cup_{n\geq 1}M[I^n]$.\\

\noindent For simplicity, we let $G=\GL_n(L)$, and $Z$ be the center of $\GL_n(L)$. When $I$ is a singleton $\{i\}$, we also use $i$ to denote $I$ for subscripts (e.g. $P_i=P_{\{i\}}$).
\addtocontents{toc}{\protect\setcounter{tocdepth}{2}}
\section{Breuil's simple $\cL$-invariants}\label{sec: lgln-LB}
\noindent In this section, we study the extensions of some locally analytic generalized Steinberg representations. We show that such extensions can be parameterized by  Breuil's simple $\cL$-invariants.
\subsection{Preliminaries}\noindent We recall some useful notation and statements on locally analytic representations.

\subsubsection{Locally analytic representations}
 \noindent Let $H$ be a strictly paracompact locally $\Q_p$-analytic group (where strictly paracompact means that any open covering of $H$ can be refined into a covering by pairwise disjoint open subsets), and denote by $\cD(H,E)$ the locally convex $E$-algebra of $E$-valued distributions on $H$ (cf. \cite[Prop. 2.3]{ST02}). Let $\cM(H)$ denote the abelian category of (abstract) $\cD(H,E)$-modules. Note that one can embed $H$ into $\cD(H,E)$ by sending $h\in H$ to $\delta_h:=[f\mapsto f(h)]$ with $f\in \cC^{\an}(H,E)$, where $\cC^{\an}(H,E)$ denotes the $E$-vector spaces of $E$-values locally analytic functions on $H$.
\\

\noindent Let $V$ be a locally $(\Q_p$-)analytic representation of $H$ over $E$ (cf. \cite[\S~3]{ST02}), the continuous dual $V^{\vee}$ is naturally equipped with a $\cD(H,E)$-module structure given by
\begin{equation}\label{equ: lgln-duala}
  (\mu \cdot w)(v)=\mu([g\mapsto w(g^{-1}v)])
\end{equation}
for $\mu\in \cD(H,E)$, $w\in V^{\vee}$ and $v\in V$. Note that the function $[g\mapsto w(g^{-1}v)]$ on $H$ lies in $\cC^{\Q_p-\an}(H,E)$, so the action in (\ref{equ: lgln-duala}) is well defined.
\\

\noindent Let $V$, $W$ be locally $\Q_p$-analytic representations of $H$ over $E$. As in \cite[D\'ef. 3.1]{Sch11}, we put
\begin{equation}\label{equ: lgln-ext}
  \Ext^r_{H}(V,W):=\Ext^r_{\cM(H)}(W^{\vee}, V^{\vee}),
\end{equation}
where the latter denotes the $r$-th extension group in the abelian category $\cM(H)$. Let $Z'$ be a locally $\Q_p$-analytic closed subgroup of the center $Z$ of $H$, and $\chi$ be a locally $\Q_p$-analytic character of $Z'$, we denote by $\cM(H)_{Z',\chi}$ the full subcategory of $\cM(H)$ of $\cD(H,E)$-modules on which $Z'$ acts by $\chi^{-1}$ \big(where the $Z'$-action is induced by $Z'\hookrightarrow H\hookrightarrow \cD(H,E)$\big). If $Z'$ acts on both $V$ and $W$ via the character $\chi'$, we put
\begin{equation*}
  \Ext^r_{H,Z'=\chi}(V,W):=\Ext^r_{\cM(H)_{Z',\chi}}(W^{\vee},V^{\vee})
\end{equation*}
where the latter denotes the $r$-th extension group in $\cM(H)_{Z',\chi}$. In the case where $Z'=Z$, we denote $\cM(H)_{\chi}:=\cM(H)_{Z,\chi}$, and $\Ext^r_{H,\chi}(V,W):=\Ext^r_{H,Z=\chi}(V,W)$. In the case where $H/Z'$ is also a strictly paracompact locally $\Q_p$-analytic group, and $H\cong H/Z'\times Z'$ (as $\Q_p$-analytic manifolds), we have a natural equivalence of categories between  $\cM(H)_{Z',1}$ and $\cM(H/Z')$. Hence in this case we have
\begin{equation*}
  \Ext^r_{H,Z'=1}(V,W)\cong \Ext^r_{H/Z'}(V,W).
\end{equation*}
%One easily sees that similar results hold in the case where $\chi$ (as a character of $Z'$) can lift to a character of $H$.
\begin{proposition}[$\text{\cite[Cor. 3.3]{Sch11}}$]\label{prop: gpcoho}
Let $V$ be a finite dimensional locally $\Q_p$-analytic representation of $H$ over $E$, then we have
\begin{equation*}
  \Ext^r_H(1,V)\cong \hH^r_{\an}(H,V)
\end{equation*}
where $\hH^r_{\an}(H,V)$ denotes the $r$-th locally analytic group cohomology defined by Casselman and Wigner (\cite{CaWi}).
\end{proposition}
\noindent By \cite[Lem. 2.1.1]{Br16} and the proof, we have
\begin{lemma}\label{lem: ext0}
Let $V$ and $W$ be admissible locally $\Q_p$-analytic representations of $H$ over $E$, and let $M\in \Ext^1_{\cM(H)}(W^{\vee}, V^{\vee})$. Then $M$ is a coadmissible $\cD(H_0, E)$-module for a certain uniform pro-$p$ open compact subgroup $H_0$ of $H$. Hence $M^{\vee}$ (equipped with the strong topology) is an admissible locally $\Q_p$-analytic representation of $H$.
\end{lemma}
\noindent Keep the notation and assumption of Lemma \ref{lem: ext0}, by equivalence of categories (\cite[Thm. 6.3]{ST03}), we see that  $\Ext^1_H(V,W)$ consists of admissible locally analytic representations which are extensions of $V$ by $W$. If moreover, $Z'$ acts on both $V$ and $W$ via $\chi$, then $\Ext^1_{H,Z'=\chi}(V,W)$ consists of extensions of $V$ by $W$ on which $Z'$ acts via  $\chi$. \\

\noindent
Let $\cD^{\infty}(H,E)$ denote the strong continuous dual of the $E$-vector space (equipped with the finest locally convex topology) of locally constant $E$-valued functions on $H$, which is in fact the quotient algebra of $\cD(H,E)$ by the closed ideal generated by the Lie algebra of $H$. The map $h\mapsto \delta_h$ induces in fact an injection $H\hookrightarrow \cD^{\infty}(H,E)$. Let $\cM(H)^{\infty}$ denote the category of abstract $\cD^{\infty}(H,E)$-modules, which is thus a full subcategory of $\cM(H)$. For a locally analytic representation $V$ of $H$, it is easy to see $V^{\vee}\in \cM(H)^{\infty}$ if and only if the $H$-action on $V$ is smooth.
For a locally $\Q_p$-analytic closed subgroup $Z'\subseteq Z$, let $\chi$ be a smooth character of $Z'$,  we denote by $\cM(H)^{\infty}_{Z',\chi}$ the category of $\cD^{\infty}(H,E)$-modules on which $Z'$ acts via $\chi^{-1}$, and $\cM(H)^{\infty}_{\chi}:=\cM(H)^{\infty}_{Z,\chi}$. Let $\Rep^{\infty}(H)$ \big(resp. $\Rep^{\infty}_{Z'=\chi}(H)$\big) denote the category of smooth representations of $H$ over $E$ (resp. smooth representations of $H$ over $E$ on which $Z'$ acts via the character $\chi$). We have the following functor
\begin{equation}\label{lgln: fcsmooth}
 \Rep^{\infty}(H) \ra \cM(H)^{\infty}, V\mapsto V^{\vee}
\end{equation}
where $V^{\vee}$ denotes the abstract dual of $V$, which coincides with the continuous dual of $V$ (as $\cD^{\infty}(H,E)$-modules) if $V$ is equipped with the finest locally convex topology, and where the $\cD^{\infty}(H,E)$-action on $V^{\vee}$ is given by the same way as in (\ref{equ: lgln-duala}). This functor is fully faithful. By \cite[Cor. 0.2]{Sch11bis}, we have
\begin{proposition}\label{prop: smExtd}
  Let $V$, $W \in \Rep^{\infty}(H)$ be admissible representations. Then the functor in (\ref{lgln: fcsmooth}) induces natural bijections
  \begin{equation*}
    \Ext^r_{\Rep^{\infty}(H)}(V,W) \cong \Ext^r_{\cM^{\infty}(H)}(W^{\vee}, V^{\vee}), \ \forall r\in \Z_{\geq 0}.
  \end{equation*}
\end{proposition}
\begin{remark}
Let $Z'$ be a locally $\Q_p$-analytic closed subgroup of $Z$ such that $H\cong H/Z'\times Z'$ as $\Q_p$-analytic manifolds, and let $\chi$ be a smooth character of $Z'$ over $E$ such that $\chi$ can lift to a smooth character of $H$. We have an equivalence of categories $\Rep_{Z'=\chi}^{\infty}(H)\xrightarrow{\sim}\Rep^{\infty}(H/Z')$, $V\mapsto V\otimes_E \chi$, which induces $\cM^{\infty}(H)_{Z',\chi}\cong \cM^{\infty}(H/Z')$. Let $V$, $W\in \Rep^{\infty}_{Z'=\chi}(H)$ be admissible, by the above discussion and Proposition \ref{prop: smExtd}, we  deduce
  \begin{equation}\label{equ: extsm}
    \Ext^r_{\Rep^{\infty}_{Z'=\chi}(H)}(V,W)\cong \Ext^r_{\cM^{\infty}(H)_{Z',\chi}}(W^{\vee},V^{\vee}), \ \forall r\in \Z_{\geq 0}.
  \end{equation}
\end{remark}

\subsubsection{Locally analytic generalized Steinberg representations}
\noindent Let $\ul{\lambda}\in X_{\Delta}^+$, and $I\subseteq \Delta$. We put (recall $G=\GL_n(L)$)
\begin{eqnarray*}
    \bI_{\overline{P}_I}^G(\ul{\lambda})&:=& \big(\Ind_{\overline{P}_I(L)}^G L(\ul{\lambda})_I\big)^{\Q_p-\an},\\
    i_{\overline{P}_I}^G(\ul{\lambda})&:=&\big(\Ind_{\overline{P}_I(L)}^G 1\big)^{\infty} \otimes_E L(\ul{\lambda}). \\
  \end{eqnarray*}
  Note that by \cite{OS}, $i_{\overline{P}_I}^G(\ul{\lambda})$ is actually the locally algebraic subrepresentation of $\bI_{\overline{P}_I}^G(\ul{\lambda})$. If $J\supseteq I$, we have natural injections $\bI_{\overline{P}_J}^G(\ul{\lambda}) \hookrightarrow \bI_{\overline{P}_I}^G(\ul{\lambda})$ and  $i_{\overline{P}_J}^G(\ul{\lambda}) \hookrightarrow i_{\overline{P}_I}^G(\ul{\lambda})$. Put
  \begin{eqnarray*}
    v_{\overline{P}_I}^{\an}(\ul{\lambda})&:=&\bI_{\overline{P}_I}^G(\ul{\lambda})/\big(\sum_{J\supsetneq I} \bI_{\overline{P}_J}^G(\ul{\lambda})\big),\\
    v_{\overline{P}_I}^{\infty}(\ul{\lambda})&:=&i_{\overline{P}_I}^G(\ul{\lambda})/\big(\sum_{J\supsetneq I} i_{\overline{P}_J}^G(\ul{\lambda})\big).
  \end{eqnarray*}
We put $\St_n^{\an}(\ul{\lambda}):=v_{\overline{B}}^{\an}(\ul{\lambda})$,  $\St_n^{\infty}(\ul{\lambda}):=v_{\overline{B}}^{\infty}(\ul{\lambda})$, $\bI_{\overline{P}_I}^G:=\bI_{\overline{P}_I}^G(\ul{0})$, $i_{\overline{P}_I}^G:=i_{\overline{P}_I}^G(\ul{0})$, $v_{\overline{P}_I}^{\an}:=v_{\overline{P}_I}^{\an}(\ul{0})$, and $v_{\overline{P}_I}^{\infty}:=v_{\overline{P}_I}^{\infty}(\ul{0})$. For $\alpha\in E^{\times}$, put $\bI_{\overline{P}_I}^G(\alpha, \ul{\lambda}):=\bI_{\overline{P}_I}^G(\ul{\lambda})\otimes_E \unr(\alpha)\circ \dett$ where $\unr(\alpha): L^{\times} \ra E^{\times}$ is an unramified character sending uniformizers to $\alpha$. We define representations $i_{\overline{P}_I}^G(\alpha, \ul{\lambda})$, $v_{\overline{P}_I}^{\an}(\alpha, \ul{\lambda})$, $v_{\overline{P}_I}^{\infty}(\alpha, \ul{\lambda})$, $\St_n^{\an}(\alpha, \ul{\lambda})$ in a similar way.
The following theorem can be easily deduced from the classical result for smooth representations, by tensoring with $L(\ul{\lambda})$.
\begin{theorem}
 Let $\ul{\lambda}\in X_{\Delta}^+$, $I\subset \Delta$, we have a natural exact sequence of locally algebraic representations of $G$:
\begin{equation}\label{equ: lgln-lalgSt}
  0 \ra i_G^G(\ul{\lambda}) \ra \bigoplus_{\substack{I\subset K\subset \Delta \\|\Delta\setminus K|=1}}i_{\overline{P}_K}^G(\ul{\lambda}) \ra \bigoplus_{\substack{I\subset K\subset \Delta \\|\Delta\setminus K|=2}}i_{\overline{P}_K}^G(\ul{\lambda}) \ra
   \cdots \ra \bigoplus_{\substack{I\subset K\subset \Delta \\|K\setminus I|=1}}i_{\overline{P}_K}^G(\ul{\lambda}) \ra i_{\overline{P}_I}^G(\ul{\lambda}) \ra v_{\overline{P}_I}^{\infty}(\ul{\lambda}) \ra 0.
\end{equation}
\end{theorem}
\noindent We also have an analogue in the locally analytic setting:
\begin{theorem}[$\text{\cite[Thm. 4.2]{OSch}}$] Let $\ul{\lambda}\in X_{\Delta}^+$, $I\subset \Delta$, we have a natural exact sequence of locally $\Q_p$-analytic representations of $G$:
\begin{equation}\label{equ: lgln-lanSt}
  0 \ra \bI_G^G(\ul{\lambda}) \ra \bigoplus_{\substack{I\subset K\subset \Delta \\|\Delta\setminus K|=1}}\bI_{\overline{P}_K}^G(\ul{\lambda}) \ra \bigoplus_{\substack{I\subset K\subset \Delta \\|\Delta\setminus K|=2}}\bI_{\overline{P}_K}^G(\ul{\lambda}) \ra
   \cdots \ra \bigoplus_{\substack{I\subset K\subset \Delta \\|K\setminus I|=1}}\bI_{\overline{P}_K}^G(\ul{\lambda}) \ra \bI_{\overline{P}_I}^G(\ul{\lambda}) \ra v_{\overline{P}_I}^{\an}(\ul{\lambda}) \ra 0.
\end{equation}
\end{theorem}

\subsubsection{Locally $S$-analytic representations}\noindent Let $V$ be a locally $\Q_p$-analytic representation of $G$ over $E$. Recall that $V$ is equipped with a natural $\Q_p$-linear action of $\ug$ given by
\begin{equation*}
  X(v)=\frac{d}{dt}\exp(tX)v|_{t=0},
\end{equation*}
for $X \in \ug$ and $v\in V$.
This action  extends naturally to an $E$-linear action of $\ug_{\Sigma_L}$ on $V$. For $S\subseteq \Sigma_L$, a vector $v\in V$ is called \emph{locally $S$-analytic}, if the $\ug_{\Sigma_L}$-action on $v$ factors through $\ug_{\Sigma_S}$; $v$ is called $\text{U}(\ug_S)$\emph{-finite}, if the orbit $\text{U}(\ug_S) v$ is finite dimensional over $E$. It is clear that if $v$ is locally $S$-analytic then $v$ is $\text{U}(\ug_{\Sigma_L\setminus S})$-finite. The representation $V$ is called \emph{locally $S$-analytic} (resp.$\text{U}(\ug_S)$\emph{-finite}) if any vector $v\in V$ is locally $S$-analytic (resp. $\text{U}(\ug_S)$-finite).
\\

\noindent Let $V$ be an $E$-vector space of compact type, and let $\cC^{\Q_p-\an}(G,V)$ denote the $E$-vector space of locally $\Q_p$-analytic $V$-valued functions on $G$. This space is equipped with a natural locally convex topology (cf. \cite[\S~2]{ST02}) and with a  natural right regular locally $\Q_p$-analytic action of $G$. The $G$-action on $\cC^{\Q_p-\an}(G,V)$ induces a $\Q_p$-linear action of $\ug$ (hence an $E$-linear action of $\ug_{\Sigma_L}$):
\begin{equation}\label{equ: lgln-right}
  (Xf)(g)=\frac{d}{dt} f(g \exp(tX))|_{t=0},
\end{equation}
for $X\in \ug$, $f\in \cC^{\Q_p-\an}(G,V)$, and $g\in G$. We have another $\Q_p$-linear action of $\ug$ (inducing an $E$-linear action of $\ug_{\Sigma_L}$) on $\cC^{\Q_p-\an}(G,V)$ given by:
\begin{equation}\label{equ: lgln-left}
  (X\cdot f)(g)=\frac{d}{dt} f(\exp(-tX)g )|_{t=0}
\end{equation}
for $X\in \ug$, $f\in \cC^{\Q_p-\an}(G,V)$, and $g\in G$. For $S\subseteq \Sigma_L$, denote by $\cC^{S-\an}(G,V)$ the subspace of locally $S$-analytic $V$-valued functions, i.e. the functions annihilated by all $X_{\Sigma_L\setminus S}\in \ug_{\Sigma_L\setminus S} \hookrightarrow \ug_{\Sigma_L}$ via the action in (\ref{equ: lgln-right}) or equivalently via the action in (\ref{equ: lgln-left}) \big(since $(X\cdot f)(g)=(\Ad_g(X) f)(g)$\big).
\\

\noindent Let $I\subseteq \Delta$, recall that the Orlik-Strauch functor (cf. \cite{OS}, see also \cite[\S~2]{Br13I}) associates, to an object $M$ in the BGG category $\co^{\overline{\fp}_{I,\Sigma_L}}_{\alg}$ together with a finite length smooth admissible representation $\pi$ of $L_I(L)$ a locally $\Q_p$-analytic representation $\cF_{\overline{P}_I}^G(M,\pi)$ of $G$.
\begin{lemma}\label{lem: lgln-Sla}Let $I\subseteq \Delta$, $S\subseteq \Sigma_L$, $W_S$ a finite dimensional algebraic representation of $\overline{\fp}_{I,S}$, and $\pi$ a finite length smooth representation of $L_I(L)$. Then we have a natural isomorphism of locally $\Q_p$-analytic representations of $G$ (note the natural projection $\ug_{\Sigma_L}\twoheadrightarrow \ug_{S}$ induces $\co^{\overline{\fp}_{I,S}}_{\alg}\subseteq \co^{\overline{\fp}_{I,\Sigma_L}}_{\alg}$)
\begin{equation*}
  \cF_{\overline{P}_I}^G\big(\text{U}(\ug_S)\otimes_{\text{U}(\overline{\fp}_{I,S})} W_S, \pi\big)\cong \big(\Ind_{\overline{P}_I(L)}^G W_S^{\vee} \otimes_E \pi\big)^{S-\an}.
\end{equation*}
\end{lemma}
\begin{proof}
For $S'\subseteq \Sigma_L$, denote by $\fd_{I,S'}$ the kernel of the natural projection $\text{U}(\ug_{S'}) \otimes_{\text{U}(\overline{\fp}_{I,S'})} 1 \twoheadrightarrow 1$ (where the two ``$1$" denote the trivial representation of $\text{U}(\overline{\fp}_{I,S'})$ and $\text{U}(\ug_{S'})$ respectively). We have an exact sequence in $\co^{\overline{\fp}_{I,\Sigma_L}}_{\alg}$ (using $\text{U}(\ug_{\Sigma_L}) \cong \text{U}(\ug_S)\otimes_E \text{U}(\ug_{\Sigma_L\setminus S})$):
\begin{equation*}
 0 \ra \fd_{I,\Sigma_L\setminus S} \otimes_E \big(\text{U}(\ug_S)\otimes_{\text{U}(\overline{\fp}_{I,S})} W_S\big) \ra \text{U}(\ug_{\Sigma_L}) \otimes_{\text{U}(\overline{\fp}_{I,\Sigma_L})} W_S \xrightarrow{\phi} \text{U}(\ug_S)\otimes_{\text{U}(\overline{\fp}_{I,S})} W_S \ra 0.
\end{equation*}
Let $f\in  \big(\Ind_{\overline{P}_I(L)}^G W_S^{\vee} \otimes_E \pi\big)^{\Q_p-\an}$. It suffices to show that $f$ is locally $S$-analytic if and only if $f$ is annihilated by $\Ker \phi$, where the action of $\text{U}(\ug_{\Sigma_L}) \otimes_{\text{U}(\overline{\fp}_{I,\Sigma_L})} W_S$ on $f$ is induced by (see \cite{OS} \cite[\S~2]{Br13I} for details)
\begin{equation}\label{equ: dotwX}
 ( (X\otimes w)\cdot f)(-)=w((X\cdot f)(-))\in \cC^{\Q_p-\an}(G, \pi),
\end{equation}
for $X\otimes w\in \text{U}(\ug_{\Sigma_L}) \otimes_{E} W_S$, $g\in G$.
\\

\noindent
Let $n_1, \cdots, n_k$ be a basis of $\fn_{I,\Sigma_L\setminus S}$ over $E$.
Using the isomorphism $\text{U}(\ug_{\Sigma_L\setminus S}) \cong \text{U}(\fn_{I, \Sigma_L\setminus S})\otimes_E\text{U}(\overline{\fp}_{I,\Sigma_L\setminus S})$, we see $\fd_{I,\Sigma_L\setminus S}$ is the $E$-vector space spanned  by $n_1^{e_1}\cdots n_k^{e_k} \otimes 1$ with $e_i\in \Z_{\geq 0}$ for all $i$, and $\sum_{i=1}^k e_i>0$.  Thus, if $f$ is locally $S$-analytic, then $X\cdot f=0$ for all $X\in \fd_{I,\Sigma_L\setminus S} \otimes_E \text{U}(\ug_{S})$, and one can easily check by (\ref{equ: dotwX}) that  $f$ is annihilated by $\Ker \phi$.  Conversely, if $f$ is annihilated by $\Ker \phi$, we can deduce, using the above description of $\fd_{I,\Sigma_L \setminus S}$ and (\ref{equ: dotwX}), that $n_i\cdot f$=0 for $i=1,\cdots, k$.  However, using $f(\overline{p}g)=\overline{p}f(g)$ for any $\overline{p}\in \overline{P}_I(L)$, $g\in G$, and the fact that the action of $\overline{P}_I(L)$ on $W_S^{\vee}\otimes_E \pi$ is locally $S$-analytic, we deduce $Y\cdot f=0$ for all $Y\in \overline{\fp}_{I,\Sigma_L\setminus S}$. Hence $X\cdot f=0$ for all $X\in \ug_{\Sigma_L\setminus S}\cong \fn_{I,\Sigma_L \setminus S}\oplus \overline{\fp}_{I,\Sigma_L\setminus S}$, i.e. $f$ is locally $S$-analytic.  This concludes the proof.
\end{proof}
\begin{remark}\label{rem: actS}
  By Lemma \ref{lem: lgln-Sla} and the proof, (\ref{equ: dotwX}) induces an action  of $\text{U}(\ug_S)\otimes_{\text{U}(\overline{\fp}_{I,S})} W_S$ on $\big(\Ind_{\overline{P}_I(L)}^G W_S^{\vee} \otimes_E \pi\big)^{S-\an}$ satisfying
  \begin{equation*}
   ( (X_S \otimes w)\cdot f)(-)=w(X_S \cdot f(-))\in \cC^{S-\an}(G, \pi),
  \end{equation*}
  for $X_S\otimes w\in \text{U}(\ug_S)\otimes_{\text{U}(\overline{\fp}_{I,S})} W_S$ and $f\in \big(\Ind_{\overline{P}_I(L)}^G W_S^{\vee} \otimes_E \pi\big)^{S-\an}$.
\end{remark}
\noindent Let $M_S\in \co^{\overline{\fp}_{I,S}}_{\alg}$, $W_S$ be a finite dimensional representation of $\overline{\fp}_{I,S}$ which generates $M_S$, and let $\phi_S$ denote the natural projection $\text{U}(\ug_S)\otimes_{\text{U}(\overline{\fp}_{I,S})} W_S \twoheadrightarrow M_S$. Let $\pi$ be a finite length smooth representation of $L_I(L)$ over $E$.
\begin{lemma}\label{lem: SlaOS}
We have an isomorphism of locally $\Q_p$-analytic representations of $G$ (cf. Remark \ref{rem: actS} for the action ``$\cdot$"):
\begin{equation}\label{equ: SanOSa}
  \cF_{\overline{P}_I}^G(M_S, \pi)\\ \cong \big\{f\in  (\Ind_{\overline{P}_I(L)}^G W_S^{\vee}\otimes_E\pi)^{S-\an}\ |\ \
  Y\cdot f=0, \  \forall\ Y\in \Ker\phi_S\big\}.
\end{equation}
\end{lemma}
\begin{proof}Consider the projection $\text{U}(\ug_{\Sigma_L})\otimes_{\text{U}(\overline{\fp}_{I,\Sigma_L})} W_S \xrightarrow{\phi} M_S$. Using the isomorphism
\begin{equation*}
\text{U}(\ug_{\Sigma_L})\otimes_{\text{U}(\overline{\fp}_{I,\Sigma_L})} W_S \cong \big(\text{U}(\ug_{S})\otimes_{\text{U}(\overline{\fp}_{I,S})} W_S\big)\otimes_E \big(\text{U}(\ug_{\Sigma_L\setminus S}) \otimes_{\text{U}(\overline{\fp}_{I,\Sigma_L\setminus S})}1\big),
\end{equation*}one easily sees
\begin{equation*}
  \Ker(\phi)=\big(\text{U}(\ug_S)\otimes_{\text{U}(\overline{\fp}_{I,S})} W_S\big)\otimes_E \fd_{I,\Sigma_L\setminus S} +\Ker(\phi_S) \otimes_E \big(\text{U}(\ug_{\Sigma_L\setminus S})\otimes_{\text{U}(\overline{\fp}_{I,\Sigma_L\setminus S})} 1\big) =:\Ker(\phi)_1+\Ker(\phi)_2.
\end{equation*}
By definition, we have
\begin{equation}\label{equ: SanOSb}
  \cF_{\overline{P}_I}^G(M_S, \pi)\cong  \big\{f\in (\Ind_{\overline{P}_I(L)}^G W_S^{\vee}\otimes_E\pi)^{\Q_p-\an} \ |\  Y\cdot f=0, \  \forall\  Y\in \Ker(\phi)\}.
\end{equation}
By Lemma \ref{lem: lgln-Sla} and the fact  $\Ker(\phi)_2/(\Ker(\phi)_1\cap \Ker(\phi)_2)\cong \Ker(\phi_S)$, we easily deduce the right hand side of (\ref{equ: SanOSa}) and (\ref{equ: SanOSb}) coincide. The lemma follows.
\end{proof}
\begin{lemma}\label{lem: SfiniOS}
Let $M_S\in \co^{\overline{\fp}_{I,S}}_{\alg}\subseteq \co^{\overline{\fp}_{I,\Sigma_L}}_{\alg}$, $\ul{\lambda}^S$ be a dominant weight of $\ft_{\Sigma_L\setminus S}$, and $\pi$ be a finite length smooth representation of $L_I(L)$. We have an isomorphism of locally $\Q_p$-analytic representations of $G$ (note $M_S \otimes_E \overline{L}(-\ul{\lambda}^S)\in \co^{\overline{\fp}_{I,\Sigma_L}}_{\alg}$):
\begin{equation*}
  \cF_{\overline{P}_I}^G\big(M_S\otimes_E \overline{L}(-\ul{\lambda}^S), \pi\big) \cong \cF_{\overline{P}_I}^G(M_S, \pi) \otimes_E L(\ul{\lambda}^S).
\end{equation*}
\end{lemma}
\begin{proof}The lemma follows by similar arguments as in the proof of \cite[Prop. 4.9 (b)]{OS}. First recall  we have an isomorphism of $\text{U}(\ug_{\Sigma_L\setminus S})$-modules
\begin{equation*}\big(\text{U}(\ug_{\Sigma_L\setminus S}) \otimes_{\text{U}(\overline{\fp}_{I,\Sigma_L\setminus S})} 1\big)\otimes_E \overline{L}(-\ul{\lambda}^S) \xlongrightarrow{\sim}\text{U}(\ug_{\Sigma_L\setminus S}) \otimes_{\text{U}(\overline{\fp}_{I,\Sigma_L\setminus S})} \overline{L}(-\ul{\lambda}^S),
\end{equation*}
sending $(1\otimes 1)\otimes w$ to $1\otimes w$, where $\big(\text{U}(\ug_{\Sigma_L\setminus S}) \otimes_{\text{U}(\overline{\fp}_{I,\Sigma_L\setminus S})} 1\big)\otimes_E \overline{L}(-\ul{\lambda}^S)$ is equipped with the diagonal action, and where $\text{U}(\overline{\fp}_{I,\Sigma_L\setminus S})$ acts on $\overline{L}(-\ul{\lambda}^S)$ via $\text{U}(\overline{\fp}_{I,\Sigma_L\setminus S}) \hookrightarrow \text{U}(\ug_{\Sigma_L\setminus S})$. Using this isomorphism, we deduce an exact sequence in $\co^{\overline{\fp}_{I,\Sigma_L \setminus S}}_{\alg}$:
\begin{multline*}
  0 \ra \fd_{I,\Sigma_L \setminus S}\otimes_E \overline{L}(-\ul{\lambda}^S) \ra \big( \big(\text{U}(\ug_{\Sigma_L\setminus S}) \otimes_{\text{U}(\overline{\fp}_{I,\Sigma_L\setminus S})} 1\big)\otimes_E \overline{L}(-\ul{\lambda}^S)  \cong \big)\\ \text{U}(\ug_{\Sigma_L\setminus S}) \otimes_{\text{U}(\overline{\fp}_{I,\Sigma_L\setminus S})}\overline{L}(-\ul{\lambda}^S)
  \xrightarrow{\phi^S} \overline{L}(-\ul{\lambda}^S) \ra 0,
\end{multline*}
where $\fd_{I,\Sigma_L \setminus S}\otimes_E \overline{L}(-\ul{\lambda}^S)$ is equipped with the diagonal $\text{U}(\ug_{\Sigma_L\setminus S})$-action.
\\

\noindent Let $W_S$ be a finite dimensional algebraic representation of  $\overline{\fp}_{I,S}$ which generates $M_S$. Denote by $\phi_S$  the projection $\widetilde{M}_S:=\text{U}(\ug_S) \otimes_{\text{U}(\overline{\fp}_{I,S})} W_S \twoheadrightarrow M_S$, and denote by $\phi$ the induced projection
\begin{multline*}
  \text{U}(\ug_{\Sigma_L})\otimes_{\text{U}(\overline{\fp}_{I,\Sigma_L})} (\overline{L}(-\ul{\lambda}^S)\otimes_E W_S)\\ \cong \big(\text{U}(\ug_S) \otimes_{\text{U}(\overline{\fp}_{I, S})} W_S\big)\otimes_E \big(\text{U}(\ug_{\Sigma_L\setminus S}) \otimes_{\text{U}(\overline{\fp}_{I,\Sigma_L\setminus S})} \overline{L}(-\ul{\lambda}^S)\big) \xlongrightarrow{\phi_S \otimes \phi^S} M_S \otimes_E \overline{L}(-\ul{\lambda}^S).
\end{multline*}
We have thus
\begin{equation}\label{equ: kerphiSS}
  \Ker(\phi)=\\ \Ker(\phi^S) \otimes_E \widetilde{M}_S + ( \text{U}(\ug_{\Sigma_L\setminus S}) \otimes_{\text{U}(\overline{\fp}_{I,\Sigma_L\setminus S})} \overline{L}(-\ul{\lambda}^S))\otimes_E \Ker(\phi_S)=:\Ker(\phi)_1+\Ker(\phi)_2.
\end{equation}
By definition, we have
\begin{equation*}
   \cF_{\overline{P}_I}^G\big(M_S\otimes_E \overline{L}(-\ul{\lambda}^S), \pi\big) \\ \cong  \big\{f\in (\Ind_{\overline{P}_I(L)}^G L(\ul{\lambda}^S)\otimes_E W_S^{\vee}\otimes_E\pi)^{\Q_p-\an} \ |\  Y \cdot f=0, \  \forall \ Y\in \Ker(\phi)\}.
\end{equation*}
Consider the following bijection
\begin{equation*}
 \iota: L(\ul{\lambda}^S)\otimes_E \cC^{\Q_p-\an}\big(G, W_S^{\vee} \otimes_E \pi\big) \xlongrightarrow{\sim}  \cC^{\Q_p-\an}\big(G, L(\ul{\lambda}^S)\otimes_E W_S^{\vee} \otimes_E \pi\big)
\end{equation*}
sending $v\otimes f$ to the map $[g\mapsto g(v)\otimes f(g)]$. One can check this map induces an isomorphism of locally $\Q_p$-analytic representations of $G$
\begin{equation*}
  L(\ul{\lambda}^S)\otimes_E (\Ind_{\overline{P}_I(L)}^G W_S^{\vee}\otimes_E \pi)^{\Q_p-\an} \xlongrightarrow{\sim} (\Ind_{\overline{P}_I(L)}^G L(\ul{\lambda}^S) \otimes_E W_S^{\vee} \otimes_E \pi)^{\Q_p-\an}.
\end{equation*}
For $Y=X\otimes v \otimes  w \in \ug_{\Sigma_L} \otimes_E \overline{L}(-\ul{\lambda}^S)\otimes_E W_S\subset \text{U}(\ug_{\Sigma_L}) \otimes_E \overline{L}(-\ul{\lambda}^S) \otimes_E W_S$, and $v'\otimes f\in L(\ul{\lambda}^S)\otimes_E (\Ind_{\overline{P}_I(L)}^G W_S^{\vee}\otimes_E \pi)^{\Q_p-\an}$, one can check
\begin{equation*}
\big(Y \cdot (\iota(v'\otimes f))\big)(g)= (v\otimes w)\big( ((-X)(gv'))\otimes f(g)+(gv')\otimes (X\cdot f)(g)\big).
\end{equation*}
In particular, for $n^S \in \fn_{I,\Sigma_L\setminus S}$, and
\begin{multline*}Y':=(n^S\otimes 1) \otimes v \otimes 1 \otimes w\in (\text{U}(\ug_{\Sigma_L\setminus S}) \otimes_{\text{U}(\overline{\fp}_{I,\Sigma_L\setminus S})} 1) \otimes_E \overline{L}(-\ul{\lambda}^S) \otimes_E \text{U}(\ug_S)\otimes_{\text{U}(\overline{\fp}_{I,S})} W_S\\ \cong \big(\text{U}(\ug_{\Sigma_L\setminus S}) \otimes_{\text{U}(\overline{\fp}_{I,\Sigma_L\setminus S})} \overline{L}(-\ul{\lambda}^S)\big)\otimes_E\big(\text{U}(\ug_S) \otimes_{\text{U}(\overline{\fp}_{I, S})} W_S\big),
\end{multline*}
one can check (note that $Y'$ is sent to $(n^S \otimes v- 1\otimes n^S(v))\otimes 1 \otimes w$ via the above isomorphism, and use that $(n^Su)(u')=-u(n^Su')$ for $u\in \overline{L}(-\ul{\lambda}^S) $, $u'\in L(\ul{\lambda}^S)$):
\begin{multline*}\label{equ: action}
  \big(Y' \cdot (\iota(v'\otimes f))\big)(g)=(v\otimes w) ((-n^S)(gv') \otimes f(g))+(v\otimes w)((gv')\otimes (n^S\cdot f)(g))-(n^S(v) \otimes w)((gv') \otimes f(g))\\
  =(v\otimes w)((gv')\otimes (n^S\cdot f)(g))=v(gv') w((n^S\cdot f)(g)).
\end{multline*}
Let $v'$ and $f$ be as above. We can deduce that the followings are equivalent:
\begin{enumerate}[(1)]\item $\iota(v'\otimes f)$ is annihilated by $\Ker(\phi^S) \otimes_E \widetilde{M}_S $,
\item $v(gv')w((X\cdot f)(g))=0$ for all $X\in \fn_{I,\Sigma_L\setminus S}$, $g\in G$, $v\in \overline{L}(-\ul{\lambda}^S)$ and $w\in W_S$,
\item $X\cdot f=0$ for all $X\in \fn_{I,\Sigma_L\setminus S}$.
\end{enumerate}
Since $Y\cdot f=0$ for all $Y\in \overline{\fp}_{I,\Sigma_L\setminus S}$ (because $W_S^{\vee}\otimes_E \pi$ is locally $S$-analytic for the action of  $\overline{P}_I$, see the proof of Lemma \ref{lem: lgln-Sla}), we deduce (3) is equivalent to that $f$ is locally $S$-analytic, i.e. $f\in (\Ind_{\overline{P}_I(L)}^G W_S^{\vee}\otimes_E \pi)^{S-\an}$.
\\

\noindent
By the equivalence between (1) and (3), and using the fact $\Ker(\phi)_2/(\Ker(\phi)_1\cap \Ker(\phi)_2)\cong \overline{L}(-\ul{\lambda}^S)\otimes_E  \Ker(\phi_S)$ together with (\ref{equ: kerphiSS}) and Lemma \ref{lem: SlaOS}, one sees $\iota(v'\otimes f)$ is  annihilated by $\Ker(\phi)=\Ker(\phi)_1+\Ker(\phi)_2$ if and only if $f$ is locally $S$-analytic and lies in $\cF_{\overline{P}_I}^G(M_S, \pi)$.
This concludes the proof.
\end{proof}

\subsection{Extensions of locally analytic representations, I}\noindent We study the extension groups of certain locally analytic generalized Steinberg representations. In particular, we show that for any $i\in \Delta$, and $\ul{\lambda}\in X_{\Delta}^+$, the $\Ext^1$ of $v_{\overline{P}_i}^{\infty}(\ul{\lambda})$ (recall our convention: $\overline{P}_i=\overline{P}_{\{i\}}$) by $\St_n^{\an}(\ul{\lambda})$ can be parametrized by Breuil's simple $\cL$-invariants (cf. Theorem \ref{cor: lgln-key}). This section follows the line of \cite{Orl} but in the setting of locally analytic representations. A key ingredient is Schraen's spectral sequence \cite[(4.39)]{Sch11}.

\begin{proposition}
We have \begin{equation}\label{equ: lgln-ext0}
      \Ext^i_{G}\big(i_{\overline{P}_I}^G(\ul{\lambda}), \bI_{\overline{P}_J}^G(\ul{\lambda})\big)\cong \begin{cases}
        \hH^i_{\an}\big(L_J(L), E\big) & \text{if } J\subseteq I \\
        0 & \text{otherwise}
      \end{cases},
\end{equation}
and (where $\chi_{\ul{\lambda}}$ is viewed as a character of $Z$ by restriction)
\begin{equation}\label{equ: lgln-ext2}
      \Ext^i_{G,\chi_{\ul{\lambda}}}\big(i_{\overline{P}_I}^G(\ul{\lambda}), \bI_{\overline{P}_J}^G(\ul{\lambda})\big)\cong \begin{cases}
        \hH^i_{\an}\big(L_J(L)/Z, E\big) & \text{if } J\subseteq I \\
        0 & \text{otherwise}
      \end{cases}.
\end{equation}
\end{proposition}
\begin{proof}We only prove (\ref{equ: lgln-ext2}) (for the case with fixed central character), and (\ref{equ: lgln-ext0}) follows by the same argument.
\\

\noindent (a) By \cite[Cor. 4.9]{Sch11}, we have a spectral sequence (note that the separated assumption in \emph{loc. cit. }is satisfied since $i_{\overline{P}_I}^G(\ul{\lambda})$ is locally algebraic):
\begin{equation*}
E_2^{r,s}=\Ext^r_{L_J(L),Z=\chi_{\ul{\lambda}}}\big(\hH_s(\overline{N}_{J},i_{\overline{P}_I}^G(\ul{\lambda})), L(\ul{\lambda})_J\big) \Rightarrow \Ext^{r+s}_{G,\chi_{\ul{\lambda}}}\big(i_{\overline{P}_I}^G(\ul{\lambda}), \bI_{\overline{P}_J}^G(\ul{\lambda})\big).
\end{equation*}
By \cite[(4.41)]{Sch11}, $\hH_s(\overline{N}_J,i_{\overline{P}_I}^G(\ul{\lambda}))=\hH_s(\overline{\fn}_{J,\Sigma_L}, L(\ul{\lambda}))\otimes_E J_{\overline{P}_J}(i_{\overline{P}_I}^G)$ \big(where $J_{\overline{P}_J}(i_{\overline{P}_I}^G)$ denotes the Jacquet module of $i_{\overline{P}_I}^G$ with respect to the parabolic subgroup $\overline{P}_J$\big). By \cite[Thm. 4.10]{Sch11}, one has an isomorphism of algebraic representations of $\fl_{J,\Sigma_L}$ (thus of $\Res^L_{\Q_p} L_J$):
\begin{equation}
\hH_s(\overline{\fn}_{J,\Sigma_L},L(\ul{\lambda}))=\bigoplus_{\substack{w\in (S_{n})^{d_L}, \lg(w)=s,\\ w\cdot \ul{\lambda}\in X_J^{+}}} L(w\cdot \ul{\lambda})_J.
\end{equation}
Thus by \cite[Prop. 4.7 (1)]{Sch11}, for all $s>1$, $r\geq 0$, we have
\begin{equation*}
\Ext^r_{L_J(L),Z=\chi_{\ul{\lambda}}}\big(\hH_s(\overline{N}_J,i_{\overline{P}_I}^G(\ul{\lambda})), L(\ul{\lambda})_J\big)=0.
\end{equation*}
Consequently, for all $r\in \Z_{\geq 0}$, we have
\begin{equation*}
\Ext^r_{G,\chi_{\ul{\lambda}}}\big(i_{\overline{P}_I}^G(\ul{\lambda}), \bI_{\overline{P}_J}^G(\ul{\lambda})\big) \cong \Ext^r_{L_J(L),Z=\chi_{\ul{\lambda}}}\big(J_{\overline{P}_J}(i_{\overline{P}_I}^G) \otimes_E L(\ul{\lambda})_J, L(\ul{\lambda})_J\big).
\end{equation*}
\\
\noindent (b) We reduce to the case $\ul{\lambda}=\ul{0}$: By \cite[(4.23)]{Sch11} (with an easy variation), we can deduce
\begin{equation*}
\Ext^r_{L_J(L),Z=\chi_{\ul{\lambda}}}\big(J_{\overline{P}_J}(i_{\overline{P}_I}^G) \otimes_E L(\ul{\lambda})_J, L(\ul{\lambda})_J\big) \cong \Ext^r_{L_J(L),Z=1}\big(J_{\overline{P}_J}(i_{\overline{P}_I}^G), L(\ul{\lambda})_J\otimes_E L(\ul{\lambda})_J^{\vee}\big).
\end{equation*}
The algebraic representation $L(\ul{\lambda})_J\otimes_E L(\ul{\lambda})_J^{\vee}$ \big(of $L_J(L)$\big) is semi-simple, in which the trivial representation has multiplicity one. Thus by \cite[Prop. 4.7 (1)]{Sch11}, we deduce
\begin{equation*}
\Ext^r_{L_J(L),Z=1}\big(J_{\overline{P}_J}(i_{\overline{P}_I}^G), L(\ul{\lambda})_J\otimes_E L(\ul{\lambda})_J^{\vee}\big) \cong \Ext^r_{L_J(L),Z=1}\big(J_{\overline{P}_J}(i_{\overline{P}_I}^G),1\big).
\end{equation*}

\noindent
(c) As in the proof of \cite[Prop. 15]{Orl}, we define a filtration on $i_{\overline{P}_I}^G$ by $\overline{P}_J(L)$-invariant subspaces:
\begin{equation*}
\cF^k i_{\overline{P}_I}^G:=\Big\{f\in i_{\overline{P}_I}^G, \  \Supp(f)\subseteq \cup_{\substack{w\in \sW_I\backslash \sW/\sW_J\\ \lg(w)\leq k}} \overline{P}_I(L)\backslash\big( \overline{P}_I(L)w\overline{P}_J(L)\big)\Big\},
\end{equation*}for $k\in \Z_{\geq 0}$, where $\sW_S$ denotes the Weyl group of the Levi subgroup $L_S$ for $S\subseteq \Delta$, and  $\lg(w)$ is the length of its Kostant-representative (which is the one of minimal length with its double coset). In the following, we identify the double cosets $\sW_I\backslash \sW /\sW_J$ with its Kostant-representatives.
As in the proof of \cite[Prop. 15]{Orl}, we have
\begin{equation*}
J_{\overline{P}_J}(\gr^k i_{\overline{P}_I}^G) \cong \bigoplus_{\substack{w\in \sW_I\backslash \sW/\sW_J, \\ \lg(w)=k}} \big(\Ind_{L_J(L)\cap (w^{-1} \overline{P}_I(L) w)}^{L_J(L)} \gamma_w\big)^{\infty},
\end{equation*}
where $\gr^k i_{\overline{P}_I}^G:=\cF^k i_{\overline{P}_I}^G/\cF^{k-1} i_{\overline{P}_I}^G$ and $\gamma_w$ is the modulus character of $\overline{P}_J(L)\cap w^{-1} \overline{P}_I(L) w$ acting on $\overline{N}_J(L)/\big(\overline{N}_J(L)\cap w^{-1} \overline{P}_I(L) w\big)$. We have (see the proof of \cite[Prop. 15]{Orl}),
\begin{equation}\label{equ: lgln-vsm}
\Ext^r_{\Rep^{\infty}_{Z=1}(L_J(L))}\big((\Ind_{L_J(L)\cap (w^{-1} \overline{P}_I w)}^{L_J(L)} \gamma_w)^{\infty}, 1\big)=0
\end{equation}
for all $r\in \Z_{\geq 0}$, if $J\nsubseteq I$ or $w\neq 1$.
\\

\noindent (d) We have a spectral sequence (e.g. see \cite[(4.27)]{Sch11}, note that the same arguments of \emph{loc. cit.} also show the existence of a similar spectral sequence in the case without fixing the central character, which can be used to prove (\ref{equ: lgln-ext0})):
\begin{equation*}
\Ext^r_{\cM(L_J(L))^{\infty}_{Z,1}}\big(V^{\vee},\hH^s\big((\overline{\fl_{J}})_{\Sigma_L},E\big)\otimes_E W^{\vee}\big) \Rightarrow \Ext^{r+s}_{L_J(L),Z=1}(W,V).
\end{equation*}    for smooth representations $W$, $V$ of $L_J(L)$ with trivial character on $Z$, where $\overline{\fl_J}$ denotes the Lie algebra of $L_J(L)/Z$. If $V$ and $W$ are moreover admissible, by (\ref{equ: extsm}), we have \begin{equation*}\Ext^r_{\cM(L_J(L))^{\infty}_{Z,1}}\big(V^{\vee},\hH^s\big((\overline{\fl_{J}})_{\Sigma_L},E\big)\otimes_E W^{\vee}\big)\cong \Ext^r_{\Rep^{\infty}_{Z=1}(L_J(L))}\big(W\otimes_E \hH^s\big((\overline{\fl_{J}})_{\Sigma_L},E\big)^{\vee}, V \big).
\end{equation*}
These combined with (\ref{equ: lgln-vsm}) imply
\begin{equation*}
\Ext^r_{L_J(L),Z=1}\big((\Ind_{L_J(L)\cap (w^{-1} \overline{P}_I w)}^{L_J(L)} \gamma_w)^{\infty}, 1\big)=0
\end{equation*}
for all $r\in \Z_{\geq 0}$, if $J\nsubseteq I$ or $w\neq 1$. From which, we deduce by d\'evissage that  if $J\nsubseteq I$, then $$\Ext^r_{L_J(L),Z=1}\big(J_{\overline{P}_J}(i_{\overline{P}_I}^G),1\big)=0,$$ and that  if $J\subseteq I$, then (see Proposition \ref{prop: gpcoho} for the second isomorphism)
\begin{equation*}
\Ext^r_{L_J(L),Z=1}\big(J_{\overline{P}_J}(i_{\overline{P}_I}^G),1\big) \cong \Ext^r_{L_J(L),Z=1}(1,1)\cong \hH^r_{\an}\big(L_J(L)/Z,E\big).
\end{equation*}
The proposition follows.
\end{proof}

\begin{remark}\label{rem: lgln-ext}
(i) Let $Z_J(L)$ be the center of $L_J(L)$, $\fd_J$ the Lie algebra of $D_J(L)$ (the derived subgroup of $L_J(L)$).
Then by \cite[Cor. 3.11, (3.28)]{Sch11}, we have
\begin{eqnarray}\label{equ: lgln-cohan}\hH^i_{\an}(L_J(L),E)&\cong& \bigoplus_{0\leq r\leq i}\Big(\big(\wedge^r \Hom(Z_J(L),E)\big) \otimes_E  \hH^{i-r}\big((\fd_J)_{\Sigma_L}, E\big)\Big),\\
\hH^i_{\an}(L_J(L)/Z,E)&\cong& \bigoplus_{0\leq r\leq i}\Big(\big(\wedge^r \Hom\big(Z_J(L)/Z,E\big)\big) \otimes_E  \hH^{i-r}\big((\fd_J)_{\Sigma_L}, E\big)\Big).
\end{eqnarray}
Since $\fd_J$ is semi-simple, $\hH^r\big((\fd_J)_{\Sigma_L}, E\big)=0$ for $r=1,2$. In particular, we deduce natural isomorphisms $\hH^i_{\an}\big(L_J(L),E\big)\cong \wedge^i \Hom\big(Z_J(L),E\big)$ and $\hH^i_{\an}(L_J(L)/Z,E)\cong \wedge^i\Hom(Z_J(L)/Z,E)$ for $i=1, 2$.
\\

\noindent
(ii) Let $J\subseteq I$ and $j=1,2$, the natural injection $L_J(L)\hookrightarrow L_I(L)$ induces a natural morphism $\hH^j_{\an}(L_I(L),E)\ra \hH^j_{\an}(L_J(L),E)$, and hence (by (i)) induces a morphism $\wedge^j\Hom(Z_I(L),E) \ra \wedge^j\Hom(Z_J(L),E)$. This morphism can also be induced by the natural composition
\begin{equation}\label{equ: dettI}Z_J\hooklongrightarrow L_J \hooklongrightarrow L_I \xlongrightarrow{\dett} Z_I.\end{equation}
In particular, we will view $\Hom(Z_I(L),E)$ as subspace of $\Hom(Z_J(L),E)$ with no further mention. The natural map $\Hom(Z_J(L)/Z_I(L),E) \hookrightarrow \Hom(Z_J(L),E)$ induces  a bijection
\begin{equation*}
\Hom(Z_J(L)/Z_I(L),E) \xlongrightarrow{\sim}  \Hom(Z_J(L),E)/\Hom(Z_I(L),E).
\end{equation*}
\end{remark}
\noindent Let $J\subseteq I$, we describe explicitly the isomorphism in (\ref{equ: lgln-ext0}) for $\Ext^1$: Let
\begin{equation*}\psi\in \hH^1_{\an}\big(L_J(L),E\big)\cong \Hom\big(Z_J(L),E\big),
\end{equation*}
which induces an extension $1_{\psi}$ of the trivial characters of $L_J(L)$:
  \begin{equation*}
    1_{\psi}(a)=\begin{pmatrix}
      1 & \psi(a) \\ 0 & 1
    \end{pmatrix}, \\ \forall ~a\in L_J(L).
  \end{equation*}
The parabolic induction $\big(\Ind_{\overline{P}_J(L)}^G L(\ul{\lambda})_J\otimes_E 1_{\psi}\big)^{\Q_p-\an}$  lies thus in an exact sequence
  \begin{equation*}
    0 \ra \bI_{\overline{P}_J}^G(\ul{\lambda}) \ra \big(\Ind_{\overline{P}_J(L)}^G L(\ul{\lambda})_J\otimes_E 1_{\psi}\big)^{\Q_p-\an} \xrightarrow{\pr} \bI_{\overline{P}_J}^G(\ul{\lambda}) \ra 0.
  \end{equation*}
  Since $J\subseteq I$, we have natural injections
  \begin{equation*}i_{\overline{P}_I}^G(\ul{\lambda}) \hooklongrightarrow \bI_{\overline{P}_I}^G(\ul{\lambda}) \hooklongrightarrow \bI_{\overline{P}_J}^G(\ul{\lambda}).
  \end{equation*}
  Let $\sE_{I}^J(\ul{\lambda},\psi):=\pr^{-1} \big(i_{\overline{P}_I}^G(\ul{\lambda})\big)$, which is thus an extension of $i_{\overline{P}_I}^G(\ul{\lambda})$ by $\bI_{\overline{P}_J}^G(\ul{\lambda})$.
  \begin{lemma}\label{descex}
    The representation $[\sE_{I}^J(\ul{\lambda},\psi)]\in \Ext^1_{G}\big(i_{\overline{P}_I}^G(\ul{\lambda}), \bI_{\overline{P}_J}^G(\ul{\lambda})\big)$ is mapped (up to non-zero scalars) to $\psi$ via the isomorphism in (\ref{equ: lgln-ext0}):
  \begin{equation}\label{equ: lgln-isoex}\Ext^1_{G}\big(i_{\overline{P}_I}^G(\ul{\lambda}), \bI_{\overline{P}_J}^G(\ul{\lambda})\big)\cong \hH^1_{\an}\big(L_J(L),E\big).\end{equation}
  \end{lemma}
\begin{proof}By the discussion in \cite[\S ~4.4]{Sch11} (in particular, see \cite[(4.32)]{Sch11}), we see that (\ref{equ: lgln-isoex}) factors though the following composition
  \begin{multline}\label{equ: lgln-lex}
    \Ext^1_{G}\big(i_{\overline{P}_I}^G(\ul{\lambda}),\bI_{\overline{P}_J}^G(\ul{\lambda})\big) \xlongrightarrow{\sim} \Ext^1_{\overline{P}_J(L)}\big(i_{\overline{P}_I}^G(\ul{\lambda})|_{\overline{P}_J(L)}, L(\ul{\lambda})_J\big)\\ \xlongrightarrow{\sim} \Ext^1_{L_J(L)}\big(J_{\overline{P}_J}(i_{\overline{P}_I}^G) \otimes_E L(\ul{\lambda})_J, L(\ul{\lambda})_J\big) \xlongrightarrow{\sim} \Ext^1_{L_J(L)}\big(L(\ul{\lambda})_J,L(\ul{\lambda})_J\big).
  \end{multline}
  where the first map sends $V$ to the push-forward of $V|_{\overline{P}_J(L)}$ via the natural evaluation map $\bI_{\overline{P}_J}^G(\ul{\lambda}) \twoheadrightarrow L(\ul{\lambda})_J$, $f\mapsto f(1)$, the second map is inverse of the map induced by the natural projection $i_{\overline{P}_I}^G(\ul{\lambda})|_{\overline{P}_J(L)} \twoheadrightarrow J_{\overline{P}_J}(i_{\overline{P}_I}^G) \otimes_E L(\ul{\lambda})_J$ (note the latter one is no other than the $\overline{N}_J(L)$-covariant quotient of the first one), and the last map is induced by the injection $L(\ul{\lambda})_J\hookrightarrow J_{\overline{P}_J}(i_{\overline{P}_I}^G) \otimes_E L(\ul{\lambda})_J$. It is sufficient to show that (\ref{equ: lgln-lex}) sends $\sE_{I}^J(\ul{\lambda},\psi)$ to $1_{\psi}\otimes_E L(\ul{\lambda})_J$. The composition map \big(which is $\overline{P}_J(L)$-invariant\big)
  \begin{equation}\label{equ: lgln-aaIl}i_{\overline{P}_I}^G(\ul{\lambda}) \lra \bI_{\overline{P}_J}^G(\ul{\lambda}) \twoheadlongrightarrow L(\ul{\lambda})_J\end{equation} factors though
  \begin{equation}\label{equ: lgln-jgf}J_{\overline{P}_J}(i_{\overline{P}_I}^G)\otimes_E L(\ul{\lambda})_J\twoheadlongrightarrow L(\ul{\lambda})_J.
   \end{equation}Moreover, it is straightforward to see the composition $L(\ul{\lambda})_J\hookrightarrow J_{\overline{P}_J}(i_{\overline{P}_I}^G)\otimes_E L(\ul{\lambda})_J \xrightarrow{(\ref{equ: lgln-jgf})} L(\ul{\lambda})_J$ is the identity map (up to non-zero scalars). Thus the inverse of the last isomorphism in (\ref{equ: lgln-lex}) can be induced by the projection (\ref{equ: lgln-jgf}), and hence the inverse of the composition of the last two maps in (\ref{equ: lgln-lex}) is induced by (\ref{equ: lgln-aaIl}), i.e. it factors as
    \begin{equation*}
    \Ext^1_{L_J(L)}\big(L(\ul{\lambda})_J,L(\ul{\lambda})_J\big) \lra \Ext^1_{\overline{P}_J(L)}\big(\bI_{\overline{P}_J}^G(\ul{\lambda})|_{\overline{P}_J(L)}, L(\ul{\lambda})_J\big) \lra \Ext^1_{\overline{P}_J(L)}\big(i_{\overline{P}_I}^G(\ul{\lambda})|_{\overline{P}_J(L)}, L(\ul{\lambda})_J\big).
  \end{equation*}
However, by construction, $\sE_{I}^J(\ul{\lambda},\psi)$ is just the image of $1_\psi\otimes_E L(\ul{\lambda})_J$ via the above composition and the inverse of the first isomorphism in (\ref{equ: lgln-lex}). We see that (\ref{equ: lgln-lex}) maps $\psi$ to $1_\psi\otimes_E L(\ul{\lambda})_J$, and the lemma follows.
\end{proof}
\begin{remark}\label{rem: lalg01}
  (i) We have a similar description of the isomorphism in (\ref{equ: lgln-ext2})  for $i=1$. And we have that $[\sE_{I}^J(\ul{\lambda},\psi)]\in \Ext^1_{G,\chi_{\ul{\lambda}}}\big(i_{\overline{P}_I}^G(\ul{\lambda}), \bI_{\overline{P}_J}^G(\ul{\lambda})\big)$ if and only if the character $\psi$ factors through $L_J(L)/Z$.
  \\

\noindent
(ii) By \cite[Prop. 15]{Orl}, we have
\begin{equation}\label{equ: lgln-extSm}\Ext^i_{\Rep^{\infty}(G)}(i_{\overline{P}_I}^G, i_{\overline{P}_J}^G)=
\begin{cases}
  \wedge^i \Hom_{\infty}\big(L_J(L),E\big) & \text{if $J\subseteq I$} \\
  0 & \text{otherwise}
\end{cases},
\end{equation}
where $\Hom_{\infty}(L_J(L),E)$ denotes the $E$-vector space of smooth $E$-valued characters on $L_J(L)$. One can describe similarly as in Lemma \ref{descex} the isomorphism for $\Ext^1$ in (\ref{equ: lgln-extSm}). On the other hand, the injection $i_{\overline{P}_J}^G \hookrightarrow \bI_{\overline{P}_J}^G$ induces an injection (using $\Hom_G(i_{\overline{P}_I}^G, \bI_{\overline{P}_J}^G/i_{\overline{P}_J}^G)=0$)
\begin{equation*}
  \Ext^1_{G}\big(i_{\overline{P}_I}^G, i_{\overline{P}_J}^G\big) \hooklongrightarrow \Ext^1_{G}\big(i_{\overline{P}_I}^G, \bI_{\overline{P}_J}^G\big).
\end{equation*}
Using the explicit description for $\Ext^1$ in Lemma \ref{descex}, we see that the following diagram is commutative (where the horizontal maps are natural injections)
\begin{equation*}
  \begin{CD}
    \Ext^1_{\Rep^{\infty}(G)}\big(i_{\overline{P}_I}^G, i_{\overline{P}_J}^G\big) @>>> \Ext^1_{G}\big(i_{\overline{P}_I}^G, \bI_{\overline{P}_J}^G\big) \\
    @V \sim VV @V \sim VV \\
    \Hom_{\infty}\big(L_J(L),E\big) @>>> \Hom\big(L_J(L),E\big)
  \end{CD}.
\end{equation*}
We have similar results for $i_{\overline{P}_I}^G(\ul{\lambda})$, $i_{\overline{P}_J}^G(\ul{\lambda})$, $\bI_{\overline{P}_J}^G(\ul{\lambda})$. And we have that $\sE_{I}^J(\ul{\lambda},\psi)$ comes by push-forward from a locally algebraic  extension of $i_{\overline{P}_I}^G(\ul{\lambda})$ by $i_{\overline{P}_J}^G(\ul{\lambda})$ if and only if the character $\psi$ is smooth.
\end{remark}

  \begin{corollary}\label{cor: lgln-ext2}
    We have
\begin{equation}\label{equ: lgln-ext4}
      \Ext^r_{G}\big(v_{\overline{P}_I}^\infty(\ul{\lambda}), \bI_{\overline{P}_J}^G(\ul{\lambda})\big)\cong
      \begin{cases}
        \hH_{\an}^{r-|\Delta\setminus I|}\big(L_J(L),E\big) & \text{if } I\cup J=\Delta \\
        0 & \text{otherwise}
      \end{cases},
    \end{equation}
    and
    \begin{equation}\label{equ: lgln-ext3}
      \Ext^r_{G,\chi_{\ul{\lambda}}}\big(v_{\overline{P}_I}^\infty(\ul{\lambda}), \bI_{\overline{P}_J}^G(\ul{\lambda})\big)\cong
      \begin{cases}
        \hH_{\an}^{r-|\Delta\setminus I|}\big(L_J(L)/Z,E\big) & \text{if } I\cup J=\Delta \\
        0 & \text{otherwise}
      \end{cases}.
    \end{equation}
  \end{corollary}
\begin{proof}We only prove (\ref{equ: lgln-ext3}), and the proof for (\ref{equ: lgln-ext4}) is similar.
Using the exact sequence (\ref{equ: lgln-lalgSt}),
the cohomology groups in the corollary can  be identified with the corresponding hypercohomology groups of the complex
    \begin{equation*}
      \Big[C_{n-1-|I|} \ra C_{n-2-|I|} \ra \cdots  \ra C_0\Big],
    \end{equation*}
    with $C_{j}=\oplus _{\substack{I\subseteq K,\\ |K/I|=j}}   i_{\overline{P}_{K}}^G(\ul{\lambda})$. By the same arguments as  in the discussion after \cite[Prop. 5.4]{Sch11} (using \cite[(25)]{Koh2011}), we obtain a spectral sequence
    \begin{equation*}
      E_1^{s,r}=\Ext_{G,\chi_{\ul{\lambda}}}^r\big(\oplus_{\substack{I\subseteq K,\\ |K/I|=s}} i_{\overline{P}_{K}}^G(\ul{\lambda}),\bI_{\overline{P}_J}^G(\ul{\lambda})\big) \Rightarrow \Ext^{r+s}_{G,\chi_{\ul{\lambda}}}\big(v_{\overline{P}_I}^\infty(\ul{\lambda}), \bI_{\overline{P}_J}^G(\ul{\lambda})\big).
    \end{equation*}
    By (\ref{equ: lgln-ext2}), for any $r\in \Z_{\geq 0}$, the $r$-th row of the $E_1$-page of the spectral sequence is given by
  \begin{equation*}\footnotesize
    \begin{matrix}& E_1^{|I\cup J|-|I|,r}& \ra & E_1^{|I\cup J|-|I|+1,r} & \ra & \cdots & \ra & E_1^{ n-1-|I|,r} \\
         &\parallel &\empty&\parallel &\empty &\empty &\empty &\parallel  \\
     &\Ext_{G,\chi_{\ul{\lambda}}}^r\big(i_{\overline{P}_{I\cup J}}^G(\ul{\lambda}),\bI_{\overline{P}_J}^G(\ul{\lambda})\big)& \ra & \oplus_{\substack{I\cup J\subseteq K,\\ |K\setminus (I\cup J)|=1}}\Ext_{G,\chi_{\ul{\lambda}}}^r\big(i_{\overline{P}_K}^G(\ul{\lambda}),\bI_{\overline{P}_J}^G(\ul{\lambda})\big) &\ra &\cdots &\ra & \Ext_{G,\chi_{\ul{\lambda}}}^r\big(i_{\overline{P}_{\Delta}}^G(\ul{\lambda}),\bI_{\overline{P}_J}^G(\ul{\lambda})\big)\\
     &\parallel &\empty&\parallel &\empty &\empty &\empty &\parallel \\
     &\hH^r_{\an}\big(L_J(L)/Z,E\big)&\ra &\oplus_{\substack{I\cup J\subseteq K,\\ |K\setminus (I\cup J)|=1}} \hH^r_{\an}\big(L_J(L)/Z,E\big)& \ra &\cdots & \ra &\hH^r_{\an}\big(L_J(L)/Z,E\big).
    \end{matrix}
  \end{equation*}
If $I\cup J\neq \Delta$,  the above sequence is exact (since it is a constant coefficient system on the standard simplex corresponding to $I\cup J$), and hence $\Ext^{r}_{G,\chi_{\ul{\lambda}}}\big(v_{\overline{P}_I}^\infty(\ul{\lambda}), \bI_{\overline{P}_J}^G(\ul{\lambda})\big)=0$ for all $r\in \Z_{\geq 0}$; if $I\cup J =\Delta$, only the elements on the $|\Delta\setminus I|$-th column of the $E_1$-page can be non-zero, and hence
    \begin{equation*}
      \hH^r_{\an}\big(L_J(L)/Z,E\big)\cong \Ext^r_{G,\chi_{\ul{\lambda}}}\big(i_{\overline{P}_{\Delta}}^G(\ul{\lambda}),\bI_{\overline{P}_J}^G(\ul{\lambda})\big)\cong \Ext^{r+|\Delta\setminus I|}_{G,\chi_{\ul{\lambda}}}\big(v_{\overline{P}_I}^{\infty}(\ul{\lambda}),\bI_{\overline{P}_J}^G(\ul{\lambda})\big).
    \end{equation*}
    The corollary follows.
  \end{proof}
\begin{corollary}\label{prop: lgln-ext2}
Let $i\in \Delta$, we have isomorphisms of $E$-vector spaces
\begin{equation*}
\Hom\big(Z_{\Delta\setminus \{i\}}(L)/Z,E\big)  \xlongrightarrow{\sim} \Ext^1_{G,\chi_{\ul{\lambda}}}\big(v_{\overline{P}_i}^{\infty}(\ul{\lambda}), \St_n^{\an}(\ul{\lambda})\big)
\xlongrightarrow{\sim} \Ext^1_{G}\big(v_{\overline{P}_i}^{\infty}(\ul{\lambda}), \St_n^{\an}(\ul{\lambda})\big).
\end{equation*}
%natural commutative diagram of isomorphisms
%\begin{equation*}
%  \begin{CD}
 %\Hom\big(Z_{\Delta\setminus \{i\}}(L)/Z,E\big)@> \sim >> \Hom\big(Z_{\Delta\setminus \{i\}}(L),E\big)/\Hom\big(Z,E\big) \\
%    @V \sim VV @V \sim VV\\
% \Ext^1_{G,\chi_{\ul{\lambda}}}\big(v_{\overline{P}_i}^{\infty}(\ul{\lambda}), \St_n^{\an}(\ul{\lambda})\big)   @>\sim >> \Ext^1_{G}\big(v_{\overline{P}_i}^{\infty}(\ul{\lambda}), \St_n^{\an}(\ul{\lambda})\big)
%  \end{CD}.
%\end{equation*}
%\begin{equation*}
%\Ext^1_{G}\big(v_{\overline{P}_I}^{\infty}(\ul{\lambda}), \St_n^{\an}(\ul{\lambda})\big) \cong %\Ext^1_{G,\chi_{\ul{\lambda}}}\big(v_{\overline{P}_I}^{\infty}(\ul{\lambda}), \St_n^{\an}(\ul{\lambda})\big) \cong %\hH^1_{\an}(\overline{L_{\Delta\setminus I}},E)\cong \Hom(\overline{Z_{\Delta\setminus I}},E).
%\end{equation*}
In particular, $\dim_E  \Ext^1_{G}\big(v_{\overline{P}_i}^{\infty}(\ul{\lambda}), \St_n^{\an}(\ul{\lambda})\big)=d_L+1$.
\end{corollary}
\begin{proof}
The second isomorphism follows from the fact that $\Hom_G\big(v_{\overline{P}_i}^{\infty}(\ul{\lambda}), \St_n^{\an}(\ul{\lambda})\big)=0$. Consider the complex (cf. (\ref{equ: lgln-lanSt}))
\begin{equation*}
\Big[C_{-(n-1)} \ra C_{-(n-2)} \ra \cdots \ra C_0\Big]
\end{equation*}
with $C_{-j}=\oplus_{ |K|=j} \bI_{\overline{P}_K}^G(\ul{\lambda})$. Similarly as in the proof of Corollary \ref{cor: lgln-ext2}, we deduce a spectral sequence
\begin{equation}\label{equ: lgln-sps}
E_1^{-s,r}=\oplus_{ |K|=s} \Ext^r_{G,\chi_{\ul{\lambda}}}\big(v_{\overline{P}_i}^{\infty}(\ul{\lambda}), \bI_{\overline{P}_K}^G(\ul{\lambda})\big) \Rightarrow \Ext^{r-s}_{G,\chi_{\ul{\lambda}}}\big(v_{\overline{P}_i}^{\infty}(\ul{\lambda}), \St_n^{\an}(\ul{\lambda})\big).
\end{equation}
By (\ref{equ: lgln-ext3}), only the objects in the $(1-n)$-th and $(2-n)$-th columns of the $E_1$-page can be non-zero, where the $(n-1)$-th row is given by
\begin{equation*}
\begin{matrix}
 & \hH^1_{\an}(L_{\Delta}(L)/Z,E) &\lra &\hH^1_{\an}(L_{\Delta \setminus \{i\}}(L)/Z,E) \\
 &\parallel &\empty &\parallel \\
 &E_1^{1-n, n-1} &\xlongrightarrow{d_1^{1-n,n-1}} &E_1^{2-n, n-1}
\end{matrix}
\end{equation*}
and the $n$-th row is given by
\begin{equation*}
\begin{matrix}
  &\hH^2_{\an}(L_{\Delta}(L)/Z,E) &\lra &\hH^2_{\an}(L_{\Delta \setminus \{i\}}(L)/Z,E)\\
  &\parallel &\empty &\parallel \\
 & E_1^{1-n,n} &\xlongrightarrow{d_1^{1-n,n}} & E_1^{2-n,n}
\end{matrix}.
\end{equation*}
By Remark \ref{rem: lgln-ext} (i),  $\hH^i_{\an}(L_{\Delta}(L)/Z,E)\cong \hH^i(\usl_{n,\Sigma_L},E)=0$ for $i=1,2$. The first isomorphism follows.
\end{proof}

\begin{corollary} \label{prop: lgln-ext}Let $i\in \Delta$, we have a natural commutative diagram  of isomorphisms:
\begin{equation}\label{equ: lgln-excd}
  \begin{CD}
   {\Hom(T(L)/Z,E)}/{\Hom(Z_i(L)/Z,E)}@> \sim >> {\Hom(T(L),E)}/{\Hom(Z_i(L),E)} \\
    @V \sim VV @V \sim VV\\
 \Ext^1_{G,\chi_{\ul{\lambda}}}\big(i_{\overline{P}_i}^G(\ul{\lambda}), \St_n^{\an}(\ul{\lambda})\big)@>\sim >>\Ext^1_{G}\big(i_{\overline{P}_i}^G(\ul{\lambda}), \St_n^{\an}(\ul{\lambda})\big)
  \end{CD}.
\end{equation}
In particular, $\dim_E \Ext^1_G(i_{\overline{P}_i}^G\big(\ul{\lambda}), \St_n^{\an}(\ul{\lambda})\big)=d_L+1.$
  \end{corollary}
  \begin{proof}
  The bottom isomorphism follows from the fact that $\Hom_G\big(i_{\overline{P}_i}^{G}(\ul{\lambda}), \St_n^{\an}(\ul{\lambda})\big)=0$. The top isomorphism follows by Remark \ref{rem: lgln-ext} (ii).
\\

\noindent
Using the long exact sequence (\ref{equ: lgln-lanSt}) as in the proof of Corollary \ref{prop: lgln-ext2} (see (\ref{equ: lgln-sps})), we can deduce a spectral sequence
\begin{equation}\label{equ: lgln-sp3}
      E_1^{-s,r}=\oplus_{|K|=s} \Ext^r_{G}\big(i_{\overline{P}_i}^G(\ul{\lambda}), \bI_{\overline{P}_K}^G(\ul{\lambda})\big) \Rightarrow \Ext^{r-s}_{G}\big(i_{\overline{P}_i}^G(\ul{\lambda}), \St_n^{\an}(\ul{\lambda})\big),
\end{equation}
\begin{equation*}
   \text{\Big(resp. } E_1^{-s,r}=\oplus_{|K|=s} \Ext^r_{G,\chi_{\ul{\lambda}}}\big(i_{\overline{P}_i}^G(\ul{\lambda}), \bI_{\overline{P}_K}^G(\ul{\lambda})\big) \Rightarrow \Ext^{r-s}_{G,\chi_{\ul{\lambda}}}\big(i_{\overline{P}_i}^G(\ul{\lambda}), \St_n^{\an}(\ul{\lambda})\big)\text{\Big)}.
\end{equation*}
We prove the vertical isomorphisms of (\ref{equ: lgln-excd}). We only prove the right isomorphism, since the left one follows by the same argument. By (\ref{equ: lgln-ext0}), only the $0$-th and $(-1)$-th columns of the $E_1$-page of (\ref{equ: lgln-sp3}) can have non-zero terms. We see the $1$-th row is given by
  \begin{equation*}
d_1^{-1,1}: \Ext^1_{G}\big(i_{\overline{P}_i}^G(\ul{\lambda}), \bI_{\overline{P}_i}^G(\ul{\lambda})\big) \longrightarrow \Ext^1_{G}\big(i_{\overline{P}_i}^G(\ul{\lambda}), \bI_{\overline{P}_{\emptyset}}^G(\ul{\lambda})\big).
  \end{equation*}
By (\ref{equ: lgln-ext0}) and (\ref{equ: lgln-cohan}), we have
  \begin{equation*}E_1^{0,1}/\Ima(d_1^{-1,1})\cong \Hom(T(L),E)/\Hom(Z_i(L),E).\end{equation*} The $2$-th row is given by
\begin{equation*}
d_1^{-1,2}: \Ext^2_{G}\big(i_{\overline{P}_i}^G(\ul{\lambda}), \bI_{\overline{P}_i}^G(\ul{\lambda})\big) \longrightarrow \Ext^2_{G}\big(i_{\overline{P}_i}^G(\ul{\lambda}), \bI_{\overline{P}_{\emptyset}}^G(\ul{\lambda})\big).
\end{equation*}
By (\ref{equ: lgln-ext0})  and (\ref{equ: lgln-cohan}), we have
\begin{eqnarray*}
  \Ext^2_{G}\big(i_{\overline{P}_i}^G(\ul{\lambda}), \bI_{\overline{P}_i}^G(\ul{\lambda})\big)&\cong& \wedge^2 \Hom(Z_i(L),E),\\
  \Ext^2_{G}\big(i_{\overline{P}_i}^G(\ul{\lambda}), \bI_{\overline{P}_{\emptyset}}^G(\ul{\lambda})\big) &\cong & \wedge^2 \Hom(T(L),E).
\end{eqnarray*}
In particular, $d_1^{-1,2}$ is injective (see also Remark \ref{rem: lgln-ext} (ii)). Together with (\ref{equ: lgln-sp3}),  we obtain the right isomorphism of (\ref{equ: lgln-excd}). Finally, using an explicit description of the vertical isomorphisms as in Remark \ref{rem: lgln-ext2} below, one sees easily that the diagram commutes.
\end{proof}

\begin{remark}\label{rem: lgln-ext2}
We describe explicitly the right isomorphism in (\ref{equ: lgln-excd}) (the description for the left one is similar). From the proof, this isomorphism is induced via the natural morphism
\begin{equation*}
   \Hom(T(L),E) \cong \Ext^1_{G}\big(i_{\overline{P}_{i}}^G(\ul{\lambda}),  \bI_{\overline{B}}^G(\ul{\lambda})\big)\lra \Ext^1_{G}\big(i_{\overline{P}_{i}}^G(\ul{\lambda}), \St_n^{\an}(\ul{\lambda})\big).
 \end{equation*}
Let $\psi\in \Hom(T(L),E)$, and $V:=\sE_{\{i\}}^{\emptyset}(\ul{\lambda}, \psi)$  (see Lemme. \ref{descex}), then the push-forward $W$ of $V$  via the the quotient map $\bI_{\overline{B}}^G(\ul{\lambda})\twoheadrightarrow \St_n^{\an}(\ul{\lambda})$ gives the extension of $i_{\overline{P}_{i}}^G(\ul{\lambda})$ by $\St_n^{\an}(\ul{\lambda})$ associated to $\psi$ via the right isomorphism in (\ref{equ: lgln-excd}). This extension is split if and only if $\psi$ comes from a character on $Z_{i}(L)$ via (\ref{equ: dettI}) (applied with $J=\emptyset$, $I=\{i\}$). Note also that by Remark \ref{rem: lalg01} (ii), $W$ comes by push-forward from a locally algebraic extension of $i_{\overline{P}_{i}}^G(\ul{\lambda})$ by $\St_n^{\infty}(\ul{\lambda})$ if and only if the character $\psi$ is smooth modulo $\Hom(Z_{i}(L),E)$.
\end{remark}

\begin{theorem}\label{cor: lgln-key}
Let $i \in \Delta$, the natural map
    \begin{equation*}\Ext^1_G(v_{\overline{P}_i}^{\infty}(\ul{\lambda}),\St_n^{\an}(\ul{\lambda})) \lra \Ext^1_G(i_{\overline{P}_i}^G(\ul{\lambda}),\St_n^{\an}(\ul{\lambda}))
     \end{equation*}is an isomorphism. In particular, we have an isomorphism
    \begin{equation}\label{equ: lgln-extc}
      \Hom(T(L),E)/\Hom(Z_i(L),E) \xlongrightarrow{\sim} \Ext^1_G(v_{\overline{P}_i}^{\infty}(\ul{\lambda}),\St_n^{\an}(\ul{\lambda})).
    \end{equation}
\end{theorem}
\begin{proof}Denote by $w_{\overline{P}_i}^{\infty}(\ul{\lambda}):=\sum_{J\supsetneq \{i\}} i_{\overline{P}_J}^G(\ul{\lambda}) \hookrightarrow i_{\overline{P}_i}^G(\ul{\lambda})$, and consider the exact sequence
\begin{equation*}
   0 \ra w_{\overline{P}_i}^{\infty}(\ul{\lambda}) \ra i_{\overline{P}_i}^G(\ul{\lambda}) \ra v_{\overline{P}_i}^{\infty}(\ul{\lambda}) \ra 0,
\end{equation*}
which induces a long exact sequence
\begin{multline*}
  0 \ra \Hom_G(v_{\overline{P}_i}^{\infty}(\ul{\lambda}),\St_n^{\an}(\ul{\lambda})) \ra \Hom_G(i_{\overline{P}_i}^G(\ul{\lambda}),\St_n^{\an}(\ul{\lambda})) \ra \Hom_G( w_{\overline{P}_i}^{\infty}(\ul{\lambda}), \St_n^{\an}(\ul{\lambda})) \\
  \ra \Ext^1_G(v_{\overline{P}_i}^{\infty}(\ul{\lambda}),\St_n^{\an}(\ul{\lambda})) \ra \Ext^1_G(i_{\overline{P}_i}^G(\ul{\lambda}),\St_n^{\an}(\ul{\lambda})) \ra \Ext_G^1( w_{\overline{P}_i}^{\infty}(\ul{\lambda}), \St_n^{\an}(\ul{\lambda})).
  \end{multline*}
  Since $\Hom_G( w_{\overline{P}_i}^{\infty}(\ul{\lambda}), \St_n^{\an}(\ul{\lambda}))=0$, we get an injection
  \begin{equation*}
    \Ext^1_G(v_{\overline{P}_i}^{\infty}(\ul{\lambda}),\St_n^{\an}(\ul{\lambda})) \hooklongrightarrow \Ext^1_G(i_{\overline{P}_i}^G(\ul{\lambda}),\St_n^{\an}(\ul{\lambda}))
  \end{equation*}
  which is actually an isomorphism since both of these two $E$-vectors spaces are $(d_L+1)$-dimensional by Proposition \ref{prop: lgln-ext2}, \ref{prop: lgln-ext}. The second part follows thus from Proposition \ref{prop: lgln-ext}.
\end{proof}

\begin{remark}\label{rem: lgln-ext3}
(i) The theorem in $\GL_2(\Q_p)$-case was discovered by Breuil (\cite{Br04}), and was used to define the so-called Breuil's $\cL$-invariants, which are roughly the counterpart of Fontaine-Mazur $\cL$-invariants (for $2$-dimensional special $\Gal_{\Q_p}$-representations) in the locally analytic representations of $\GL_2(\Q_p)$. The correspondence between Breuil's $\cL$-invariants and Fontaine-Mazur $\cL$-invariant was actually the first discovered non-trivial example in $p$-adic Langlands program. In the case of $\GL_2(L)$ and $\GL_3(\Q_p)$, the theorem was obtained by Schraen in \cite{Sch10} \cite{Sch11}.
\\

\noindent
(ii) Let $\Psi\in \Hom(T(L),E)$, $i\in \Delta$, and $V\in \Ext^1_G\big(i_{\overline{P}_{i}}^G(\ul{\lambda}),\St_n^{\an}(\ul{\lambda})\big)$ be the extension associated to $\Psi$ as in Remark \ref{rem: lgln-ext2}. By Theorem \ref{cor: lgln-key}, the pull-back of $V$ via the natural injection $w_{\overline{P}_i}^{\infty}(\ul{\lambda})\hookrightarrow i_{\overline{P}_i}^G(\ul{\lambda})$ is split \big(as an extension of $w_{\overline{P}_i}^{\infty}(\ul{\lambda})$ by $\St_n^{\an}(\ul{\lambda})$\big). Quotient by $w_{\overline{P}_i}^{\infty}(\ul{\lambda})$, one obtains thus an extension, denoted by $\tilde{\Sigma}_i(\ul{\lambda}, \Psi)$, of $v_{\overline{P}_i}^{\infty}(\ul{\lambda})$ by $\St_n^{\an}(\ul{\lambda})$, which is associated to $\Psi$ via (\ref{equ: lgln-extc}). In particular,  this extension only depends on the image of $\Psi$ in $\Hom(T(L),E)/\Hom(Z_{i}(L),E)$, and we have that $\tilde{\Sigma}_i(\ul{\lambda}; \Psi)$ comes by push-forward from a locally algebraic extension of $v_{\overline{P}_i}^{\infty}(\ul{\lambda})$ by $\St_n^{\infty}(\ul{\lambda})$ if and only if $\Psi$ is smooth modulo $\Hom(Z_i(L),E)$.
%
%We can construct $\Sigma(\cL_i)$ for $i\in \Delta$, the locally algebraic subrepresentation of $\Sigma(\cL_i)$ is $v_B^{\infty}$. Let $V$ be a locally analytic representation, $v_B^{\an}\hookrightarrow V$, and suppose
 %   \begin{equation*}
  %    \Hom(v_B^{\an}, V) \xlongrightarrow{\sim} \Hom(v_B^{\infty}, V).
  %  \end{equation*}
  %  Suppose the injection $j: v_B^{\infty}\hookrightarrow V$ induces %$f:\Sigma(\cL_i)\lra V$. Under which assumption,
 %if we have another map $f':\Sigma(\cL_i')\lra V$ with restriction to $v_B^{\infty}$ equal to $j$, we show $\cL_i'=\cL_i$. Firstly, $f'=f$ when restricted to $v_B^{\an}$ by the bijection. Consider the representation $\Sigma(\cL_i')\oplus_{v_B^{\an}} \Sigma(\cL_i)$, by ? we can construct a smooth subrepresentation of $V$ which contradicts to classical local Langlands.
\\

\noindent
 (iii) For each $i\in \Delta$, we have an isomorphism
\begin{eqnarray}\label{equ: lgln-Linv0}
  \iota_i: \Hom(L^{\times},E) &\xlongrightarrow{\sim} & \Hom(T(L),E)/\Hom(Z_{i}(L),E), \\ \psi& \mapsto& [(a_1,\cdots, a_n)\mapsto \psi(a_i/a_{i+1})]\nonumber.
\end{eqnarray}
For $\psi\in \Hom(L^{\times}, E)$, we denote by $\tilde{\Sigma}_i(\ul{\lambda}, \psi):=\tilde{\Sigma}_i(\ul{\lambda}, \iota_i(\psi))$.
%Thus using certain specific basis of $\Hom(L^{\times}, E)$, the extensions of $v_{\overline{P}_{\{i\}}}^{\infty}(\ul{\lambda})$ by $\St_n^{\an}(\ul{\lambda})$ can be parametrized by elements in $\bP^{d_L+1}(E)$, which we call \emph{Breuil's simple $\cL$-invariants.}
%Let $\varsigma_i$ denote the following  projection %\big(which naturally factors though the quotient $\Hom(T(L),E)/\Hom(Z_{\{i\}}(L),E)$\big)
%\begin{equation}\label{equ: lgln-varsigmai}
%  \varsigma_i: \Hom(T(L),E) \lra \Hom(L^{\times}, E), \  (\psi_1,\cdots, \psi_n)\mapsto \psi_i-\psi_{i+1}.
%\end{equation} It's straightforward to see that $\varsigma_i\circ \iota_i=\id$. Let
%\begin{equation}\label{equ: lgln-varsigma}
%\varsigma:=(\varsigma_i):  \Hom(T(L),E) \twoheadlongrightarrow\prod_{i\in \Delta}\Hom(L^{\times},E).
%\end{equation}
%which factors through $\Hom(T(L),E)/\Hom(Z,E)$.
\end{remark}

%$\iota_i^{-1}\big(\overline{(\psi_1,\cdots, \psi_n)}\big)=\psi_i-\psi_{i+1}$ (note $T(L)\cong (L^{\times})^n$). We will often use the following composition
%\begin{equation}
%which factors through $\Hom(T(L),E)/\Hom(Z,E)$, and sends $(\psi_1,\cdots, \psi_n)$ to $(\psi_i-\psi_{i+1})_{i\in \Delta}$.
\noindent Let $\alpha\in E^{\times}$, we denote by $*(\alpha, \ul{\lambda}):=*(\ul{\lambda}) \otimes_E \unr(\alpha) \circ \dett$ for any representation $*(\ul{\lambda})$ of $\GL_n(L)$ as above (for example, $*(\ul{\lambda})=\St_n^{\an}(\ul{\lambda})$, $v_{\overline{P}_i}^{\infty}(\ul{\lambda})$ etc.). It is clear that Theorem \ref{cor: lgln-key} holds with $v_{\overline{P}_i}^{\infty}(\ul{\lambda})$, $\St_n^{\an}(\ul{\lambda})$, $i_{\overline{P}_i}^G(\ul{\lambda})$ replaced by $v_{\overline{P}_i}^{\infty}(\alpha, \ul{\lambda})$, $\St_n^{\an}(\alpha, \ul{\lambda})$, $i_{\overline{P}_i}^G(\alpha, \ul{\lambda})$ respectively. We put $\tilde{\Sigma}_i(\alpha, \ul{\lambda}, ?):=\tilde{\Sigma}_i(\ul{\lambda},?)\otimes_E \unr(\alpha)\circ \dett$ with $?$ equal to $\psi\in \Hom(L^{\times},E)$ or $\Psi\in \Hom(T(L),E)$.

\subsection{Extensions of locally analytic representations, II}
\noindent Let  $\ul{\lambda}\in X^+$. In this section, we study certain subrepresentations of $\St_n^{\an}(\ul{\lambda})$ which have simpler (and more clear) structure than $\St_n^{\an}(\ul{\lambda})$ (e.g. see (\ref{soc})). And we show that the extensions of $v_{\overline{P}_i}^{\infty}(\ul{\lambda})$ by $\St_n^{\an}(\ul{\lambda})$  actually come by push-forward from the extensions of $v_{\overline{P}_i}^{\infty}(\ul{\lambda})$ by such subrepresentations (cf. (\ref{equ: BrLi}) (\ref{equ: Lsigma})).
\subsubsection{Subrepresentations of $\St_n^{\an}(\ul{\lambda})$} \noindent Let $I\subseteq \Delta$, denote by $\St_I^{\infty}$ the standard smooth Steinberg representation of $L_I(L)$, i.e.
\begin{equation*}
  \St_I^{\infty}\cong (\Ind_{\overline{B}(L)\cap L_I(L)}^{L_I(L)} 1)^{\infty}/\sum_{\emptyset\neq J\subseteq I} (\Ind_{\overline{P}_J(L)\cap L_I(L)}^{L_I(L)}1)^{\infty},
\end{equation*}
and put $\St_I^{\infty}(\ul{\lambda}):=\St_I^{\infty}\otimes_E L(\ul{\lambda})_I$. Similarly, let
\begin{equation*}
  \St^{\an}_{I}(\ul{\lambda}):= \big(\Ind_{\overline{B}(L)\cap L_I(L)}^{L_I(L)} \chi_{\ul{\lambda}}\big)^{\Q_p-\an}/\sum_{\emptyset \neq J\subseteq I} (\Ind_{\overline{P}_J(L)\cap L_I(L)}^{L_I(L)} L(\ul{\lambda})_J)^{\Q_p-\an}.
\end{equation*}
By \cite[Thm]{OS}, we see  $\St_I^{\infty}(\ul{\lambda})$ is actually the locally algebraic subrepresentation of $\St^{\an}_I(\ul{\lambda})$.
\\

\noindent
By the tranversity of parabolic inductions, we have \begin{equation}\label{equ: lgln-Istein}
    (\Ind_{\overline{P}_{I}(L)}^G \St_{I}^{\infty})^{\infty}  \cong \Big(\Ind_{\overline{P}_{I}(L)}^G \Big((\Ind_{\overline{B}(L)\cap L_{I}(L)}^{L_{I}(L)} 1)^{\infty}/\sum_{\emptyset \neq J\subseteq I} (\Ind_{\overline{P}_J(L)\cap L_{I}(L)}^{L_{I}(L)} 1)^{\infty}\Big)\Big)^{\infty}
    \cong i_{\overline{B}}^G /\sum_{\emptyset \neq J\subseteq I} i_{\overline{P}_J}^G.
  \end{equation}
Similarly, we have an isomorphism
\begin{equation}\label{equ: Istan}
  (\Ind_{\overline{P}_{I}(L)}^G \St_{I}^{\an}(\ul{\lambda}))^{\Q_p-\an} \cong \bI_{\overline{B}}^G(\ul{\lambda})/\sum_{\emptyset \neq J\subseteq I} \bI_{\overline{P}_J}^G(\ul{\lambda}),
\end{equation}
and hence a natural projection
\begin{equation*}
  (\Ind_{\overline{P}_{I}(L)}^G \St_{I}^{\an}(\ul{\lambda}))^{\Q_p-\an} \twoheadlongrightarrow \St_n^{\an}(\ul{\lambda}).
\end{equation*}
Consider the following composition
\begin{equation}\label{equ: SigmaI}
  (\Ind_{\overline{P}_{I}(L)}^G \St_{I}^{\infty}(\ul{\lambda}))^{\Q_p-\an} \hooklongrightarrow (\Ind_{\overline{P}_{I}(L)}^G \St_{I}^{\an}(\ul{\lambda}))^{\Q_p-\an} \twoheadlongrightarrow \St_n^{\an}(\ul{\lambda}),
\end{equation}
and denote by $\tilde{\Sigma}_{\Delta\setminus I}(\ul{\lambda})$ its image. When $I=\Delta \setminus \{i\}$, we denote by $\tilde{\Sigma}_i(\ul{\lambda}):=\tilde{\Sigma}_{\{i\}}(\ul{\lambda})$.
\begin{lemma}\label{lem: lgln-Sigma_i} Let $i\in \Delta$, then we have a natural exact sequence
\begin{equation*}
  0 \lra v_{\overline{P}_{i}}^{\infty}(\ul{\lambda}) \lra (\Ind_{\overline{P}_{\Delta\setminus \{i\}}(L)}^G \St_{\Delta\setminus \{i\}}^{\infty}(\ul{\lambda}))^{\Q_p-\an} \lra \tilde{\Sigma}_i(\ul{\lambda}) \lra 0.
\end{equation*}
\end{lemma}
\begin{proof}
In terms of Orlik-Strauch (\cite{OS}), we have
  \begin{equation*}
     (\Ind_{\overline{P}_{\Delta\setminus \{i\}}(L)}^G \St_{\Delta\setminus \{i\}}^{\infty}(\ul{\lambda}))^{\Q_p-\an} \cong \cF_{\overline{P}_{\Delta\setminus \{i\}}}^G\big(\text{U}(\ug_{\Sigma_L}) \otimes_{\text{U}(\overline{\fp}_{\Delta\setminus \{i\}, \Sigma_L})}\overline{L}(-\ul{\lambda})_{\Delta \setminus \{i\}}, \St_{\Delta\setminus \{i\}}^{\infty}\big).
  \end{equation*}
By \cite[Thm]{OS}, it is easy to see the  locally algebraic subrepresentation of $(\Ind_{\overline{P}_{\Delta\setminus \{i\}}(L)}^G \St_{\Delta\setminus \{i\}}^{\infty}(\ul{\lambda}))^{\Q_p-\an}$ is given by
  \begin{equation*}
    \cF_{G}^G\big(\overline{L}(-\ul{\lambda}), (\Ind_{\overline{P}_{\Delta\setminus \{i\}}(L)}^G \St_{\Delta\setminus \{i\}}^{\infty})^{\infty}\big) \cong
      (\Ind_{\overline{P}_{\Delta\setminus \{i\}}(L)}^G \St^{\infty}_{\Delta\setminus \{i\}})^{\infty} \otimes_E L(\ul{\lambda}),
  \end{equation*}
which, by (\ref{equ: lgln-Istein}), is an extension of $\St_n^{\infty}(\ul{\lambda})$ by $v_{\overline{P}_i}^{\infty}(\ul{\lambda})$. It is easy to see the composition in (\ref{equ: SigmaI}) factors through
  \begin{equation}\label{equ: SigmaIb}
(\Ind_{\overline{P}_{\Delta\setminus \{i\}}(L)}^G \St_{\Delta\setminus \{i\}}^{\infty}(\ul{\lambda}))^{\Q_p-\an}/v_{\overline{P}_{i}}^{\infty}(\ul{\lambda}) \lra \St_n^{\an}(\ul{\lambda}),
  \end{equation}
and it suffices to show (\ref{equ: SigmaIb}) is injective. We already know its restriction on locally algebraic vectors is injective. Suppose (\ref{equ: SigmaIb}) is not injective, and let $V$ be an irreducible constituent of the kernel of (\ref{equ: SigmaIb}), which is thus not locally algebraic. By \cite[Thm]{OS}, $V$ has the form
\begin{equation}\label{equ: lgln-OSst_i}
  \cF_{\overline{P}_{\Delta\setminus \{i\}}}^G(\overline{L}(-s\cdot \ul{\lambda}), \St_{\Delta\setminus \{i\}}^{\infty})
\end{equation}
with  $s\in S_n^{|\Sigma_L|}$ satisfying  $s\cdot \ul{\lambda}\in X_{\Delta\setminus \{i\}}^+$. By (\ref{equ: Istan})(\ref{equ: SigmaI}), the kernel of (\ref{equ: SigmaIb}) is a subquotient of $\sum_{\emptyset \neq I \subset \Delta} \bI_{\overline{P}_I}^G(\ul{\lambda})$, and hence  $V$ is  an irreducible constituent of \begin{equation*}\bI_{\overline{P}_I}^G(\ul{\lambda})\cong \cF_{\overline{P}_I}^G(\text{U}(\ug_{\Sigma_L})\otimes_{\text{U}(\overline{\fp}_{I, \Sigma_L})} \overline{L}(-\ul{\lambda})_I,1),
\end{equation*}
for some $\emptyset \neq I \subset \Delta$ (more precisely, for $I=\{i\}$). However, using the results in \cite{OS} and \cite[Cor. 3.6]{OSch}, any representation of the form (\ref{equ: lgln-OSst_i}) can not appear as irreducible constituent in $\bI_{\overline{P}_I}^G(\ul{\lambda})$ for any  $\emptyset\neq I\subseteq \Delta$, a contradiction. The lemma follows.
\end{proof}
\begin{lemma}
  The injections $\tilde{\Sigma}_i(\ul{\lambda})\hookrightarrow \St_n^{\an}(\ul{\lambda})$ for all $i\in \Delta$ induce an injection
  \begin{equation*}
    \bigoplus_{\St_n^{\infty}(\ul{\lambda})}^{i\in \Delta}\tilde{\Sigma}_i(\ul{\lambda}) \hooklongrightarrow \St_n^{\an}(\ul{\lambda}).
  \end{equation*}
\end{lemma}
\begin{proof}
By \cite[Cor. 3.6]{OSch}, any non locally algebraic irreducible constituent of $\tilde{\Sigma}_i(\ul{\lambda})$ (thus is of  the form  as in  (\ref{equ: lgln-OSst_i})) can not appear as irreducible constituent of $\tilde{\Sigma}_j(\ul{\lambda})$ for $j\neq i$. The lemma follows.
\end{proof}
\begin{proposition}\label{soc}
$\soc_G(\tilde{\Sigma}_i(\ul{\lambda}))\cong \St_n^{\infty}(\ul{\lambda})$.
\end{proposition}
\begin{proof}
First, we have
  \begin{equation} \label{equ: lgln-socInd_i}
    \soc_G \big(\Ind_{\overline{P}_{\Delta\setminus \{i\}}(L)}^G \St^{\infty}_{\Delta\setminus \{i\}}(\ul{\lambda})\big)^{\Q_p-\an} \cong   \soc_G \cF_{\overline{P}_{\Delta\setminus \{i\}}}^G\big(L(-\ul{\lambda}), \St_{\Delta\setminus \{i\}}^{\infty}\big) \cong v_{\overline{P}_i}^{\infty}(\ul{\lambda}).
\end{equation}where the first isomorphism follows from \cite[Cor. 3.3]{Br13I}, and the second follows from \cite[Prop. 17]{Orl}. Suppose there exists an irreducible constituent  $W$ of $\tilde{\Sigma}_i(\ul{\lambda})/\St_n^{\infty}(\ul{\lambda})$ such that $W\hookrightarrow \soc_G(\tilde{\Sigma}_i(\ul{\lambda}))$. The pull-back $V$ of $(\Ind_{\overline{P}_{\Delta\setminus \{i\}}(L)}^G \St_{\Delta\setminus \{i\}}^{\infty}(\ul{\lambda}))^{\Q_p-\an}$ via this injection gives an extension of $W$ by $v_{\overline{P}_i}^{\infty}(\ul{\lambda})$, which is moreover \emph{non-split} by (\ref{equ: lgln-socInd_i}). As in the proof of Lemma \ref{lem: lgln-Sigma_i}, we see that there exists $s\in S_n^{|\Sigma_L|}$, $s\cdot \ul{\lambda}\in X_{\Delta\setminus \{i\}}^+$, such that $W \cong \cF_{\overline{P}_{\Delta\setminus \{i\}}}^G(\overline{L}(-s\cdot \ul{\lambda}), \St_{\Delta\setminus \{i\}}^{\infty})$. Since $V\hookrightarrow (\Ind_{\overline{P}_{\Delta\setminus \{i\}}(L)}^G \St_{\Delta\setminus \{i\}}^{\infty}(\ul{\lambda}))^{\Q_p-\an}$, we know moreover that $V$ is very strongly admissible in the sense of  \cite[Def. 0.12]{Em2}. Indeed, by \cite[Prop. 2.1.2]{Em2}, it is sufficient to show $\St_{\Delta\setminus \{i\}}^{\infty}(\ul{\lambda})$ is very strongly admissible. We then reduce to show that $\St_{\Delta \setminus \{i\}}^{\infty}$ is very strongly admissible (using the fact that an admissible Banach representation tensoring with a finite dimensional algebraic representation of $L_{\Delta\setminus \{i\}}(L)$ is still an admissible Banach representation of $L_{\Delta\setminus \{i\}}(L)$). However, this follows from the natural injection
\begin{equation*}
\St_{\Delta\setminus \{i\}}^{\infty} \hooklongrightarrow (\Ind_{\overline{B}(L)\cap L_{\Delta\setminus \{i\}}(L)}^{L_{\Delta\setminus \{i\}}(L)} 1)^{\cC^0}/\sum_{\emptyset\neq J\subseteq (\Delta \setminus \{i\})} (\Ind_{\overline{P}_J(L)\cap L_{\Delta \setminus \{i\}}(L)}^{L_{\Delta\setminus \{i\}}(L)}1)^{\cC^0}.
\end{equation*}
where we use $(\Ind -)^{\cC_0}$ to denote \emph{continuous} parabolic inductions \big(noting that the smooth subrepresentation of $(\Ind_{\overline{P}_J(L)\cap L_{\Delta \setminus \{i\}}(L)}^{L_{\Delta\setminus \{i\}}(L)}1)^{\cC^0}$ is $(\Ind_{\overline{P}_J(L)\cap L_{\Delta \setminus \{i\}}(L)}^{L_{\Delta\setminus \{i\}}(L)}1)^{\infty}$\big). The proposition follows then from the following lemma.
\end{proof}
\begin{lemma}
Let $\mu\in X_{\Delta\setminus \{i\}}^+\setminus X^+$, and $V$ be an extension of $\cF_{\overline{P}_{\Delta\setminus \{i\}}}^G\big(\overline{L}(-\mu), \St_{\Delta\setminus \{i\}}^{\infty}\big)$ by  $v_{\overline{P}_i}^{\infty}(\ul{\lambda})$. Suppose $V$ is very strongly admissible, then $V$ is split.
\end{lemma}
\begin{proof}
Suppose we have a non-split exact sequence
\begin{equation}\label{equ: hynonsp}
  0 \lra v_{\overline{P}_i}^{\infty}(\ul{\lambda})  \lra V \lra \cF_{\overline{P}_{\Delta\setminus \{i\}}}^G(\overline{L}(-\mu), \St_{\Delta\setminus \{i\}}^{\infty}) \lra 0.
\end{equation}
The rough idea of the proof is to construct a section of (\ref{equ: hynonsp}) (hence a contradiction) using the adjunction property of the Jacquet-Emerton functor.
\\

\noindent (a) By \cite[Prop. 3.6]{Br13II}, $\cF_{\overline{P}_{\Delta\setminus \{i\}}}^G(\overline{L}(-\mu), \St_{\Delta\setminus \{i\}}^{\infty})$ admits a natural $(\ug_{\Sigma_L}, P_{\Delta\setminus \{i\}}(L))$-submodule
\begin{equation*}\cF_{\overline{P}_{\Delta\setminus \{i\}}}^G(\overline{L}(-\mu), \St_{\Delta\setminus \{i\}}^{\infty})^{\Q_p-\alg}\cong \cC^{\infty}_c(N_{\Delta\setminus \{i\}}(L), \St_{\Delta\setminus \{i\}}^{\infty}) \otimes_E L(\mu),\end{equation*}
which is equipped with the finest locally convex topology, and is an $E$-vector space of compact type. Here we use the natural isomorphism of $\text{U}(\ug_{\Sigma_L})$-modules: $\Hom_E(\overline{L}(-\mu),E)^{\fn_{\Delta\setminus \{i\}}^{\infty}}\cong L(\mu)$ where $(\cdot)^{\fn_{\Delta\setminus \{i\}}^{\infty}}$ denotes the subspace generated by the vectors annihilated by a certain power of $\fn_{\Delta\setminus \{i\}, \Sigma_L}$. Note also that the $P_{\Delta\setminus \{i\}}(L)$-action on $\cF_{\overline{P}_{\Delta\setminus \{i\}}}^G(\overline{L}(-\mu), \St_{\Delta\setminus \{i\}}^{\infty})^{\Q_p-\alg}$ is locally $\Q_p$-analytic (see the discussion above \cite[Lem. 3.2]{Br13II} for details). Let $V_0$ be the pull-back of $V$ via the $P_{\Delta\setminus \{i\}}(L)$-equivariant continuous injection
\begin{equation*}
 \cF_{\overline{P}_{\Delta\setminus \{i\}}}^G(\overline{L}(-\mu), \St_{\Delta\setminus \{i\}}^{\infty})^{\Q_p-\alg} \hooklongrightarrow \cF_{\overline{P}_{\Delta\setminus \{i\}}}^G(\overline{L}(-\mu), \St_{\Delta\setminus \{i\}}^{\infty}).
\end{equation*}
We have thus an exact sequence of locally $\Q_p$-analytic representations of $P_{\Delta\setminus \{i\}}(L)$
\begin{equation}\label{equ: extalg}
  0 \lra v_{\overline{P}_i}^{\infty}(\ul{\lambda})  \lra V_0 \lra \cF_{\overline{P}_{\Delta\setminus \{i\}}}^G(\overline{L}(-\mu), \St_{\Delta\setminus \{i\}}^{\infty})^{\Q_p-\alg} \lra 0.
\end{equation}
Since $\cF_{\overline{P}_{\Delta\setminus \{i\}}}^G(\overline{L}(-\mu), \St_{\Delta\setminus \{i\}}^{\infty})^{\Q_p-\alg}$ is equipped with the finest locally convex topology, we see the extension (\ref{equ: extalg}) admits a continuous section (may not be $P_{\Delta\setminus \{i\}}(L)$-equivariant).
\\

\noindent
(b) We recall a little on locally analytic cohomology theory. Let $U$ be a locally $\Q_p$-analytic representation of $P_{\Delta\setminus \{i\}}(L)$ over an $E$-vector space of compact type. For $n\in \Z_{\geq 0}$ and a compact open subgroup $X$ of $P_{\Delta\setminus \{i\}}(L)$, denote by $\cC^{\Q_p-\an}(X^n, U)$ the set of locally $\Q_p$-analytic $U$-valued maps on $X^n$. We have natural edge maps $\partial_n: \cC^{\Q_p-\an}(X^n, U) \ra \cC^{\an}(X^{n+1},U)$ (which are defined in the same way as the usual group cohomology, and we refer to \cite[\S~3.1]{Sch10} for details), and we let $\hH^n_{\an}(X, U):= \Ker \partial_n/\Ima \partial_{n-1}$.

\noindent Fix a compact open subgroup $N_0$ of $N_{\Delta \setminus \{i\}}(L)$.
For $z\in P_{\Delta\setminus \{i\}}(L)$, we have natural maps
  \begin{equation*}
    \cC^{\Q_p-\an}(N_0^n, U) \lra \cC^{\Q_p-\an}((zN_0z^{-1})^n, U),\ f\mapsto [(g_0, \cdots, g_{n-1})\mapsto zf(z^{-1} g_0 z, \cdots, z^{-1} g_{n-1} z)].
  \end{equation*}
It is straightforward to check these maps are compatible with $\{\partial_n\}$, and hence induce maps $r_z:\hH^n(N_0, U) \ra \hH^n(zN_0z^{-1},U)$. Similar as in the setting of usual group cohomology, if $z\in N_0$, then $r_z=1$. Denote by  $L_{\Delta\setminus \{i\}}^+:=\{z\in L_{\Delta \setminus \{i\}}(L)\ |\ zN_0z^{-1}\subset N_0\}$. For $z\in L_{\Delta\setminus \{i\}}^+$, we have a natural transfer map
\begin{equation}\label{equ: tran}
 \tr:  \hH^n_{\an}(zN_0 z^{-1}, U) \lra \hH^n_{\an}(N_0, U).
\end{equation}
Actually, let $R$ be a set of right representatives of $zN_0z^{-1}$ in $N_0$, and consider
\begin{equation*}
  \tr_R: \cC^{\Q_p-\an}((zN_0z^{-1})^n, U) \lra \cC^{\Q_p-\an}((N_0)^n,U)
\end{equation*}
with $\tr_R(f)(a_0)(g_0, \cdots, g_{n-1})=\sum_{h\in R} h_0^{-1} f(h_0 g_0 h_1^{-1}, \cdots, h_{n-1}g_{n-1} h_n^{-1})$,  where $h_0=h$ and  $h_i$ is the unique element in $R$ satisfying $h_i\in (zN_0z^{-1}) h_{i-1} g_{i-1}$. Then the transfer map (\ref{equ: tran})  is induced by $\tr_R$ and is independent of the choice of $R$. For $z\in L_{\Delta\setminus \{i\}}^+$, $n\in \Z_{\geq 0}$, we have thus a Hecke operator $\pi_z$ on $\hH^n_{\an}(N_0, U)$ given by
\begin{equation*}
  \pi_z: \hH^n_{\an}(N_0, U) \xlongrightarrow{\frac{1}{|N_0/zN_0z^{-1}|} \cdot r_z} \hH^n_{\an}(zN_0 z^{-1}, U) \xlongrightarrow{\tr} \hH^n_{\an}(N_0,U).
\end{equation*}
\\

\noindent
(c) Since (\ref{equ: extalg}) admits a continuous section (of topological $E$-vector spaces), by \cite[Prop. 3.3]{Sch10}, (\ref{equ: extalg}) induces  an exact sequence
\begin{equation}\label{equ: lgln-N_0coho}
  0 \ra \hH^0_{\an}\big(N_0, v_{\overline{P}_i}^{\infty}(\ul{\lambda})\big) \ra V_0^{N_0}
   \ra \hH^0_{\an}\big(N_0,  \cC^{\infty}_c(N_{\Delta\setminus \{i\}}(L), \St_{\Delta\setminus \{i\}}^{\infty}) \otimes_E L(\mu)\big) \ra \hH^1_{\an}\big(N_0, v_{\overline{P}_i}^{\infty}(\ul{\lambda})\big),
\end{equation}
which is moreover $L_{\Delta\setminus \{i\}}^+$-equivariant (for the Hecke actinon of $L_{\Delta\setminus \{i\}}^+$ defined in (b)). By \cite[(3.5.4)]{Em11} and the fact $L(\mu)^{N_0} \cong L(\mu)_{\Delta\setminus \{i\}}$, we have a natural $L_{\Delta\setminus \{i\}}^+$-equivariant injection
\begin{equation}\label{injst00}\St_{\Delta\setminus \{i\}}^{\infty}\otimes_E \delta_{\Delta\setminus \{i\}}\otimes_E L(\mu)_{\Delta\setminus \{i\}} \hooklongrightarrow \hH^0_{\an}\big(N_0,  \cC^{\infty}_c(N_{\Delta\setminus \{i\}}(L), \St_{\Delta\setminus \{i\}}^{\infty}) \otimes_E L(\mu)\big),
\end{equation}
where $\delta_{\Delta \setminus \{i\}}$ denotes the modulus character of $P_{\Delta \setminus \{i\}}(L)$ (which is hence smooth).
\\

\noindent (d) We calculate $\hH^r_{\an}\big(N_0, v_{\overline{P}_i}^{\infty}(\ul{\lambda})\big)$.  By \cite[Thm. 4.10]{Koh2011}  (which applies to our case by \cite[Thm. 6.5]{Koh2011}, note also that by \cite[Rem. 2.17, Prop. 2.18]{Koh2011}, $\hH^r_{\an}(N_0,  v_{\overline{P}_i}^{\infty}(\ul{\lambda}))$ defined as in (b) coincides with that  in \cite[Def. 2.5]{Koh2011}), we deduce $\fl_{\Delta\setminus \{i\},\Sigma_L}$-equivariant isomorphisms
\begin{equation}\label{equ: N0an}
  \hH^r_{\an}\big(N_0, v_{\overline{P}_i}^{\infty}(\ul{\lambda})\big) \cong \hH^r\big(\fn_{\Delta\setminus \{i\},\Sigma_L},  v_{\overline{P}_i}^{\infty}(\ul{\lambda})\big)^{N_0} \cong
  \big(v_{\overline{P}_i}^{\infty} \otimes_E \hH^r\big(\fn_{\Delta \setminus \{i\}, \Sigma_L}, L(\ul{\lambda})\big)\big)^{N_0}.
   \end{equation}
By \cite[Thm. 4.10]{Sch11}, we have
\begin{equation}\label{equ: N0lie}
\hH^r\big(\fn_{\Delta \setminus \{i\}, \Sigma_L}, L(\ul{\lambda})\big)^{N_0}  \cong \hH^r\big(\fn_{\Delta \setminus \{i\}, \Sigma_L}, L(\ul{\lambda})\big)\cong \big(\oplus_{\substack{\lg(s)=r\\ s\cdot \ul{\lambda}\in X_{\Delta\setminus \{i\}}^+}}L(\mu_s)_{\Delta\setminus \{i\}}\big),
\end{equation}
where $\mu_s \in X_{\Delta\setminus \{i\}}^+$ is the weight satisfying  $L(\mu_s)_{\Delta \setminus \{i\}}\cong (L(s \cdot \ul{\lambda}')_{\Delta \setminus \{i\}})^{\vee}$ with $\ul{\lambda}'=(-\lambda_{n,\sigma}, \cdots, -\lambda_{1,\sigma})_{\sigma\in \Sigma_L}\in X^+$. We obtain thus an $\fl_{\Delta\setminus \{i\},\Sigma_L}$-equivariant isomorphism
\begin{equation}
\label{equ: h1iso}
  \hH^r_{\an}\big(N_0, v_{\overline{P}_i}^{\infty}(\ul{\lambda})\big) \cong \hH^0_{\an}\big(N_0, v_{\overline{P}_i}^{\infty}\big) \otimes_E \big(\oplus_{\substack{\lg(s)=r\\ s\cdot \ul{\lambda}\in X_{\Delta\setminus \{i\}}^+}}L(\mu_s)_{\Delta\setminus \{i\}}\big).
\end{equation}We show that (\ref{equ: h1iso}) is moreover $L_{\Delta\setminus \{i\}}^+$-equivariant, where $L_{\Delta\setminus \{i\}}^+$ acts on the object on the right hand side via the diagonal action,  the $L_{\Delta\setminus \{i\}}^+$-action on $\hH^0_{\an}\big(N_0, v_{\overline{P}_i}^{\infty})$ is given as in (b), and where the $L_{\Delta\setminus \{i\}}^+$-action on $L(\mu_s)_{\Delta\setminus \{i\}}$ is induced by the natural algebraic action of $L_{\Delta\setminus \{i\}}(L)$ via the injection $L_{\Delta\setminus \{i\}}^+\hookrightarrow L_{\Delta\setminus \{i\}}(L)$. Let $z\in L_{\Delta\setminus \{i\}}^+$, we have a commutative diagram
\begin{equation}\label{equ: diaga}
\begin{CD}
 \hH^0_{\an}\big(N_0, v_{\overline{P}_i}^{\infty}\big) \otimes_E \hH^r_{\an}(N_0, L(\ul{\lambda})) @>>> \hH^r_{\an}\big(N_0, v_{\overline{P}_i}^{\infty}(\ul{\lambda})\big) \\
 @V r_z\otimes r_z VV @V r_z VV \\
  \hH^0_{\an}\big(zN_0z^{-1}, v_{\overline{P}_i}^{\infty}\big) \otimes_E \hH^r_{\an}(zN_0z^{-1}, L(\ul{\lambda})) @>>> \hH^r_{\an}\big(zN_0z^{-1}, v_{\overline{P}_i}^{\infty}(\ul{\lambda})\big)
  \end{CD}
\end{equation}
where the horizontal maps are induced by $v\otimes f\mapsto [g\mapsto v\otimes f(g)]$. By \cite[Thm. 4.10]{Koh2011}, we have a natural commutative diagram
\begin{equation}\label{equ: congs}
\begin{CD}
  \hH^r_{\an}(N_0, L(\ul{\lambda})) @> \sim >> \hH^r(\fn_{\Delta\setminus \{i\}, \Sigma_L}, L(\ul{\lambda}))^{N_0} \\
  @V \iota VV @VVV \\
    \hH^r_{\an}(zN_0z^{-1}, L(\ul{\lambda})) @> \sim >> \hH^r(\fn_{\Delta\setminus \{i\}, \Sigma_L}, L(\ul{\lambda}))^{zN_0z^{-1}}\\
\end{CD}
\end{equation}
where $\iota$ is the  natural restriction map. By (\ref{equ: N0lie}), the right vertical map of (\ref{equ: congs})  is an isomorphism, hence so is $\iota$. Moreover, using the top horizontal isomorphism in (\ref{equ: congs}) and (\ref{equ: N0lie}), the isomorphism (\ref{equ: h1iso}) coincides with the top horizontal map in  (\ref{equ: diaga}). Similarly,  the bottom horizontal map in  (\ref{equ: diaga}) is also an isomorphism. The composition
\begin{equation}\label{equ: rzLie}
\big(\oplus_{\substack{\lg(s)=r\\ s\cdot \ul{\lambda}\in X_{\Delta\setminus \{i\}}^+}}L(\mu_s)_{\Delta\setminus \{i\}}\big) \cong\hH^r_{\an}(N_0, L(\ul{\lambda})) \xlongrightarrow{r_z} \hH^r_{\an}(zN_0z^{-1}, L(\ul{\lambda})) \cong \big(\oplus_{\substack{\lg(s)=r\\ s\cdot \ul{\lambda}\in X_{\Delta\setminus \{i\}}^+}}L(\mu_s)_{\Delta\setminus \{i\}}\big)
\end{equation}
 coincides with the natural $z$-action on $\big(\oplus_{\substack{\lg(s)=1\\ s\cdot \ul{\lambda}\in X_{\Delta\setminus \{i\}}^+}}L(\mu_s)_{\Delta\setminus \{i\}}\big)$ (noting (\ref{equ: N0lie}) also holds with $N_0$ replaced by $zN_0z^{-1}$).
On the other hand, for $v\in \hH^0_{\an}\big(zN_0z^{-1}, v_{\overline{P}_i}^{\infty}\big)$ and $f\in \hH^r_{\an}(N_0, L(\ul{\lambda}))$, one can check $\tr(v\otimes \iota(f))=\tr(v)\otimes f$ (see (\ref{equ: tran})). For
\begin{equation*}v_1\otimes v_2\in \hH^0_{\an}\big(N_0, v_{\overline{P}_i}^{\infty}\big) \otimes_E \big(\oplus_{\substack{\lg(s)=r\\ s\cdot \ul{\lambda}\in X_{\Delta\setminus \{i\}}^+}}L(\mu_s)_{\Delta\setminus \{i\}}\big) \cong
 \hH^r_{\an}\big(N_0, v_{\overline{P}_i}^{\infty}(\ul{\lambda})\big),
\end{equation*}
using the isomorphisms in (\ref{equ: diaga}) (\ref{equ: N0lie}), by (\ref{equ: diaga}) and  the above discussion, we have
\begin{equation*}
  \pi_z(v_1\otimes v_2)=\frac{1}{|N_0/zN_0z^{-1}|} \tr(r_z(v_1)\otimes z(v_2))=\pi_z(v_2) \otimes z(v_2).
\end{equation*}
Hence (\ref{equ: h1iso}) is $L_{\Delta\setminus \{i\}}^+$-equivariant.
\\

\noindent
(e) We show \begin{equation}\label{equ: lgln-dJE}
   \Hom_{L_{\Delta\setminus \{i\}}^+}\big(\St_{\Delta\setminus \{i\}}^{\infty} \otimes_E \delta_{\Delta\setminus \{i\}}\otimes_E L(\mu)_{\Delta\setminus \{i\}},\hH^1_{\an}\big(N_0, v_{\overline{P}_i}^{\infty}(\ul{\lambda})\big)\big)= 0.
\end{equation}
 Indeed, if the vector space in (\ref{equ: lgln-dJE}) is non-zero, we deduce from (\ref{equ: h1iso}) that there exists $s$ such that $\mu=\mu_s$ (e.g. by considering the action of $\fl_{\Delta\setminus \{i\},\Sigma_L}$) and that
\begin{equation*}
  \Hom_{L_{\Delta\setminus \{i\}}^+}\big(\St_{\Delta\setminus \{i\}}^{\infty}\otimes_E \delta_{\Delta\setminus \{i\}}, \hH^0_{\an}\big(N_0, v_{\overline{P}_i}^{\infty}\big)\big)\neq 0.
\end{equation*}
By the adjunction property (e.g. see \cite[(0.2)]{Em2}), we deduce then
\begin{equation*}
  \Hom_{G}\big((\Ind_{\overline{P}_{\Delta\setminus \{i\}}(L)}^G \St_{\Delta\setminus \{i\}}^{\infty})^{\infty}, v_{\overline{P}_i}^{\infty}\big) \neq 0,
\end{equation*}
which implies $\Hom_G\big(i_{\overline{B}}^G, v_{\overline{P}_i}^{\infty}\big)\neq 0$. However by \cite[Prop. 15]{Orl}, $\Ext^i_{G/Z}\big(i_{\overline{B}}^G, i_{\overline{P}}^G\big)=0$ for $i\in \Z_{\geq 0}$ and $\overline{P}\supsetneq \overline{B}$, from which one easily deduces $\Hom_{G}\big(i_{\overline{B}}^G, v_{\overline{P}_i}^{\infty}\big)=0$, a contradiction.
\\

\noindent (f) We finish the proof. Let $U$ be the pull-back of $V_0^{N_0}$ via the injection (\ref{injst00}) (which is well-defined by (e)). We have an $L_{\Delta \setminus \{i\}}^+$-equivariant exact sequence
\begin{equation*}
  0 \ra \hH^0_{\an}(N_0, v_{\overline{P}_i}^{\infty}) \otimes_E L(\ul{\lambda})_{\Delta\setminus \{i\}} \ra U \ra \St_{\Delta\setminus \{i\}}^{\infty}\otimes_E \delta_{\Delta\setminus \{i\}}\otimes_E L(\mu)_{\Delta\setminus \{i\}} \ra 0.
\end{equation*}
Since $\mu \neq \lambda$, applying $(-\otimes_E L(\mu)_{\Delta\setminus \{i\}}^{\vee})^{\fl_{\Delta\setminus \{i\},\Sigma_L}}$ and using $\hH^r(\fl_{\Delta\setminus \{i\},\Sigma_L}, L(\ul{\lambda})_{\Delta\setminus \{i\}} \otimes_E L(\mu)_{\Delta\setminus \{i\}}^{\vee})=0$ for $r=0, 1$, and $\hH^0(\fl_{\Delta\setminus \{i\},\Sigma_L}, L(\mu)_{\Delta\setminus \{i\}} \otimes_E L(\mu)_{\Delta\setminus \{i\}}^{\vee})=E$,  we obtain $L_{\Delta \setminus \{i\}}^+$-equivariant isomorphisms
\begin{equation*}
  (U \otimes_E L(\mu)_{\Delta\setminus \{i\}}^{\vee})^{\fl_{\Delta\setminus \{i\},\Sigma_L}} \xlongrightarrow{\sim} (\St_{\Delta\setminus \{i\}}^{\infty}\otimes_E \delta_{\Delta\setminus \{i\}}\otimes_E L(\mu)_{\Delta\setminus \{i\}}\otimes_E L(\mu)_{\Delta\setminus \{i\}}^{\vee})^{\fl_{\Delta\setminus \{i\},\Sigma_L}} \cong \St_{\Delta\setminus \{i\}}^{\infty}\otimes_E \delta_{\Delta\setminus \{i\}}.
\end{equation*}
where $L_{\Delta \setminus \{i\}}^+$ acts via the diagonal action on the tensor products, and acts via $L_{\Delta\setminus \{i\}}^+ \hookrightarrow L_{\Delta \setminus \{i\}}(L)$ on the algebraic representations of $L_{\Delta \setminus \{i\}}(L)$. We see that the natural $L_{\Delta\setminus \{i\}}^+$-equivariant map (induced by $L(\mu)_{\Delta\setminus \{i\}} \otimes_E L(\mu)_{\Delta\setminus \{i\}}^{\vee} \ra E$)
\begin{equation*}
   (U \otimes_E L(\mu)_{\Delta\setminus \{i\}}^{\vee})^{\fl_{\Delta\setminus \{i\},\Sigma_L}}  \otimes_E L(\mu)_{\Delta\setminus \{i\}} \lra U
\end{equation*}
is injective, because its composition with $U \ra \St_{\Delta\setminus \{i\}}^{\infty}\otimes_E \delta_{\Delta\setminus \{i\}}\otimes_E L(\mu)_{\Delta\setminus \{i\}}$ is bijective. Thus we have $L_{\Delta \setminus \{i\}}^+$-equivariant injections
\begin{equation*}
  \St_{\Delta\setminus \{i\}}^{\infty}\otimes_E \delta_{\Delta\setminus \{i\}} \otimes_E L(\mu)_{\Delta\setminus \{i\}} \hooklongrightarrow U \hooklongrightarrow V_0^{N_0} \hooklongrightarrow V^{N_0}.
\end{equation*}
However, since $V$ is very strongly admissible, by \cite[Thm. 4.3]{Br13II}, it is not difficult to deduce from the above composition a section of (\ref{equ: hynonsp}), a contradiction.
\end{proof}
\begin{remark}
It might be true that $\St_n^{\infty}(\ul{\lambda})\xrightarrow{\sim}\soc_G \St_n^{\an}(\ul{\lambda})$, but the author does not how to prove this.
\end{remark}
\noindent Let  $i\in \Delta$, $\sigma\in \Sigma_L$, $\ul{\lambda}_{\sigma}:=(\lambda_{1,\sigma}, \cdots, \lambda_{n,\sigma})$, and $\ul{\lambda}^{\sigma}:=(\lambda_{1,\sigma'}, \cdots, \lambda_{n,\sigma'})_{\sigma'\in \Sigma_L\setminus \{\sigma\}}$. Consider the locally $\sigma$-analytic parabolic induction
\begin{equation*}
   \big(\Ind_{\overline{P}_{I}(L)}^G \St_{\Delta \setminus \{i\}}^{\infty}(\ul{\lambda}_{\sigma}) \big)^{\sigma-\an}
\end{equation*}
where $\St_{\Delta \setminus \{i\}}^{\infty}(\ul{\lambda}_{\sigma})\cong \St_{\Delta\setminus \{i\}}^{\infty}\otimes_E L(\ul{\lambda}_{\sigma})_{\Delta \setminus \{i\}}$ is a locally $\sigma$-analytic representation of $L_{\Delta\setminus \{i\}}(L)$. By \cite[Thm]{OS} and Lemma \ref{lem: SfiniOS}, we have
\begin{multline*}
\big(\Ind_{\overline{P}_{\Delta\setminus \{i\}}(L)}^G \St_{\Delta \setminus \{i\}}^{\infty}\big)^{\infty}\otimes_E L(\ul{\lambda})  \hooklongrightarrow    \big(\Ind_{\overline{P}_{\Delta \setminus \{i\}}(L)}^G \St_{\Delta \setminus \{i\}}^{\infty}(\ul{\lambda}_{\sigma}) \big)^{\sigma-\an} \otimes_E L(\ul{\lambda}^{\sigma}) \\ \hooklongrightarrow   \big(\Ind_{\overline{P}_{\Delta \setminus \{i\}}(L)}^G \St_{\Delta \setminus \{i\}}^{\infty}(\ul{\lambda}) \big)^{\Q_p-\an}
\lra \St_n^{\an}(\ul{\lambda}).
\end{multline*}
Denote by $\tilde{\Sigma}_{i,\sigma}(\ul{\lambda}):=\big(\big(\Ind_{\overline{P}_{\Delta \setminus \{i\}}(L)}^G \St_{\Delta \setminus \{i\}}^{\infty}(\ul{\lambda}_{\sigma}) \big)^{\sigma-\an}\otimes_E L(\ul{\lambda}^{\sigma})\big)/v_{\overline{P}_i}^{\infty}(\ul{\lambda})\hookrightarrow \tilde{\Sigma}_i(\ul{\lambda}) \hookrightarrow \St_n^{\an}(\ul{\lambda})$. By Proposition \ref{soc}, we have $\soc_{\GL_n(L)}\tilde{\Sigma}_{i,\sigma}(\ul{\lambda})\cong \St_n^{\infty}(\ul{\lambda})$. Let
\begin{equation*}C_{i,\sigma}:=\cF_{\overline{P}_{\Delta \setminus \{i\}}}^G\big(\overline{L}(-\ul{\lambda}^{\sigma})\otimes_E \overline{L}(-s_{i,\sigma}\cdot \ul{\lambda}_{\sigma}), \St_{\Delta \setminus \{i\}}^{\infty}\big),
\end{equation*}
which is an irreducible subrepresentation of $\tilde{\Sigma}_{i,\sigma}(\ul{\lambda})/\St_n^{\infty}(\ul{\lambda})$. By \cite[Thm. 4.6]{OSch}, $C_{i,\sigma}$ has multiplicity one as irreducible constituent in $\St_n^{\an}(\ul{\lambda})$. Denote by $\Sigma_{i,\sigma}(\ul{\lambda})$ the extension of $C_{i,\sigma}$ by $\St_n^{\infty}(\ul{\lambda})$ appearing as a sub of $\tilde{\Sigma}_{i,\sigma}(\ul{\lambda})$, and put
\begin{equation*}\Sigma_i(\ul{\lambda}):=\bigoplus^{\sigma\in \Sigma_L}_{\St_n^{\infty}(\ul{\lambda})} \Sigma_{i,\sigma}(\ul{\lambda})\hookrightarrow \tilde{\Sigma}_i(\ul{\lambda}),
\
\Sigma(\ul{\lambda}):=\bigoplus^{i\in \Delta}_{\St_n^{\infty}(\ul{\lambda})} \Sigma_i(\ul{\lambda}).
\end{equation*}
By Proposition \ref{soc}, we have $\soc_{\GL_n(L)} \Sigma_i(\ul{\lambda})\cong \St_n^{\infty}(\ul{\lambda})$. By \cite[Cor. 3.6]{OSch}, $\{C_{i,\sigma}\}_{i\in \Delta, \sigma\in \Sigma_L}$ are all distinct, hence we also have $\soc_{\GL_n(L)} \Sigma(\ul{\lambda})\cong \St_n^{\infty}(\ul{\lambda})$.

\subsubsection{Extensions of locally analytic representations}
\noindent
The following lemma is an easy consequence of the results in \cite[\S 4.4]{Sch11}.
\begin{lemma}\label{lem: vanish}Let $\mu\in X^+$, $\pi$ be a smooth admissible representation of $G$ over $E$, $W:=L(\mu)\otimes_E \pi$. Let $s\in S_{n}^{|d_L|}\cong \sW_{\Sigma_L}$, $I\subseteq \Delta$ such that $ s\cdot \mu\in X_I^+$, and $\pi_I$ be a finite length smooth representation of $L_I(L)$ over $E$. Suppose $\lg s>1$, then we have

(1) $\Ext^1_G\big(W, (\Ind_{\overline{P}_I(L)}^G L(s\cdot \mu)_I\otimes_E \pi_I)^{\Q_p-\an}\big)=0$;

(2) $\Ext^1_G\big(W, \cF_{\overline{P}_I}^G(\overline{L}(-s\cdot \mu), \pi_I)\big)=0$.
\end{lemma}
\begin{proof}
By \cite[Thm.]{OS}, $\cF_{\overline{P}_I}^G(\overline{L}(-s\cdot \mu),\pi_I)$ is a subrepresentation of $(\Ind_{\overline{P}_I(L)}^G L(s\cdot \mu)_I\otimes_E \pi_I)^{\Q_p-\an}$, and the quotient $(\Ind_{\overline{P}_I(L)}^G L(s\cdot \mu)_I\otimes_E \pi_I)^{\Q_p-\an}/\cF_{\overline{P}_I}^G(\overline{L}(-s\cdot \mu),\pi_I)$ does not have any non-zero locally algebraic vectors (hence $\Hom_G(W,V)=0$). Thus we can easily deduce (2) from (1).
\\

\noindent
We prove (1). By \cite[(4.43), Thm. 4.10]{Sch11}, we have
\begin{equation*}
    \hH_r(\overline{N}_I(L), W)\cong J_{\overline{P}_I}(\pi) \otimes_E \big(\oplus_{\substack{\lg(w)=r,\\ \text{$w\cdot \mu$ is $I$-dominant}}} L(w\cdot \mu)_I\big).
 \end{equation*}
Since $\lg s>1$, by \cite[Prop. 4.7]{Sch11}, for $i=0, 1$, and $r\in \Z_{\geq 0}$, we have
\begin{equation*}
\Ext^r_{L_I(L)}\big( \hH_i(\overline{N}_I(L), W), L(s\cdot \mu)_I\otimes_E \pi_I\big)=0.
\end{equation*}
  Part (1) then follows from the spectral sequence \cite[(4.39)]{Sch11} (note the separateness assumption is satisfied since $W$ is locally algebraic).
\end{proof}
\begin{proposition}\label{prop: iso-simp}
  For any $i, j\in \Delta$, $\sigma\in \Sigma_L$, the following natural morphisms are isomorphisms
  \begin{eqnarray}
  \label{equ: isoSst} \Ext^1_G\big(v_{\overline{P}_i}^{\infty}(\ul{\lambda}), \Sigma(\ul{\lambda})\big) &\xlongrightarrow{\sim}& \Ext^1_G\big(v_{\overline{P}_i}^{\infty}(\ul{\lambda}), \St_n^{\an}(\ul{\lambda})\big),\\
    \Ext^1_G\big(v_{\overline{P}_i}^{\infty}(\ul{\lambda}), \Sigma_j(\ul{\lambda})\big) &\xlongrightarrow{\sim}& \Ext^1_G\big(v_{\overline{P}_i}^{\infty}(\ul{\lambda}), \tilde{\Sigma}_j(\ul{\lambda})\big), \nonumber \\
    \Ext^1_G\big(v_{\overline{P}_i}^{\infty}(\ul{\lambda}), \Sigma_{j,\sigma}(\ul{\lambda})\big) &\xlongrightarrow{\sim}& \Ext^1_G\big(v_{\overline{P}_i}^{\infty}(\ul{\lambda}), \tilde{\Sigma}_{j,\sigma}(\ul{\lambda})\big). \nonumber
  \end{eqnarray}
\end{proposition}
\begin{proof}
  By \cite[Thm. 4.6]{OSch}, any irreducible constituent of $\St_n^{\an}(\ul{\lambda})/\Sigma(\ul{\lambda})$ has the form $\cF_{\overline{P}_J}^G(\overline{L}(-s\cdot \ul{\lambda}), \pi)$ with $\lg(s)>1$. The first isomorphism then follows from Lemma \ref{lem: vanish} by an easy d\'evissage argument. The other isomorphisms follow by the same arguments.
\end{proof}
\begin{lemma}\label{lem: stij}
  Let $i, j \in \Delta$, then for any $\sigma\in \Sigma_L$, we have

  (1) $\dim_E\Ext^1_G\big(v_{\overline{P}_i}^{\infty}(\ul{\lambda}), (\Ind_{\overline{P}_{\Delta\setminus \{j\}}}^{G} L(s_{j, \sigma}\cdot \ul{\lambda})_{\Delta\setminus \{j\}}\otimes_E \St_{\Delta\setminus \{j\}}^{\infty})^{\Q_p-\an}\big)=\begin{cases}
1 & i = j \\
    0 & i\neq j
  \end{cases}$;

  (2) $\dim_E\Ext^1_G\big(v_{\overline{P}_i}^{\infty}(\ul{\lambda}), C_{j,\sigma}\big)=\begin{cases}
   1 & i =j \\
0 & i \neq j
  \end{cases}$.
\end{lemma}
\begin{proof}By \cite[Thm]{OS}, $C_{j,\sigma}\hookrightarrow (\Ind_{\overline{P}_{\Delta\setminus \{j\}}}^G L(s_{j, \sigma}\cdot \ul{\lambda})_{\Delta\setminus \{j\}}\otimes_E \St_{\Delta\setminus \{j\}}^{\infty})^{\Q_p-\an}$, and any irreducible constituent of the quotient $(\Ind_{\overline{P}_{\Delta\setminus \{j\}}}^G L(s_{j, \sigma}\cdot \ul{\lambda})_{\Delta\setminus \{j\}}\otimes_E \St_{\Delta\setminus \{j\}}^{\infty})^{\Q_p-\an}/ C_{j,\sigma}$
has the form $\cF_{\overline{P}}^G(\overline{L}(-s\cdot \ul{\lambda}), \pi)$ with $\lg s>1$. By Lemma \ref{lem: vanish} and an easy d\'evissage argument,  the part (2) follows from (1).
\\

\noindent
  We prove (1). We have
  \begin{multline*}
    \Ext^1_G\big(v_{\overline{P}_i}^{\infty}(\ul{\lambda}), (\Ind_{\overline{P}_{\Delta\setminus \{j\}}}^G L(s_{j, \sigma}\cdot \ul{\lambda})_{\Delta\setminus \{i\}}\otimes_E \St_{\Delta\setminus \{j\}}^{\infty})^{\Q_p-\an}\big)\\ \cong \Hom_{L_{\Delta\setminus \{j\}}(L)}\big(L(s_{j, \sigma}\cdot \ul{\lambda})_{\Delta\setminus \{j\}}\otimes_E J_{\overline{P}_{\Delta\setminus \{j\}}}(v_{\overline{P}_i}^{\infty}), L(s_{j, \sigma}\cdot \ul{\lambda})_{\Delta\setminus \{j\}} \otimes_E \St_{\Delta\setminus \{j\}}^{\infty}\big)
    \\ \cong \Hom_{L_{\Delta\setminus \{j\}}(L)}\big(J_{\overline{P}_{\Delta\setminus \{j\}}}(v_{\overline{P}_i}^{\infty}), \St_{\Delta\setminus \{j\}}^{\infty}\big) \cong \Hom_{G}\big(v_{\overline{P}_i}^{\infty}, (\Ind_{\overline{P}_{\Delta\setminus \{j\}}(L)}^G \St_{\Delta\setminus \{j\}}^{\infty})^{\infty}\big),
  \end{multline*}
  where the first isomorphism follows easily from the spectral sequence \cite[(4.37)]{Sch11} together with \cite[Thm. 4.10, Prop. 4.7]{Sch11}, the second isomorphism follows from \cite[Prop. 4.7]{Sch11} and where the third one follows from the adjunction formula for the classical Jacquet functor $J_{\overline{P}_{\Delta\setminus \{j\}}}(\cdot)$ (e.g. see \cite[(0.2)]{Em2}). However,  $(\Ind_{\overline{P}_{\Delta\setminus \{j\}}(L)}^G \St_{\Delta\setminus \{j\}}^{\infty})^{\infty}$ is a \emph{non-split} extension of $\St_n^{\infty}$ by $v_{\overline{P}_j}^{\infty}$, thus
\begin{equation*}
  \dim_E \Hom_{G}\big(v_{\overline{P}_i}^{\infty}, (\Ind_{\overline{P}_{\Delta\setminus \{j\}}(L)}^G \St_{\Delta\setminus \{j\}}^{\infty})^{\infty}\big)=\begin{cases}
    1 & i=j \\
    0 & i \neq j
  \end{cases}.
\end{equation*}
(1) follows.
\end{proof}
\begin{lemma}\label{lem: ext1}Let $i, j \in \Delta$.

(1) If $i\neq j$, then the following natural morphism is an isomorphism:
\begin{equation*}
  \Ext^1_G\big(v_{\overline{P}_i}^{\infty}(\ul{\lambda}), \St_n^{\infty}(\ul{\lambda})\big)\xlongrightarrow{\sim} \Ext^1_G\big(v_{\overline{P}_i}^{\infty}(\ul{\lambda}),\Sigma_{j}(\ul{\lambda})\big) .
\end{equation*}

(2) The following natural morphism is an isomorphism:
\begin{equation*}
  \Ext^1_G\big(v_{\overline{P}_i}^{\infty}(\ul{\lambda}), \Sigma_i(\ul{\lambda})\big)\xlongrightarrow{\sim} \Ext^1_G\big(v_{\overline{P}_i}^{\infty}(\ul{\lambda}),\Sigma(\ul{\lambda})\big).
\end{equation*}

(3) For $\sigma\in \Sigma_L$, $\dim_E \Ext^1_G\big(v_{\overline{P}_i}^{\infty}(\ul{\lambda}),  \Sigma_{i,\sigma}(\ul{\lambda}) \big)= 2$.
\end{lemma}
\begin{proof}We have a  natural exact sequence
\begin{equation*}
    0 \lra\Ext^1_{G}\big(v_{\overline{P}_i}^{\infty}(\ul{\lambda}), \St_n^{\infty}(\ul{\lambda})\big) \lra \Ext^1_{G}\big(v_{\overline{P}_i}^{\infty}(\ul{\lambda}),  \Sigma_j(\ul{\lambda}) \big) \lra \Ext^1_{G}\big(v_{\overline{P}_i}^{\infty}(\ul{\lambda}), \oplus_{\sigma\in \Sigma_L}C_{j,\sigma}\big).
\end{equation*}
Thus  (1) follows from Lemma \ref{lem: stij} (2). The isomorphism in part (2) also follows from Lemma \ref{lem: stij} (2) by a similar argument.  Now let $\sigma\in \Sigma_L$, and consider the following natural exact sequence
\begin{multline*}
  0 \ra\Ext^1_{G,Z=\chi_{\ul{\lambda}}}\big(v_{\overline{P}_i}^{\infty}(\ul{\lambda}), \St_n^{\infty}(\ul{\lambda})\big) \ra \Ext^1_{G,Z=\chi_{\ul{\lambda}}}\big(v_{\overline{P}_i}^{\infty}(\ul{\lambda}),  \Sigma_{i,\sigma}(\ul{\lambda}) \big)\\ \ra \Ext^1_{G,Z=\chi_{\ul{\lambda}}}\big(v_{\overline{P}_i}^{\infty}(\ul{\lambda}), C_{i,\sigma}\big) \ra \Ext^2_{G,Z=\chi_{\ul{\lambda}}}\big(v_{\overline{P}_i}^{\infty}(\ul{\lambda}), \St_n^{\infty}(\ul{\lambda})\big).
\end{multline*}
For $V= \St_n^{\infty}(\ul{\lambda})$,  $\Sigma_{i,\sigma}(\ul{\lambda})$, $C_{i,\sigma}$, we deduce from $\Hom_G\big(v_{\overline{P}_i}^{\infty}(\ul{\lambda}), V\big)=0$ that $\Ext^1_{G,Z=\chi_{\ul{\lambda}}}\big(v_{\overline{P}_i}^{\infty}(\ul{\lambda}), V\big)\cong \Ext^1_{G}\big(v_{\overline{P}_i}^{\infty}(\ul{\lambda}), V\big)$. By \cite[Thm. 1]{Orl} and \cite[Prop. 4.7]{Sch11}, we deduce
$\Ext^2_{G,Z=\chi_{\ul{\lambda}}}\big(v_{\overline{P}_i}^{\infty}(\ul{\lambda}), \St_n^{\infty}(\ul{\lambda})\big)=0$ and $\dim_E \Ext^1_{G,Z=\chi_{\ul{\lambda}}}\big(v_{\overline{P}_i}^{\infty}(\ul{\lambda}), \St_n^{\infty}(\ul{\lambda})\big)=1$. Together with Lemma \ref{lem: stij} (2), the part (3) follows.
\end{proof}
\noindent
For $i\in \Delta$, by (\ref{equ: lgln-extc}) (\ref{equ: isoSst}) and Lemma \ref{lem: ext1} (2), one has a natural isomorphism
\begin{equation}\label{equ: BrLi}
 \Hom(T(L),E)/\Hom(Z_i(L),E)\xlongrightarrow{\sim}  \Ext^1_G\big(v_{\overline{P}_i}^{\infty}(\ul{\lambda}), \Sigma_{i}(\ul{\lambda})\big).
\end{equation}
\begin{proposition}\label{prop: sigmaan}
  Let $i\in \Delta$ and $\sigma\in \Sigma_L$, the isomorphism (\ref{equ: BrLi}) induces an isomorphism
  \begin{equation}\label{equ: Lsigma}
     \Hom_{\sigma}(T(L)/Z_i(L), E) \xlongrightarrow{\sim} \Ext^1_G\big(v_{\overline{P}_i}^{\infty}(\ul{\lambda}), \Sigma_{i,\sigma}(\ul{\lambda})\big).
  \end{equation}
\end{proposition}
\begin{proof}
  Let $\psi\in \Hom_{\sigma}(T(L)/Z_{i}(L), E)\cong \Hom_{\sigma}(T(L),E)/\Hom_{\sigma}(Z_i(L),E)$ (cf. Remark \ref{rem: lgln-ext} (ii) and note that the same statements hold with $\Hom$ replaced by $\Hom_{\sigma}$) and let  $\Psi \in \Hom_{\sigma}(T(L),E)$ be a lifting of $\psi$ with $\Psi|_{Z_i(L)}=1$. The induced extension $1_{\Psi}$ of trivial characters of $T(L)$:
\begin{equation*}
  1_{\Psi}(a)=\begin{pmatrix}
    1 & \Psi(a) \\ 0 & 1
  \end{pmatrix}
\end{equation*}
is thus locally $\sigma$-analytic. We have a commutative diagram
\begin{equation*}\footnotesize
  \begin{CD}
 0 @>>> \bI_{\overline{B}}^G(\ul{\lambda}_{\sigma})^{\sigma-\an}\otimes_E L(\ul{\lambda}^{\sigma}) @>>> \big(\Ind_{\overline{B}(L)}^G \chi_{\ul{\lambda}_{\sigma}}\otimes_E 1_{\Psi}\big)^{\sigma-\an}\otimes_E L(\ul{\lambda}^{\sigma}) @> \pr_1 >> \bI_{\overline{B}}^G(\ul{\lambda}_{\sigma})^{\sigma-\an} \otimes_E L(\ul{\lambda}^{\sigma}) @>>> 0 \\
 @. @VVV @VVV @VVV @. \\
    0 @>>> \bI_{\overline{B}}^G(\ul{\lambda}) @>>> \big(\Ind_{\overline{B}(L)}^G\chi_{\ul{\lambda}}\otimes_E 1_{\Psi}\big)^{\Q_p-\an} @> \pr_2 >> \bI_{\overline{B}}^G(\ul{\lambda})@>>> 0
  \end{CD},
\end{equation*}
where all the vertical maps are natural injections, e.g. the middle map sends $f \otimes v$ to $[g\mapsto f(g) \otimes gv]$. This diagram induces a commutative diagram
\begin{equation}\label{equ: extsigma}
  \begin{CD}
 0 @>>> \tilde{\Sigma}_{\sigma}(\ul{\lambda})@>>> \pr_1^{-1}\big(i_{\overline{P}_i}^G(\ul{\lambda})\big) @> \pr_1 >> i_{\overline{P}_i}^G(\ul{\lambda}) @>>> 0 \\
 @. @VVV @VVV @V \id VV @. \\
    0 @>>> \St_n^{\an}(\ul{\lambda}) @>>> \pr_2^{-1}\big(i_{\overline{P}_i}^G(\ul{\lambda})\big) @> \pr_2 >> i_{\overline{P}_i}^G(\ul{\lambda}) @>>> 0
  \end{CD}
\end{equation}
where $\tilde{\Sigma}_{\sigma}(\ul{\lambda})$ denotes the image of the composition
\begin{equation*}
  \bI_{\overline{B}}^G(\ul{\lambda}_{\sigma})^{\sigma-\an}\otimes_E L(\ul{\lambda}^{\sigma}) \hooklongrightarrow \bI_{\overline{B}}^G(\ul{\lambda}) \twoheadlongrightarrow \St_n^{\an}(\ul{\lambda}).
\end{equation*}
By the proof of Theorem \ref{cor: lgln-key} (and let $w_{\overline{P}_i}^{\infty}(\ul{\lambda})$ be the representation as in \emph{loc. cit.}), the pull-back of $\pr_2^{-1}\big(i_{\overline{P}_i}^G(\ul{\lambda})\big)$ via the natural injection $\iota: w_{\overline{P}_i}^{\infty}(\ul{\lambda}) \hookrightarrow i_{\overline{P}_i}^G(\ul{\lambda})$ is split, from which we see the pull-back of $\pr_1^{-1}\big(i_{\overline{P}_i}^G(\ul{\lambda})\big)$ via $\iota$ is also split (noting $\Ext^1(w_{\overline{P}_i}^{\infty}(\ul{\lambda}), \tilde{\Sigma}_{\sigma}(\ul{\lambda}))\hookrightarrow \Ext^1(w_{\overline{P}_i}^{\infty}(\ul{\lambda}), \St_n^{\an}(\ul{\lambda}))$ since \begin{equation*}\Hom_G(w_{\overline{P}_i}^{\infty}(\ul{\lambda}),\St_n^{\an}(\ul{\lambda})/\tilde{\Sigma}_{\sigma}(\ul{\lambda}))=0).\end{equation*}
 Denote by $\tilde{\Sigma}_{\sigma}(\ul{\lambda},\psi)$ the induced extension of $v_{\overline{P}_i}^{\infty}(\ul{\lambda})$ by $\tilde{\Sigma}_{\sigma}(\ul{\lambda})$. The diagram (\ref{equ: extsigma}) induces then a commutative diagram
\begin{equation*}
    \begin{CD}
 0 @>>> \tilde{\Sigma}_{\sigma}(\ul{\lambda})@>>> \tilde{\Sigma}_{\sigma}(\ul{\lambda},\psi) @> \pr_1 >> v_{\overline{P}_i}^{\infty}(\ul{\lambda}) @>>> 0 \\
 @. @VVV @VVV @V \id VV @. \\
    0 @>>> \St_n^{\an}(\ul{\lambda}) @>>> \tilde{\Sigma}(\ul{\lambda};\psi) @> \pr_2 >> v_{\overline{P}_i}^{\infty}(\ul{\lambda}) @>>> 0
  \end{CD}
\end{equation*}
with all vertical maps injective. Hence $\tilde{\Sigma}(\ul{\lambda},\psi)$ is isomorphic to the push-forward of $\tilde{\Sigma}_{\sigma}(\ul{\lambda},\psi)$. Consequently, the following  morphism induced by (\ref{equ: lgln-extc}):
\begin{equation*}
  \Hom_{\sigma}(T(L),E)/\Hom_{\sigma}(Z_i(L),E) \hooklongrightarrow \Ext^1_G\big(v_{\overline{P}_i}^{\infty}(\ul{\lambda}), \St_n^{\an}(\ul{\lambda})\big),
\end{equation*}
factors through
\begin{equation*}
  \Hom_{\sigma}(T(L),E)/\Hom_{\sigma}(Z_{i}(L),E) \hooklongrightarrow \Ext^1_G\big(v_{\overline{P}_i}^{\infty}(\ul{\lambda}), \tilde{\Sigma}_{\sigma}(\ul{\lambda})\big).
\end{equation*}
By the same argument as in the proof of Proposition \ref{prop: iso-simp}, one can show that the natural morphism
\begin{equation*}
 \Ext^1_G\big(v_{\overline{P}_i}^{\infty}(\ul{\lambda}), \Sigma_{i,\sigma}(\ul{\lambda})\big) \lra \Ext^1_G\big(v_{\overline{P}_i}^{\infty}(\ul{\lambda}), \tilde{\Sigma}_{\sigma}(\ul{\lambda})\big)
\end{equation*}
is bijective. We have thus a natural injection
\begin{equation*}
  \Hom_{\sigma}(T(L),E)/\Hom_{\sigma}(Z_{i}(L),E) \hooklongrightarrow \Ext^1_G\big(v_{\overline{P}_i}^{\infty}(\ul{\lambda}), \Sigma_{i,\sigma}(\ul{\lambda})\big),
\end{equation*}
which has to be bijective since  both of the source and target are two dimensional $E$-vector spaces (cf. Lemma \ref{lem: ext1} (3)). The proposition follows.
\end{proof}

\begin{remark}\label{rem: sigmaan}(i) By the proof, we see that any  $G$-representation $V\in \Ext^1_G\big(v_{\overline{P}_i}^{\infty}(\ul{\lambda}), \Sigma_{i,\sigma}(\ul{\lambda})\big)$ is $\text{U}(\ug_{\Sigma_L\setminus \{\sigma\}})$-finite.
\\

\noindent
(ii) Let $\Pi_{i,\sigma}^1(\ul{\lambda})\in \Ext^1_{G}\big(v_{\overline{P}_i}^{\infty}(\ul{\lambda}), C_{i,\sigma}\big)$
be the unique non-split element (cf. Lemma \ref{lem: stij} (2), note also $\Pi_{i,\sigma}^1(\ul{\lambda})$ is actually locally  $\text{U}(\ug_{\Sigma_L\setminus \{\sigma\}})$-finite by (i)). One can expect
\begin{equation*}
  \dim_E \Ext^1_G\big(\Pi_{i,\sigma}^1(\ul{\lambda}), \St_n^{\infty}(\ul{\lambda})\big){\buildrel {?}\over =}2,
\end{equation*}
and that $\Pi_{i,\sigma}^1(\ul{\lambda})$ is a subrepresentation (up to twist by locally algebraic representations) of the conjectural representations in \cite[(EXT)]{Br16}.
\end{remark}
\noindent
Let $i \in \Delta$, and $\psi \in \Hom(L^{\times}, E)\cong \Hom(T(L),E)/\Hom(Z_i,E)$ (see (\ref{equ: lgln-Linv0})). We denote by $\Sigma_i(\ul{\lambda}, \psi)$ the extension of $v_{\overline{P}_i}^{\infty}(\ul{\lambda})$ by $\Sigma_i(\ul{\lambda})$ attached to $\psi$ via  (\ref{equ: BrLi}), which is thus a subrepresentation of $\tilde{\Sigma}_i(\ul{\lambda}, \psi)$. Moreover, if $\psi\in \Hom_{\sigma}(L^{\times}, E)$, $\Sigma_i(\ul{\lambda}, \psi)$ is isomorphic to the push-forward of an extension, denoted by $\Sigma_{i,\sigma}(\ul{\lambda}, \psi)$, of $v_{\overline{P}_i}^{\infty}(\ul{\lambda})$ by $\Sigma_{i,\sigma}(\ul{\lambda})$. Actually $\Sigma_{i,\sigma}(\ul{\lambda}, \psi)$ is isomorphic to the extension attached to $\psi$ via (\ref{equ: Lsigma}). For $\alpha\in E^{\times}$, we put $\Sigma_*(\alpha, \ul{\lambda}):=\Sigma(\ul{\lambda}) \otimes_E \unr(\alpha)\circ \dett$ with $*\in \{\emptyset, i,\{i,\sigma\}\}$, and $\Sigma_i(\alpha, \ul{\lambda}, \psi):=\Sigma_i(\ul{\lambda}, \psi)\otimes_E \unr(\alpha) \circ \dett$; if $\psi$ is locally $\sigma$-analytic, we put $\Sigma_{i,\sigma}(\alpha, \ul{\lambda}, \psi)=\Sigma_{i,\sigma}(\ul{\lambda}, \psi)\otimes_E \unr(\alpha)\circ \dett$. Note that if $\psi\neq 0$, then
\begin{equation}\label{soc2}
  \soc_G \Sigma_i(\alpha, \ul{\lambda}, \psi) \cong \soc_G\Sigma(\alpha, \ul{\lambda})\cong \St_n^{\infty}(\alpha, \ul{\lambda}).
\end{equation}
\section{Fontaine-Mazur simple $\cL$-invariants}\label{sec: lgln-LFM}
\subsection{Fontaine-Mazur simple $\cL$-invariants}\label{sec: lgln-LFM1}
\noindent Let $k_{\sigma}\in \Z_{\geq 1}$ for all $\sigma\in \Sigma_L$, and  $\delta:=\unr(q_L^{-1})\prod_{\sigma\in \Sigma_L} \sigma^{k_{\sigma}}=\varepsilon\prod_{\sigma\in \Sigma_L} \sigma^{k_{\sigma}-1}$. We call such characters \emph{special}.
Let $\cR_L:=B_{\rig,L}^{\dagger}$, and $\cR_E:=\cR_L\otimes_{\Q_p} E$.  The cup-product induces a perfect pairing (e.g. see \cite[Lem. 1.13]{Ding4})
\begin{equation}\label{equ: lgln-ccup}
\langle \cdot, \cdot \rangle:  \hH^1_{(\varphi,\Gamma)}(\cR_E(\delta)) \times \hH^1_{(\varphi,\Gamma)}(\cR_E) \xlongrightarrow{\cup} \hH^2_{(\varphi,\Gamma)}(\cR_E(\delta))\cong E.
\end{equation}
The local class field theory gives a natural isomorphism $\hH^1_{(\varphi,\Gamma)}(\cR_E)\cong \Hom(L^{\times}, E)$. For $[D]\in \hH^1_{(\varphi,\Gamma)}(\cR_E(\delta))$, $[D]$ non-split, then $\cL(D):=[D]^{\perp}:=\{\psi \in \Hom(L^{\times}, E), \langle [D], \psi \rangle=0\}$ is an $E$-vector subspace of $\Hom(L^{\times}, E)$ of codimension one, which determines $[D]$ (since (\ref{equ: lgln-ccup}) is perfect). Recall also (\ref{equ: lgln-ccup}) induces an isomorphism $\hH^1_{(\varphi,\Gamma),e}(\cR_E(\delta))^{\perp} \cong \Hom_{\infty}(L^{\times}, E)$ (e.g. see \cite[Prop. 1.9, Lem. 1.15]{Ding4}), thus $\Hom_{\infty}(L^{\times}, E)\subseteq \cL(D)$ if and only if $D$ is crystalline.
\begin{remark}\label{rem: lgln-Lfmex}
For $\sigma\in \Sigma_L$, denote by $\cL(D)_{\sigma}:=\cL(D)\cap \Hom_{\sigma}(L^{\times}, E)$. We know $\Hom_{\sigma}(L^{\times}, E)$ is $2$-dimensional and admits a basis $\log_{p,\sigma}:=\sigma\circ \log_p$ and $\val_L$, where $\log_p: L^{\times} \ra L$ is equal to the $p$-adic logarithme when restricted to $\co_L^{\times}$, and $\log_p(p)=0$. Suppose $D$ is non-crystalline, then $\dim_E \cL(D)_{\sigma}=1$ for all $\sigma\in \Sigma_L$, and $\cL(D)\cong \oplus_{\sigma\in \Sigma_L} \cL(D)_{\sigma}$. Moreover, there exists $\cL_{\sigma}\in E$ such that $\cL(D)_{\sigma}$ is generated by $\log_{p,\sigma} -\cL_{\sigma} \val_L$. When $D$ is associated (up to twist) to a non-critical Galois representation $\rho_L$ of $\Gal_L$, one can actually prove $\{\cL_{\sigma}\}_{\sigma\in \Sigma_L}$ coincides with the usual Fontaine-Mazur $\cL$-invariants of $\rho_L$ defined in terms of the Hodge filtration.
\end{remark}
\noindent Let $D$ be a rank $n$ triangular $(\varphi,\Gamma)$-module, with a fixed triangulation $\cF$ of parameter given by the ordered set $\{\delta_i\}_{i=1,\cdots, n}$ of characters $\delta_i: L^{\times} \ra E^{\times}$, i.e. $\cF$ signifies an increasing filtration \begin{equation}\label{Fil0}
\cF: \ 0 =\Fil^0 D \subsetneq \Fil^1 D \subsetneq \cdots \subsetneq \Fil^{n} D=D
\end{equation}
where $\Fil^i D$ is a $(\varphi,\Gamma)$-submodule of $D$ such that $\Fil^i D/\Fil^{i-1} D\cong \cR_E(\delta_i)$ for $i=1, \cdots, n$. We denote by $D_i^j:=\Fil^j D/\Fil^i D$ for $i\leq j$ which is thus of parameter $(\delta_i, \cdots, \delta_j)$. We call $(D, \{\delta_i\})$ \emph{non-critical special} if $D_i^{i+1}$ is non-split and $\delta_i\delta_{i+1}^{-1}$ is special for all $i\in \Delta=\{1, \cdots, n-1\}$. For a non-critical special $(D,\{\delta_i\})$, and $i\in \Delta$, we can attach $\cL(D)_i:=\cL(D_i^{i+1} \otimes_{\cR_E} \cR_E(\delta_{i+1}^{-1}))\subseteq \Hom(L^{\times}, E)$ with $\dim_E \cL(D)_i=d_L$ (since $D_i^{i+1}$ is assumed to be non-split).
We call $\cL(D):=\prod_{i\in \Delta} \cL(D)_i$ the Fontaine-Mazur \emph{simple} $\cL$-invariants of $(D,\{\delta_i\})$, or for simplicity, of $D$.  Suppose moreover $D$ is semi-stable (note $D$ is always semi-stable up to twist), there exists thus $\alpha\in E^{\times}$ such that $\delta_1$ is the product of $\unr(\alpha)$ with an algebraic character. Let $D_{\st}(D)$ be the attached filtered $(\varphi, N)$-module.
\begin{lemma}\label{monofull}
  We have $N^{n-1}\neq 0$ on $D_{\st}(D)$ if and only if $D_i^{i+1}$ is semi-stable non-crystalline for all $i\in \Delta$.
\end{lemma}
\begin{proof}
Suppose there exists $i\in \Delta$ such that $D_i^{i+1}$ is crystalline. We have thus $N^{n-i-1}=0$ on $D_{\st}(D_{i+1}^n)$, $N^{i-1}=0$ on $D_{\st}(D_1^{i-1})$ and $N=0$ on $D_i^{i+1}$. By an easy d\'evissage argument,  it is straightforward  to deduce $N^{n-1}=N^{(n-i-1)+(i-1)+1}=0$ on $D_{\st}(D)$. The ``only if" part follows.  Conversely, suppose $D_i^{i+1}$ is non-crystalline for all $i\in \Delta$. we use induction on $n$ to show $N^{n-1}\neq 0$.  If $\rk D=2$, it is  trivial. Suppose that $N^{n-2}\neq 0$ on $D_1^{n-1}$. Consider the natural exact sequence of filtered $(\varphi, N)$-modules
\begin{equation*}
  0 \ra D_{\st}(D_1^{n-1}) \ra D_{\st}(D) \ra D_{\st}(\cR_E(\delta_n)) \ra 0.
\end{equation*}
Let $e_n\in D_{\st}(D)$ be a $\varphi^{f_L}$-eigenvector of eigenvalue $q_L^{n-1}\alpha$ (recall $f_L$ is the unramified degree of $L$ over $\Q_p$), which is thus a lifting of a non-zero element in $D_{\st}(\cR_E(\delta_n))$ (note by assumption, $\delta_n$ is a product of $\unr(q_L^{n-1}\alpha)$ with an algebraic character). By assumption, $N(e_n)\in D_{\st}(D_1^{n-1})$ is non-zero since $D_{n-1}^n$ is non-crystalline, and $\varphi^{f_L}(N(e_n))=q_L^{n-2}\alpha$. By induction hypothesis, we also know $N^{n-2}(N(e_n))\neq 0$ (since otherwise, this together with the fact $N^{n-2}=0$ on $D_{\st}(D_1^{n-2})$ imply $N^{n-2}=0$ on $D_{\st}(D_1^{n-1})$, a contradiction) and hence $N^{n-1}(e_n)\neq 0$.
\end{proof}
\noindent In summary, for a non-critical special semi-stable $D$, we can attach $\alpha\in E^{\times}$, a strictly dominant weight $\wt(\ul{\delta})=(\wt(\delta_1)_{\sigma}, \cdots, \wt(\delta_n)_{\sigma})_{\sigma\in \Sigma_L}\in X_{\Delta}^+$ and $\cL(D)=\prod_{i\in \Delta} \cL(D)_i$. Note that when $n>2$, in general, one can not recover $D$ from the data $\{\alpha, \wt(\ul{\delta}), \cL(D)\}$, and we refer to \cite{BD1} for a discussion on other (higher) $\cL$-invariants.
\subsection{A sub candidate in locally analytic $p$-adic local Langlands program}\label{sec: lgln-ll}
\noindent Let $D$ be a non-critical special semi-stable $(\varphi, \Gamma)$-module of rank $n$ over $\cR_E$, with the attached data $\{\alpha, \wt(\ul{\delta}), \cL(D)\}$ as in the precedent section (thus $\delta_i=\unr(q_L^{i-1}\alpha) \prod_{\sigma\in \Sigma_L} \sigma^{\wt(\delta_i)_{\sigma}}$). Put
\begin{equation*}\ul{\lambda}:=(\lambda_{1,\sigma}, \cdots, \lambda_{n,\sigma})_{\sigma\in \Sigma_L}\in X_{\Delta}^+\end{equation*} with $\lambda_{i,\sigma}:=\wt(\delta_i)_{\sigma}+i-1$. For $i\in \Delta$, let $\{\psi_{i,1}, \cdots, \psi_{i,d_L}\}$ be a basis of $\cL(D)_i$ and put
\begin{eqnarray*}
  \tilde{\Sigma}_i(\alpha,\ul{\lambda}, \cL(D)_i)&:=&\bigoplus_{\St_n^{\an}(\alpha, \ul{\lambda})}^{j=1,\cdots, d_L} \tilde{\Sigma}_i(\alpha, \ul{\lambda}, \psi_{i,j}), \\
\Sigma_i(\alpha,\ul{\lambda}, \cL(D)_i)&:=&\bigoplus_{\Sigma_i(\alpha, \ul{\lambda})}^{j=1,\cdots, d_L} \Sigma_i(\alpha, \ul{\lambda}, \psi_{i,j}).
\end{eqnarray*}
Thus $\Sigma_i(\alpha,\ul{\lambda}, \cL(D)_i)\subset \tilde{\Sigma}_i(\alpha, \ul{\lambda}, \cL(D)_i)$, and by (\ref{equ: lgln-extc}) (\ref{equ: BrLi})  both of the representations are independent of the choice of the basis of $\cL(D)_i$ and determine $\cL(D)_i$. Put
\begin{eqnarray*}
  \tilde{\Sigma}(\alpha,\ul{\lambda}, \cL(D))&:=&\bigoplus_{\St_n^{\an}(\alpha, \ul{\lambda})}^{i\in \Delta} \tilde{\Sigma}_i(\alpha, \ul{\lambda}, \cL(D)_i), \\
\Sigma(\alpha,\ul{\lambda}, \cL(D))&:=&\bigoplus_{\St_n^{\infty}(\alpha, \ul{\lambda})}^{i\in \Delta} \Sigma_i(\alpha, \ul{\lambda}, \cL(D)_i).
\end{eqnarray*}
We have $\Sigma(\alpha,\ul{\lambda}, \cL(D))\subset \tilde{\Sigma}(\alpha,\ul{\lambda}, \cL(D))$, and the both determine exactly the data $\{\alpha, \ul{\lambda}, \cL(D)\}$ and vice versa.
\begin{lemma}
  (1) $\soc_G\Sigma_i(\alpha,\ul{\lambda}, \cL(D)_i)\cong \soc_G \Sigma(\alpha, \ul{\lambda}, \cL(D))\cong \St_n^{\infty}(\alpha, \ul{\lambda})$.

\noindent
  (2) The locally algebraic subrepresentation of $\tilde{\Sigma}(\alpha,\ul{\lambda}, \cL(D))$ (resp. of $\tilde{\Sigma}_i(\alpha,\ul{\lambda}, \cL(D)_i)$ for $i\in \Delta$) is isomorphic to $\St_n^{\infty}(\alpha, \ul{\lambda})$ if and only if $N^{n-1}\neq 0$ on $D_{\st}(D)$ (resp. if and only if $D_i^{i+1}$ is non-crystalline).
\end{lemma}
\begin{proof}
  (1) follows from (\ref{soc2}).

  \noindent
  By Remark \ref{rem: lgln-ext3} (ii), we see that the locally algebraic subrepresentation of $\tilde{\Sigma}_i(\alpha,\ul{\lambda}, \cL(D)_i)$ is strictly bigger than $\St_n^{\infty}(\ul{\lambda})$ if and only if $\Hom_{\infty}(L^{\times}, E)\subseteq \cL(D)_i$ which is equivalent to $D_i^{i+1}$ being crystalline (see the discussion in \S \ref{sec: lgln-LFM1}). The statement for $\tilde{\Sigma}_i(\alpha,\ul{\lambda}, \cL(D)_i)$ follows. By definition, it is easy to see $\tilde{\Sigma}(\alpha, \ul{\lambda}, \cL(D))$ has locally algebraic vectors other than $\St_n^{\infty}(\alpha, \ul{\lambda})$ if and only if so does $\tilde{\Sigma}_i(\alpha, \ul{\lambda}, \cL(D)_i)$ for certain $i\in \Delta$. Together with Lemma \ref{monofull}, the statement for $\tilde{\Sigma}(\alpha, \ul{\lambda}, \cL(D))$ follows.
\end{proof}
\noindent Suppose $N^{n-1}\neq 0$ on $D_{\st}(D)$. For $i\in \Delta$, $\sigma\in \Sigma_L$, there exists $\cL_{i,\sigma}\in E$ such that $\cL(D)_i\cap \Hom_{\sigma}(L^{\times}, E)$ is generated by $\psi_{i,\sigma}:=\log_{p, \sigma}-\cL_{i,\sigma} \val_L$. The characters $\{\psi_{i,\sigma}\}_{\sigma\in \Sigma_L}$ form a basis of $\cL(D)_i$. Since $\psi_{i,\sigma}$ is locally $\sigma$-analytic, by (\ref{equ: BrLi}) Proposition \ref{prop: sigmaan}, we have (recall $\Sigma_{i,\sigma}(\alpha, \ul{\lambda}, \psi_{i,\sigma})$ is $\text{U}(\ug_{\Sigma_L\setminus\{\sigma\}})$-finite)
\begin{equation*}\Sigma_i(\alpha, \ul{\lambda}, \psi_{i,\sigma})\cong \Sigma_{i,\sigma}(\alpha, \ul{\lambda}, \psi_{i,\sigma}) \oplus_{\St_n^{\infty}(\alpha, \ul{\lambda})} \big(\bigoplus_{\St_n^{\infty}(\alpha, \ul{\lambda})}^{\tau\in \Sigma_L, \tau\neq \sigma} \Sigma_{i,\tau}(\alpha, \ul{\lambda})\big).\end{equation*}
In this case, using the basis $\{\psi_{i,\sigma}\}_{i\in \Delta, \sigma\in \Sigma_L}$ of $\cL(D)$, we see that $\Sigma(\alpha, \ul{\lambda}, \cL(D))$ has the following form (where a line denotes a \emph{non-split} extension with the left object the sub, and the right object the quotient, and we enumerate the embeddings in $\Sigma_L$ for convenience):
\begin{equation}\label{stru}
  \begindc{\commdiag}[200]
    \obj(0,3)[a]{$\St_n^{\infty}(\alpha, \ul{\lambda})$}
    \obj(4,6)[b]{$C_{1,\sigma_1}$}
    \obj(8,6)[c]{$v_{\overline{P}_1}^{\infty}(\alpha,\ul{\lambda})$}
    \obj(4,5)[d]{$\vdots$}
    \obj(8,5)[e]{$\vdots$}
    \obj(4,4)[f]{$C_{1, \sigma_{d_L}}$}
    \obj(8,4)[g]{$v_{\overline{P}_1}^{\infty}(\alpha, \ul{\lambda})$}
    \obj(4,3)[h]{$\vdots$}
    \obj(8,3)[i]{$\vdots$}
    \obj(4,2)[j]{$C_{n-1,\sigma_1}$}
    \obj(8,2)[k]{$v_{\overline{P}_{n-1}}^{\infty}(\alpha, \ul{\lambda})$}
    \obj(4,1)[l]{$\vdots$}
    \obj(8,1)[m]{$\vdots$}
    \obj(4,0)[n]{$C_{n-1,\sigma_{d_L}}$}
    \obj(8,0)[o]{$v_{\overline{P}_{n-1}}^{\infty}(\alpha, \ul{\lambda})$}
    \mor{a}{b}{}[+1,\solidline]
    \mor{b}{c}{}[+1,\solidline]
    \mor{a}{f}{}[+2,\solidline]
    \mor{f}{g}{}[+1,\solidline]
       \mor{a}{j}{}[+1,\solidline]
    \mor{j}{k}{}[+1,\solidline]
       \mor{a}{n}{}[+1,\solidline]
    \mor{n}{o}{}[+1,\solidline]
  \enddc,
\end{equation}
where the $(i,\sigma)$-th branch carries the exact information on $\cL_{i,\sigma}$, and is $\text{U}(\ug_{\Sigma_L\setminus \{\sigma\}})$-finite.
\\

\noindent
Let $\rho_L$ be an $n$-dimensional semi-stable representation of $\Gal_L$ over $E$ with $N^{n-1}\neq 0$ on $D_{\st}(\rho_L)$, and let $D:=D_{\rig}(\rho_L)$. Suppose $D$ is non-critical (actually the condition $N^{n-1}\neq 0$ implies that  $D$ admits a unique triangulation, so in this case, we just say $D$ is non-critical if its triangulation is non-critical). We put $\cL(\rho_L):=\cL(D)$. The information on $\cL(\rho_L)$ is lost when passing from $\rho_L$ to its associated Weil-Deligne representation and its Hodge-Tate weights, and hence is invisible in classical local Langlands correspondance. Actually, by classical local Langlands correspondence (combined with the theory on weights), we can only see the locally algebraic representation $\St_n^{\infty}(\alpha, \ul{\lambda})$ (where $\alpha$, $\ul{\lambda}$ are attached to $D$ as above). We expect that $\Sigma(\alpha, \ul{\lambda}, \cL(\rho_L))$ is a sub of the right representation corresponding to $\rho_L$ in the  $p$-adic Langlands program.

\subsection{Colmez-Greenberg-Stevens formula}\noindent
Let $A$ be an affinoid $E$-algebra, and $D_A$ be a rank $n$ triangulable $(\varphi,\Gamma)$-module over $\cR_A:=\cR_L \widehat{\otimes}_{\Q_p} A$ of trianguline parameter $\{\delta_{A,i}\}_{i=1,\cdots, n}$, i.e. $D_A$ is a successive extension of $\cR_A(\delta_{A,i})$, which is the rank $1$ $(\varphi,\Gamma)$-module over $\cR_A$ associated to $\delta_{A,i}: L^{\times} \ra A^{\times}$ (cf. \cite[Const. 6.2.4]{KPX}). Let  $z\in \Spm A$ be an $E$-point, $D:=D_A|_z$, $\delta_i:=\delta_{A,i}|_z$, and suppose that  $(D, \delta_i)$ is non-critical special.  Recall (e.g. see \cite[Thm. 0.1]{Ding4})
\begin{theorem}\label{thm: CGS}Let $i\in \Delta=\{1,\cdots, n-1\}$,  let $t:\Spec E[\epsilon]/\epsilon^2 \ra \Spm A$ be an element in the tangent space of $A$ at $z$, $\widetilde{\delta}_i:=t^* \delta_{A,i}: L^{\times} \ra (E[\epsilon]/\epsilon^2)^{\times}$, and  $\psi\in \Hom(L^{\times},E)$ be the additive character such that $\widetilde{\delta}_i\widetilde{\delta}_{i+1}^{-1}=\delta_i\delta_{i+1}^{-1}(1+\epsilon\psi)$, then
    $\psi\in \cL(D)_i$.
\end{theorem}
\noindent We also have a ``\emph{converse}'' version of the theorem.
\begin{proposition}\label{prop: lgln-trd}
Let $\psi_i\in \cL(D)_i$ for $i\in \Delta$, let $\widetilde{\delta}_1: L^{\times} \ra (E[\epsilon]/\epsilon^2)^{\times}$ with $\widetilde{\delta}_1\equiv \delta_1 \pmod{\epsilon}$, and let $\widetilde{\delta}_{i+1}: L^{\times} \ra (E[\epsilon]/\epsilon^2)^{\times}$ be such that $\widetilde{\delta}_i\widetilde{\delta}_{i+1}^{-1}=\delta_i\delta_{i+1}^{-1}(1+\epsilon\psi_i)$. Then there exists a trianguline $(\varphi,\Gamma)$-module $\tilde{D}$ over $\cR_{E[\epsilon]/\epsilon^2}$ of parameter $\big(\widetilde{\delta}_1, \cdots, \widetilde{\delta}_n\big)$ such that $\tilde{D}\equiv D \pmod{\epsilon}$.
\end{proposition}
\begin{proof}
  We use induction on the rank of $D$. The rank $1$ case is trivial. Suppose the statement holds if the rank of $D$ is less than $n-1$, and now assume $D$ is of rank $n$. By the induction hypothesis, there exists a trianguline $(\varphi,\Gamma)$-module $\tilde{D}_1^{n-1}$ over $\cR_{E[\epsilon]/\epsilon^2}$ of parameter $\big(\widetilde{\delta}_1,\cdots, \widetilde{\delta}_{n-1}\big)$ such that  $\tilde{D}_1^{n-1} \equiv D_1^{n-1} \pmod{\epsilon}$. Twisting $D$ (resp. $\tilde{D}_1^{n-1}$) by $\cR_E(\delta_n^{-1})$ (resp. by $\cR_{E[\epsilon]/\epsilon^2}(\widetilde{\delta}_n^{-1})$), we can and do assume $\delta_n=\widetilde{\delta}_n=1$. We have $D\in \hH^1_{(\varphi,\Gamma)}(D_1^{n-1})$, $D_{n-1}^n\in \hH^1_{(\varphi,\Gamma)}(\cR_E(\delta_{n-1}))$.
  We have the following commutative diagram of $(\varphi,\Gamma)$-modules
  \begin{equation*}
    \begin{CD}
      0 @>>> D_{1}^{n-1} @>>> \tilde{D}_1^{n-1} @>>> D_1^{n-1} @>>> 0 \\
      @. @VVV  @VVV @VVV @. \\
      0 @>>> \cR_E(\delta_{n-1}) @>>> \cR_{E[\epsilon]/\epsilon^2}(\widetilde{\delta}_{n-1}) @>>> \cR_E(\delta_{n-1}) @>>> 0
    \end{CD}
  \end{equation*}
where the vertical maps are natural projections. Taking cohomology, we get
\begin{equation*}
  \begin{CD}
    \hH^1_{(\varphi,\Gamma)}\big(\tilde{D}_1^{n-1}\big) @> p_0 >> \hH^1_{(\varphi,\Gamma)}(D_1^{n-1}) @>\delta >> \hH^2_{(\varphi,\Gamma)}(D_{1,n-1}) \\
    @V p_1 VV @V p_2 VV @V p_3 VV \\
    \hH^1_{(\varphi,\Gamma)}\big(\cR_{E[\epsilon]/\epsilon^2}(\widetilde{\delta}_{n-1})\big) @> p_0' >> \hH^1_{(\varphi,\Gamma)}(\cR_E(\delta_{n-1})) @>\delta' >>  \hH^2_{(\varphi,\Gamma)}(\cR_E(\delta_{n-1}))
  \end{CD}.
\end{equation*}
It is sufficient to prove $D\in \Ima(p_0)$. However, by \cite[Thm. 1.2]{Liu07} and an easy d\'evissage argument, we have that $p_3$ is an isomorphism, and $p_1, p_2$ are surjective \big(indeed, $\hH^2_{(\varphi,\Gamma)}(\cR_E(\delta_i))=0$ for $i\leq n-2$\big). By an easy diagram chasing, we are reduced to prove $D_{n-1}^n\in \Ima(p_0')$. However, this follows from our assumption on $\psi_{n-1}$ and \cite[(2.3)]{Ding4}. This concludes the proof.
\end{proof}

\subsection{Trianguline deformations}\label{sec: triDef}
\noindent We study trianguline deformations of non-critical special $(\varphi,\Gamma)$-modules. Let $(D, \delta_i)$ be a  non-critical special $(\varphi,\Gamma)$-module. Recall that the trianguline deformation functor
\begin{equation*}
\normalsize F_{D,\cF}:  \Big\{\substack{\text{Artinian local $E$-algebras}\\\text{ with residue field $E$}}\Big\} \lra \big\{\text{Sets}\big\}
\end{equation*}
sends $A$ to the set of isomorphism classes $F_{D,\cF}(A)=\{(D_A,\cF_A, \pi_A)\}/\sim$ where
\begin{enumerate}
  \item $D_A$ is a rank $n$ $(\varphi,\Gamma)$-module over $\cR_A$ with $\pi_A: D_A \otimes_{\cR_A} \cR_E\xrightarrow{\sim} D$,
  \item $\cF_A$ is an increasing filtration (see (\ref{Fil0}))
    \begin{equation*}
\cF_A: \       0 =\Fil^0_A D_A \subsetneq \Fil^1_A D_A \subsetneq \cdots \subsetneq \Fil^n_A D_A=D_A,
    \end{equation*}
where $\Fil^i A$ is a $(\varphi,\Gamma)$-submodule of $D_A$ over $\cR_A$ being  a direct summand as $\cR_A$-submodule.
\end{enumerate}
Two objects $(D_A, \cF_A, \pi_A)$ and $(D_A', \cF_A', \pi_A')$ are isomorphic if there exists an isomorphism of $(\varphi,\Gamma)$-modules over $\cR_A$: $f: D_A\ra D_A'$ which respects the filtrations and $\pi_A=\pi'_A \circ f$. For $(D_A, \cF_A, \pi_A)\in F_{D,\cF}(A)$ and $i=1,\cdots, n$, $\Fil^i_A D_A/\Fil^{i-1}_A D_A$ is thus a $(\varphi,\Gamma)$-module of rank $1$ over $\cR_A$, which, by \cite[Prop. 2.3.1]{Bch} \cite[Prop. 2.16]{Na2}, is isomorphic to $\cR_A(\delta_{A,i})$ for a unique continuous character $\delta_{A,i}: L^{\times} \ra A^{\times}$. We have the following $E$-linear maps
\begin{equation}\label{equ: kappa}
  \kappa: F_{D,\cF}(E[\epsilon]/\epsilon^2) \lra \Hom(T(L),E) \xlongrightarrow{\kappa_L} \prod_{i\in \Delta} \Hom(L^{\times},E)
\end{equation}
where the first map sends $(D_{E[\epsilon]/\epsilon^2}, \cF_{E[\epsilon]/\epsilon^2}, \pi_{E[\epsilon]/\epsilon^2})$ to $\big(\big(\delta_{E[\epsilon]/\epsilon^2,i}\delta_i^{-1}-1\big)/\epsilon\big)_{i=1,\cdots, n}$ with $\{\delta_{E[\epsilon]/\epsilon^2, i}\}_{i=1,\cdots, n}$ associated to $\cF_{E[\epsilon]/\epsilon^2}$ as above, and $\kappa_L$ sends $(\psi_1, \cdots, \psi_n)$ to $(\psi_i-\psi_{i+1})_{i\in \Delta}$.
\begin{proposition}\label{prop: lgln-dtd}
(1) The functor $F_{D,\cF}$ is pro-representable, and is formally smooth of dimension $1+\frac{n(n+1)}{2} d_L$.

\noindent (2) The map $\kappa$ factors through a surjective map
\begin{equation*}
  \kappa: F_{D,\cF}(E[\epsilon]/\epsilon^2) \twoheadlongrightarrow \prod_{i\in \Delta} \cL(D)_i=\cL(D).
\end{equation*}
\end{proposition}
\begin{proof}The case for $n=1$ is trivial. We assume $n\geq 2$. (2) follows from Proposition \ref{prop: lgln-trd}. We prove (1).
We first show $\End_{(\varphi,\Gamma)}(D)\cong E$. Denote by $\End(D)$ the $\cR_E$-module of $\cR_E$-endomorphisms on $D$, which has a natural $(\varphi,\Gamma)$-module structure given by $\gamma(u)(x)=\gamma(u(\gamma^{-1}(x)))$ and $\varphi(u)(\varphi(x))=\varphi(u(x))$, for $u\in \End(D)$. We have $\End_{(\varphi,\Gamma)}(D)\cong \hH^0_{(\varphi,\Gamma)}(\End(D))$. We have moreover an exact sequence of $(\varphi,\Gamma)$-modules over $\cR_E$:
\begin{equation*}
  0 \ra \Hom_{\cR_E}(\cR_E(\delta_n), D) \ra \End(D) \ra \Hom_{\cR_E}(D_1^{n-1}, D) \ra 0.
\end{equation*}
We claim $\hH^0_{(\varphi,\Gamma)}(\Hom_{\cR_E}(\cR_E(\delta_n), D))=0$, and $\hH^0_{(\varphi,\Gamma)}(\End(D_1^{n-1})) \xrightarrow{\sim} \hH^0_{(\varphi,\Gamma)}(\Hom_{\cR_E}(D_1^{n-1}, D))$. Assuming the claim, then by an easy induction argument, we can deduce $E\cong \End_{(\varphi,\Gamma)}(D)$ (noting we have \emph{a priori} $E\hookrightarrow \End_{(\varphi,\Gamma)}(D)$). We prove the claim. Since $D_{n-1}^n$ is non-split (by the non-critical special assumption), we deduce
\begin{equation*}\hH^0_{(\varphi,\Gamma)}(\Hom_{\cR_E}(\cR_E(\delta_n), D_{n-1}^n))\cong \Hom_{(\varphi,\Gamma)}(\cR_E(\delta_n), D_{n-1}^n)=0\end{equation*}
 and hence, by an easy d\'evissage argument, $\hH^0_{(\varphi,\Gamma)}(\Hom_{\cR_E}(\cR_E(\delta_n), D))=0$. It is also clear that \begin{equation*}\hH^0_{(\varphi,\Gamma)}(\Hom_{\cR_E}(D_1^{n-1}, \cR_E(\delta_n)))\cong \Hom_{(\varphi,\Gamma)}(D_1^{n-1}, \cR_E(\delta_n))=0.\end{equation*}
From which, we deduce the second part of the claim.
\\

\noindent Put
\begin{equation*}
  \End_{\cF}(D):=\{f\in \End(D)\ |\ f(\Fil^i D)\subseteq \Fil^i D\},
\end{equation*}
which is a saturated $(\varphi,\Gamma)$-submodule of $\End(D)$. We calculate $\hH^i_{(\varphi,\Gamma)}(\End_{\cF}(D))$. Since $\End_{(\varphi,\Gamma)}(D)\cong E$, we easily see $\hH^0_{(\varphi,\Gamma)}(\End_{\cF}(D))=E$. We have a natural exact sequence of $(\varphi,\Gamma)$-modules over $\cR_E$
\begin{equation*}
  0 \ra \Hom_{\cR_E}(\cR_E(\delta_n), D) \ra \End_{\cF}(D) \ra \End_{\cF}(D_1^{n-1}) \ra 0
\end{equation*}
where we also use $\cF$ to denote the induced filtration on $D_1^{n-1}\subset D$. Continuing with such argument, we easily deduce the following facts:
\begin{itemize}
  \item $\End_{\cF}(D)$ is a $(\varphi,\Gamma)$-module over $\cR_E$ of rank $n+(n-1)+ \cdots +1=\frac{n(n+1)}{2}$;
  \item $\End_{\cF}(D)$ is isomorphic to an extension of $(\varphi,\Gamma)$-modules $D_1^i \otimes_{\cR_E} \cR_E(\delta_i^{-1})$ with $i=1, \cdots, n$.
\end{itemize}
By \cite[Thm. 1.2 (2)]{Liu07} , we have  an isomorphism for $i=1,\cdots, n$:
\begin{equation*}\hH^2_{(\varphi,\Gamma)}(D_1^i \otimes_{\cR_E} \cR_E(\delta_i^{-1}))\cong \hH^0_{(\varphi,\Gamma)}((D_1^i)^{\vee}\otimes_{\cR_E} \cR_E(\delta_i\varepsilon)).\end{equation*}
Let $M_i:=(D_{i-1}^i)^{\vee} \otimes_{\cR_E} \cR_E(\delta_i \varepsilon)$ if $i\geq 2$, and $M_1:=(D_1^1)^{\vee} \otimes_{\cR_E} \cR_E(\delta_1 \varepsilon)\cong \cR_E(\varepsilon)$.
By an easy d\'evissage argument (and the explicit structure of special characters), we have for $i\geq 1$
\begin{equation}\label{devi00}
  \hH^0_{(\varphi,\Gamma)}(M_i)\xlongrightarrow{\sim}\hH^0_{(\varphi,\Gamma)}((D_1^i)^{\vee}\otimes_{\cR_E} \cR_E(\delta_i\varepsilon)).
\end{equation}
It is clear that $ \hH^0_{(\varphi,\Gamma)}(M_1)=0$. For $i\geq 2$, since $D_i^{i+1}$ is non-split,  $M_i$ is isomorphic to a non-split extension of $\cR_E(\delta_{i-1}^{-1}\delta_i \varepsilon)$ by $\cR_E(\varepsilon)$. If $\hH^0_{(\varphi,\Gamma)}(M_i)\neq 0$, we have an injection of $(\varphi,\Gamma)$-modules $j: \cR_E \hookrightarrow M_i$. By considering the Sen weights, it is not difficult to see the $\Ima(j)[1/t]\cap M_i\cong \cR_E(\delta_{i-1}^{-1}\delta_i \varepsilon)$ and hence $M_i$ is split, a contradiction. Together with (\ref{devi00}), we see then for all $i=1,\cdots, n$
\begin{equation*}
  \hH^2_{(\varphi,\Gamma)}(D_1^i \otimes_{\cR_E} \cR_E(\delta_i^{-1}))=0
\end{equation*}
and hence by d\'evissage, $\hH^2_{(\varphi,\Gamma)}(\End_{\cF}(D))=0$. By \cite[Thm. 1.2 (1)]{Liu07}, we see
\begin{equation*}\dim_E \hH^1_{(\varphi,\Gamma)}(\End_{\cF}(D))=1+\frac{n(n+1)}{2}d_L.
\end{equation*}
By \cite[Prop. 3.6]{Che11} (and an easy generalization), (1) follows.
\end{proof}

\section{Local-global compatibility}\label{sec: lgln-LG}
\noindent In this section, we prove some local-global compatibility result in semi-stable non-crystalline case. In particular, we show that the correspondence  between Fontaine-Mazur and Breuil's simple $\cL$-invariants can be realised in the space of the (patched) $p$-adic automorphc representations.
%We choose to work with the patched Banach representations (\cite{CEGGPS}) rather than those obtained via genuine $p$-adic automorphic representations (e.g. see (\ref{equ: lgln-afp})), roughly because the geometry of the resulted so-called patched eigenvariety (\cite{BHS1}) seems easier to study than the genuine eigenvariety (for our concern), and similar results in the classical setting could be in general deduced from those in the  patched setting (e.g. see the discussion in \S \ref{sec: lgln-pbr}).
\subsection{Patched eigenvariety}
\subsubsection{Patched Banach representation of $\GL_n(L)$}\label{sec: lgln-pbr}
\noindent We recall the (patched) $p$-adic automorphic representations considered in \cite{CEGGPS}. We follow the notation of \emph{loc. cit.} except for the CM field $F$, the $p$-adic field $L$, and the $p$-adic place $\wp$ (denoted by $\widetilde{F}$, $F$, $\fp$ respectively in \emph{loc. cit.}). Suppose $p\nmid 2n$, and let $\overline{r}: \Gal_L \ra \GL_n(k_E)$ be a continuous representation such that $\overline{r}$ admits a potentially crystalline lift $r_{\mathrm{pot.diag}}: \Gal_L \ra \GL_n(k_E)$  of regular weight which is potentially diagonalisable. Let $\xi$ (resp. $\tau$) be the weight (resp. the inertial type) of $r_{\mathrm{pot.diag}}$.  In this case, by \cite[\S~2.1]{CEGGPS}, we can find a triple $(F,F^+, \overline{\rho})$ \big(where $(F,F^+)$ is denoted by $(\widetilde{F}, \widetilde{F}^+)$ in \emph{loc. cit.}\big) where $F$ is an imaginary CM field with maximal totally real subfield $F^+$ satisfying that
 \begin{itemize}\item the extension $F/F^+$ is unramified at all finite places,
  \item for any $v|p$ of $F^+$, $v$ splits in $F$, and $F^+_v\cong L$,
  \item $\overline{\rho}$ is a \emph{suitable globalisation} (cf. \emph{loc. cit.}) of $\overline{r}$.
\end{itemize}

\noindent We use the setting of \cite[\S 2.3]{CEGGPS} (and we refer to \emph{loc. cit.} for details). In particular, we have the following objects
\begin{equation*}
  \{\widetilde{G}, v_1, S_p, \wp, \{U_m\}_{m\in \Z}\},
\end{equation*}
where
\begin{itemize} \item $\widetilde{G}$ is a certain definite unitary group over $F^+$ with a model over $\co_{F^+}$ such that
 for each place $v$ of $F^+$ which splits as $w w^c$ in $F$, we have (and fix) an isomorphism $i_w: \widetilde{G}(\co_{F^+})\xrightarrow{\sim} \GL_n(\co_{F_w})$;
 \item $v_1$ is a certain finite place of $F^+$ prime to $p$, and we refer to \emph{loc. cit.} for properties of $v_1$, noting $v_1$ splits in $F$;
 \item $S_p$ is the set of places of $F^+$ above $p$, and for any $v\in S_p\cup\{v_1\}$, let $\widetilde{v}|v$ be a place of $F$ such that if $v\in S_p$, then $\overline{\rho}|_{\Gal_{F_{\widetilde{v}}}}\cong \overline{r}$ (the existence of $\widetilde{v}$ follows from that $\overline{\rho}$ is a suitable globalisation of $\overline{r}$);
 \item $\wp\in S_p$;
 \item $\{U_m=\prod_{v} U_{m,v}\}_{m\in \Z_{\geq 0}}$ is a tower of compact open subgroups of $\widetilde{G}(\bA_{F^+}^{\infty})$ where
\begin{itemize}\item  $U_{m,v}=\widetilde{G}(\co_{F^+_v})$ if $v$ splits in $F$ and $v\neq \wp, v_1$,
  \item $U_{m,v}$ hyperspecial if $v$ is inert in $F$,
  \item $U_{m,v_1}$ is the preimage of upper triangular matrices of the composition ($k_{F_{\widetilde{v_1}}}$ denoting the corresponding residue field)
  \begin{equation*}
    \widetilde{G}(\co_{F^+_{v_1}}) \xrightarrow[\sim]{i_{\widetilde{v_1}}} \GL_n(\co_{F_{\widetilde{v_1}}}) \twoheadrightarrow \GL_n(k_{F_{\widetilde{v_1}}}),
  \end{equation*}
  \item $U_{m,\wp}$ is the kernel of $\widetilde{G}(\co_{F^+_{\wp}}) \rightarrow \widetilde{G}(\co_{F^+_{\wp}}/\varpi_{F^+_{\wp}}^m)$, where $\varpi_{F^+_{\wp}}$ denotes a uniformizer of $\co_{F^+_{\wp}}$.
  \end{itemize}
\end{itemize}

\noindent
For $(\xi,\tau)$, as in \emph{loc. cit.}, we can attach a finite free $\co_E$-module $L_{\xi,\tau}$ equipped with a locally algebraic representation of $\GL_n(\co_L)$. Put $W_{\xi,\tau}:=\otimes_{v\in S_p\setminus \{\wp\}} L_{\xi,\tau}$, which is equipped with an action of $\prod_{v\in S_p\setminus \{\wp\}} U_{m,v}$ with  $U_{m,v}$ acting on the factor corresponding to $v$ via $U_{m,v}\cong \GL_n(\co_L)$. Let
\begin{multline*}
  S_{\xi,\tau}(U_m, \co_E/\varpi_E^k):=\big\{f: \widetilde{G}(F^+)\backslash \widetilde{G}(\bA_{F^+}^{\infty}) \ra W_{\xi,\tau}\otimes_{\co_E} \co_E/\varpi_E^k\ |\ \\ f(gu)=u^{-1} f(g) \text{ for $g\in \widetilde{G}(\bA_{F^+}^{\infty})$, $u\in U_m$}\big\}
\end{multline*}
where $U_m$ acts on $W_{\xi,\tau}\otimes_{\co_E}  \co_E/\varpi_E^k$ via the projection $U_m\twoheadrightarrow \prod_{v\in S_p\setminus \{\wp\}} U_{m,v}$.  The $\co_E/\varpi_E^k$-module $S_{\xi,\tau}(U_m, \co_E/\varpi_E^k)$ is equipped with a natural action of the spherical Hecke operators
\begin{equation*}T_w^{(j)}=\bigg[U_{m,v} i_w^{-1}\bigg(\begin{pmatrix} \varpi_{F_w} 1_j & 0 \\ 0 & 1_{n-1}\end{pmatrix}\bigg)U_{m,v}\bigg]\end{equation*} where $w$ is a place of $F$ lying over a place $v\notin S_p \cup \{v_1\}$ of $F^+$ which splits in $F$, $\varpi_{F_w}$ is a uniformizer of $F_w$ and $j\in \{1, \cdots, n\}$. We denote by $\bT^{S_p,\univ}$ the $\co_E$-polynomial algebra generated by such $T_w^{(j)}$ and the formal variables $T_{\widetilde{v_1}}^{(j)}$. One can associate a maximal ideal $\fm$ of $\bT^{S_p,\univ}$ to the global Galois representation $\overline{\rho}$ as in the end of \cite[\S~2.3]{CEGGPS}.
We put $U^{\wp}:=\prod_{v\neq \wp} U_{m,v}$ (which is independent of $m$), and put
\begin{eqnarray}
  \widehat{S}_{\xi,\tau}(U^{\wp}, \co_E)_*&:=&\varprojlim_k \varinjlim_m S_{\xi,\tau}(U_m,\co_E/\varpi_E^k)_*, \nonumber\\
  \label{equ: lgln-afp}\widehat{S}_{\xi,\tau}(U^{\wp},E)_*&:=&\widehat{S}_{\xi,\tau}(U^{\wp}, \co_E)_*\otimes_{\co_E} E, \nonumber
\end{eqnarray}
with $*\in \{\fm, \emptyset\}$ (and $S_{\xi,\tau}(U_m,\co_E/\varpi_E^k)_{\fm}$ denotes the localization of $S_{\xi,\tau}(U_m,\co_E/\varpi_E^k)$ at $\fm$). Thus  $\widehat{S}_{\xi,\tau}(U^{\wp},E)_{*}$ is an admissible unitary Banach representation of $\GL_n(L)$. Actually, $\widehat{S}_{\xi,\tau}(U^{\wp},E)_{\fm}$ is a direct summand of $\widehat{S}_{\xi,\tau}(U^{\wp},E)$.  The action of $\bT^{S_p,\univ}$ on the localisation $\widehat{S}_{\xi,\tau}(U^{\wp}, \co_E)_{\fm}$ factors through the Hecke algebra
\begin{equation*}
 \bT_{\xi,\tau}^{S_p}(U^{\wp}, \co_E)_{\fm}:=\varprojlim_m \varprojlim_k \bT_{\xi,\tau}(U_m, \co_E/\varpi_E^k)_{\fm},
\end{equation*}
where $\bT_{\xi,\tau}(U_m,\co_E/\varpi_E^k)_{\fm}$ denotes the $\co_E/\varpi_E^k$-subalgebra of $\End_{\co_E/\varpi_E^k}\big(S_{\xi,\tau}(U_m, \co_E/\varpi_E^k)_{\fm}\big)$ generated by the operators in $\bT^{S_p,\univ}$.
\\

\noindent We denote by $R_{\widetilde{v}}^{\square}$ the maximal reduced and $p$-torsion free quotient of the universal $\co_E$-lifting ring of $\overline{\rho}|_{\Gal_{F_{\widetilde{v}}}}$ ($\cong \overline{r}$). For $v\in S_p\setminus \{\wp\}$, we denote by $R_{\widetilde{v}}^{\square, \xi,\tau}$ for the reduced and $p$-torsion free quotient of $R_{\widetilde{v}}^{\square}$ corresponding to potentially crystalline lifts of weight $\xi$ and inertial type $\tau$. Let $\cS$ denote the global deformation problem as in \cite[\S ~2.4]{CEGGPS}, $R_{\cS}^{\univ}$ the universal deformation ring, and $\rho_{\cS}^{\univ}$ the universal deformation over $R_{\cS}^{\univ}$. Note that one has a natural morphism $R_{\cS}^{\univ} \ra \bT_{\xi,\tau}^{S_p}(U^{\wp}, \co_E)_{\fm}$, in particular, $\widehat{S}_{\xi,\tau}(U^{\wp}, \co_E)_{\fm}$ is naturally equipped with an action of $R_{\cS}^{\univ}$. Following \cite[\S~2.8]{CEGGPS} we put
\begin{eqnarray*}
  R^{\loc}&:=&R_{\widetilde{\wp}}^{\square} \widehat{\otimes} \Big(\widehat{\otimes}_{S_p\setminus \{\wp\}}R_{\widetilde{v}}^{\square, \xi,\tau}\Big)\widehat{\otimes} R_{\widetilde{v_1}}^{\square},\\
  R_{\infty}&:=&R^{\loc}\llbracket x_1,\cdots, x_g\rrbracket,\\
  S_{\infty}&:=&\co_E\llbracket z_1,\cdots, z_{n^2(|S_p|+1)}, y_1,\cdots, y_q\rrbracket,
\end{eqnarray*}
where $q\geq [F^+:\Q]\frac{n(n-1)}{2}$ is a certain integer as in \emph{loc. cit.}, $g=q-[F^+:\Q]\frac{n(n-1)}{2}$, and $x_i$, $y_i$, $z_i$ are formal variables. By \cite[\S~2.8]{CEGGPS}, one can obtain
\begin{enumerate}
  \item a continuous $R_{\infty}$-admissible unitary representation $\Pi_{\infty}$ of $G=\GL_n(L)$ over $E$ together with a $G$-stable and $R_{\infty}$-stable unit ball $\Pi_{\infty}^o\subset \Pi_{\infty}$;
  \item a morphism of local $\co_E$-algebras $S_{\infty}\ra R_{\infty}$ such that $M_{\infty}:= \Hom_{\co_L}(\Pi_{\infty}^o, \co_E)$ is finite projective as $S_{\infty}\llbracket \GL_n(\co_L)\rrbracket$-module;
  \item a closed ideal $\fa$ of $R_{\infty}$, a surjection $R_{\infty}/\fa R_{\infty}\twoheadrightarrow R_{\cS}^{\univ}$ and a  $G   \times R_{\infty}/\fa R_{\infty}$-invariant isomorphism $\Pi_{\infty}[\fa]\cong \widehat{S}_{\xi,\tau}(U^{\wp},E)_{\fm}$, where $R_{\infty}$ acts on $\widehat{S}_{\xi,\tau}(U^{\wp},E)_{\fm}$ via $R_{\infty}/\fa R_{\infty}\twoheadrightarrow  R_{\cS}^{\univ}$.
\end{enumerate}
\subsubsection{Patched eigenvariety}
\noindent We briefly recall the patched eigenvariety of Breuil-Hellmann-Schraen \cite{BHS1}. Indeed, although our input as in \S~\ref{sec: lgln-pbr} is slightly different from that in \cite{BHS1}, one sees easily that all the arguments in \emph{loc. cit. } apply in our case.
\\

\noindent Let $\Pi_{\infty}^{R_{\infty}-\an}$ denote the subspace of the locally $R_{\infty}$-analytic vectors of $\Pi_{\infty}$ (cf. \cite[\S~3.1]{BHS1}), which are locally $\Q_p$-analytic vectors for the action $G\times \Z_p^s$ with respect to one (or any) presentation $\co_E\llbracket \Z_p^s \rrbracket\twoheadrightarrow R_{\infty}$. Applying the Jacquet-Emerton functor (with respect to the upper triangular Borel subgroup $B$), we get a locally $\Q_p$-analytic representation $J_B(\Pi_{\infty}^{R_{\infty}-\an})$ of $T(L)$ equipped with a continuous action of $R_{\infty}$, which is moreover an essentially admissible locally $\Q_p$-analytic representation of $T(L)\times \Z_p^s$ (with respect to  a chosen presentation $\co_E\llbracket \Z_p^s \rrbracket\twoheadrightarrow R_{\infty}$).
Let $\fX_{\infty}:=(\Spf R_{\infty})^{\rig}$, $R_{\infty}^{\rig}:=\co\big(\fX_{\infty}\big)$, and let $\widehat{T}$ denote the character space of $T(L)$ \big(i.e. the rigid space parametrizing continuous characters of $T(L)$\big).  The strong dual $J_{B}\big(\Pi_{\infty}^{R_{\infty}-\an}\big)^{\vee}$ is a coadmissible $R_{\infty}^{\rig}\widehat{\otimes}_E \co(\widehat{T})$-module, which corresponds to a coherent sheaf $\cM_{\infty}$ over $\fX_{\infty}\times \widehat{T}$ such that
\begin{equation*}
   \Gamma\big(\fX_{\infty}\times \widehat{T}, \cM_{\infty}\big) \cong  J_{B}\big(\Pi_{\infty}^{R_{\infty}-\an}\big)^{\vee}.
\end{equation*}
Let $X_{\wp}(\overline{\rho})$ be the support of $\cM_{\infty}$ on $\fX_{\infty}\times \widehat{T}$,  called the \emph{patched eigenvariety}. For $x=(y,\chi)\in \fX_{\infty}\times \widehat{T}$, $x\in X_{\wp}(\overline{\rho})$ if and only if the corresponding eigenspace
\begin{equation*}
J _{B}\big(\Pi_{\infty}^{R_{\infty}-\an}\big)[\fm_y,T(L)=\chi]\neq 0,
\end{equation*}
where $\fm_y$ denotes the maximal ideal of $R_{\infty}[\frac{1}{p}]$ corresponding to $y$. By the same arguments as in  \cite[Cor. 3.12, Thm. 3.19, Cor. 3.20]{BHS1} \cite[Lem. 3.8]{BHS2}, we have
\begin{theorem}\label{thm: patev}
  (1) The rigid space $X_{\wp}(\overline{\rho})$ is reduced and equidimensional of dimension
  \begin{equation*}
 q+n^2(|S_p|+1)+nd_L=g+\frac{n(n-1)}{2}[F^+:\Q]+n^2(|S_p|+1)+nd_L.
  \end{equation*}
 % There exists a covering of $X_{\wp}(\overline{\rho})$ such that....

 \noindent  (2) The coherent sheaf $\cM_{\infty}$ over $X_{\wp}(\overline{\rho})$ is Cohen-Macauley.

 \noindent (3) The set of very classical points (cf. \cite[Def. 3.17]{BHS1}) is Zariski-dense in $X_{\wp}(\overline{\rho})$ and is an accumulation set.
\end{theorem}
\noindent Let $X_{\tri}^{\square}(\overline{r})$ be the trianguline variety associated to $\overline{r}$ (cf. \cite[D\'ef. 2.4]{BHS1}), which is in particular a closed reduced rigid subspace of $\fX^{\square}_{\overline{r}} \times \widehat{T}$, equidimensional of dimension $n^2+\frac{n(n+1)}{2} d_L$, where $\fX_{\overline{r}}^{\square}:=\big(\Spf R_{\overline{r}}^{\square}\big)^{\rig}$. Let $R^{\wp}:=\Big(\widehat{\otimes}_{v\in S_p\setminus \{\wp\}}R_{\widetilde{v}}^{\square, \xi,\tau}\Big)\widehat{\otimes} R_{\widetilde{v_1}}^{\square}$, $R_{\infty}^{\wp}:=R^{\wp}\llbracket x_1,\cdots, x_g\rrbracket$,  $\fX_{\overline{\rho}^{\wp}}^{\square}:=(\Spf R^{\wp})^{\rig}$, and $\bU$ be the open unit ball in $\bA^1$, we have thus $(\Spf R_{\infty})^{\rig}\cong \fX_{\overline{\rho}^{\wp}}^{\square} \times\bU^g \times \fX_{\overline{r}}^{\square} \cong (\Spf R_{\infty}^{\wp})^{\rig}\times \fX_{\overline{r}}^{\square}$. Put
\begin{equation}\label{equ: zeta}
  \zeta:=\unr(q_L^{1-n}) \otimes \cdots \otimes \big(\unr(q_L^{i-n})\prod_{\sigma\in \Sigma_L} \sigma^{i-1}\big) \otimes \cdots \otimes \big(\prod_{\sigma\in \Sigma_L} \sigma^{n-1}\big)
\end{equation}(which is a character of $T(L)$) and denote by $\iota_L$ the following isomorphism
\begin{equation*}
  \iota_L: \widehat{T} \xlongrightarrow{\sim} \widehat{T}, \ \delta \mapsto \delta\zeta.
\end{equation*} Let $\iota_L^{-1}\big(X_{\tri}^{\square}(\overline{r})\big):=X_{\tri}^{\square}(\overline{r}) \times_{\widehat{T}, \iota_L} \widehat{T}$ (which is also a closed reduced rigid subspace of $\fX_{\overline{r}}^{\square}\times \widehat{T}$). By the same argument as in \cite[Thm.3.21]{BHS1}, we have
\begin{theorem}\label{thm: cclg-pwR}
  The natural embedding
  \begin{equation}X_{\wp}(\overline{\rho}) \hooklongrightarrow (\Spf R_{\infty})^{\rig}\times \widehat{T} \cong \fX_{\overline{\rho}^{\wp}}^{\square}\times \bU^g \times \fX_{\overline{r}}^{\square}\times \widehat{T}\end{equation} factors though
  \begin{equation}\label{equ: cclg-pow}X_{\wp}(\overline{\rho})\hooklongrightarrow \fX_{\overline{\rho}^{\wp}}^{\square}\times \bU^g \times \iota_L^{-1}\big(X_{\tri}^{\square}(\overline{r})\big),\end{equation} and induces an isomorphism between $X_{\wp}(\overline{\rho})$ and a union of irreducible components (equipped with the reduced closed rigid subspace structure) of $\fX_{\overline{\rho}^{\wp}}^{\square}\times \bU^g \times \iota_L^{-1}\big(X_{\tri}^{\square}(\overline{r})\big)$.
\end{theorem}

\subsubsection{Non-critical special points}
\noindent
Let $\rho_L: \Gal_L \ra \GL_n(E)$ be a semi-stable representation, and $D:=D_{\rig}(\rho_L)$. Let $\delta_1, \cdots, \delta_n$ be a a trianguline parameter of $D$, suppose that $(D,\delta_i)$ is non-critical special (cf. \S~\ref{sec: lgln-LFM1}). We assume $x_L=(\rho_L, \delta)\in X_{\tri}^{\square}(\overline{r})$, with $\delta=\delta_1\otimes \cdots \otimes \delta_n$. Let $X$ be a union of irreducible components of an open subset of $X_{\tri}^{\square}(\overline{r})$, and suppose $X$ satisfies the accumulation property at $x_L$ (cf. \cite[Def. 2.11]{BHS2}). Consider the composition (see (\ref{equ: kappa}) for $\kappa_L$)
\begin{equation}\label{equ: lgln-trvt}
  T_{X,x_L} \lra \Hom(T(L),E)\xlongrightarrow{\kappa_L} \prod_{i\in \Delta} \Hom(L^{\times},E),
\end{equation}
where $T_{X,x_L}$ denotes the tangent space of $X$ at $x_L$, and the first map is the tangent map of the natural morphism $X\ra \widehat{T}$ at the point $x_L$.
\begin{proposition}\label{prop: lgln-tv}$X$ is smooth at the point $x_L$, and (\ref{equ: lgln-trvt}) factors though a surjective map
\begin{equation*}
T_{X,x_L}\twoheadlongrightarrow \cL(D).
\end{equation*}
\end{proposition}
\begin{proof}The natural  morphism $X\ra \fX_{\overline{\rho}_{\widetilde{\wp}}}^{\square} \cong \fX_{\overline{r}}^{\square}$ induces an exact sequence (see \cite[(4.21)]{BHS2}):
\begin{equation}\label{equ: cclg-xXL}
  0 \ra K(\rho_L) \cap T_{X,x_L} \ra T_{X,x_L} \xrightarrow{f} \Ext_{\Gal_L}^1(\rho_L,\rho_L)\cong \Ext_{(\varphi,\Gamma)}^1(D,D)
\end{equation}
where $K(\rho_L)$ is an $E$-vector space of dimension $n^2-\dim_{E} \End_{\Gal_L}(\rho_L)=n^2-1$ (see \cite[Lem. 4.13]{BHS2}). On the other hand, since $X$ satisfies the accumulation property, by global triangulation theory (\cite[Prop. 4.3.5]{Liu}) and the fact that $(D, \delta_i)$ is non-critical, we know the image of $f$ is contained in the $E$-vector space $F_{D,\cF}(E[\epsilon]/\epsilon^2)$ (note by  \cite[Cor. 2.38]{Na2}, $F_{D,\cF}(E[\epsilon]/\epsilon^2)\subseteq \Ext_{(\varphi,\Gamma)}^1(D,D)$). Together with Proposition \ref{prop: lgln-dtd} (1), we deduce
\begin{equation*}\dim_E T_{X,x_L}\leq n^2 +\frac{n(n+1)}{2} d_L=\dim X.\end{equation*} Thus $X$ is smooth at the point $x$, and $\Ima(f)=F_{D,\cF}(E[\epsilon]/\epsilon^2)$. Moreover, in this case, the first map in (\ref{equ: lgln-trvt}) coincides with the first map in (\ref{equ: kappa}) (by global triangulation theory). The second part of the proposition follows then from Proposition \ref{prop: lgln-dtd} (2).
\end{proof}
\noindent Now  assume moreover there exists $x^{\wp}\in (\Spf R^{\wp}_{\infty})^{\rig}$ such that
\begin{equation*}x=(x^{\wp}, \rho_L, \delta\zeta^{-1}) \in X_{\wp}(\overline{r})\hooklongrightarrow (\Spf R^{\wp}_{\infty})^{\rig}\times \fX_{\overline{r}}^{\square}\times \widehat{T}.\end{equation*}
Consider the following composition
\begin{equation}\label{equ: lgln-paev}
 T_{X_{\wp}(\overline{\rho}),x} \lra \Hom(T(L),E)  \xlongrightarrow{\kappa_L} \prod_{i\in \Delta} \Hom(L^{\times},E),
\end{equation}
where $T_{X_{\wp}(\overline{\rho}),x}$ denotes the tangent space of $X_{\wp}(\overline{\rho})$ at the point $x$ and the first map is the tangent map at $x$ of the natural morphism $X_{\wp}(\overline{\rho}) \lra \widehat{T}$.
\begin{corollary}\label{cor: lgln-patch}
$X_{\wp}(\overline{\rho})$ is smooth at the point $x$, and (\ref{equ: lgln-paev}) factors through a surjective map
\begin{equation}\label{equ: tangL}
T_{X_{\wp}(\overline{\rho}),x}\twoheadlongrightarrow \cL(D).
\end{equation}
\end{corollary}
\begin{proof}
  Let $X$ be the union of  irreducible component of $X_{\wp}(\overline{\rho})$ containing $x$, which, by Theorem \ref{thm: cclg-pwR}, has the form $\cup_i \big(X_i^{\wp}\times \bU^g \times \iota_L^{-1}(X_{i,\wp})\big)$ where $X_{i,\wp}$ is an irreducible component of $X_{\tri}^{\square}(\overline{r})$ containing $x_{\wp}:=(\rho_L,\delta)$, and $X_i^{\wp}$ is an irreducible component of $\fX_{\overline{\rho}^{\wp}}$. By \cite[Thm. 3.3.8]{Kis08} and \cite[Lem. 2.5]{CEGGPS} (see also \cite[Cor. A.2]{CEGGPS}),   $\fX_{\overline{\rho}^{\wp}} \times \bU^g$ is  smooth at $x^{\wp}$, and hence $\{X_i^{\wp}\}_i$ is a singleton. By (the proof of) \cite[Thm. 3.19]{BHS1}, $\cup_i X_{i,\wp}$ satisfies the accumulation property at $x_{\wp}$, and hence by Proposition \ref{prop: lgln-tv}, $\cup_i X_{i,\wp}$ is smooth at $x_{\wp}$ and $\{X_{i,\wp}\}$ is a singleton. So $X_{\wp}(\overline{\rho})$ is smooth at the point $x$. By Proposition \ref{prop: lgln-tv}, (\ref{equ: tangL}) also follows (noting the first map of (\ref{equ: tangL}) factors through the first map in (\ref{equ: lgln-trvt}) with $X=X_{i,\wp}$, $x_L=x_{\wp}$).
\end{proof}
\begin{remark}\label{rem: eig}
The projection (\ref{equ: tangL}) should also hold in genuine eigenvariety case. For example, let $\cE$ be the eigenvariety in \cite{Che11} (for $L=\Q_p$), $x$ be a classical point such that $D_{\rig}(\rho_{x,\Q_p})$ is special non-critical with $N^{n-1}\neq 0$ on $D_{\st}(\rho_{x,\Q_p})$ (where $\rho_{x,\Q_p}$ denotes the $\Gal_{\Q_p}$ representation attached to $x$). By global triangulation theory together with Theorem \ref{thm: CGS}, similarly as in (\ref{equ: tangL}), the natural morphism $\cE \ra \widehat{T}$ induces
\begin{equation}\label{equ: kappaL2}T_{\cE,x}\lra \cL(D).\end{equation} On the other hand, using the same arguments as in the proof of \cite[Thm. 4.8, 4.10]{Che11}, one might show (under multiplicity one hypothesis) that $\cE$ is \'etale over the weight space (i.e. the rigid space parameterizing continuous characters of $(\Z_p^{\times})^n$) at the point $x$. In particular, the map $T_{\cE,x} \ra \Hom(\Z_p^{\times}, E)^n$ (induced by $\cE\ra \widehat{T}$) should be surjective. We have a commutative diagram
\begin{equation*}
  \begin{CD}
    T_{\cE,x} @> (\ref{equ: kappaL2}) >> \cL(D) \\
    @VVV @VVV \\
    \Hom(\Z_p^{\times}, E)^{n} @>>> \Hom(\Z_p^{\times}, E)^{n-1},
  \end{CD}
\end{equation*}
where the right vertical map is the restriction map (thus is bijective, since $\cL(D)_i\cap  \Hom_{\infty}(\Q_p^{\times}, E)=0$ for all $i\in \Delta$), and the map below sends $(\psi_1, \cdots, \psi_n)$ to $(\psi_i-\psi_{i+1})_{i\in \Delta}$ (thus surjective). We see the surjectivity of the left vertical map will imply the surjectivity of  (\ref{equ: kappaL2}).
\end{remark}
\subsection{Local-global compatibility}\label{sec:lg}\noindent Let $\rho_L: \Gal_L \ra \GL_n(E)$ be a continuous  representation and suppose $\rho_L$ appears in the patched eigenvariety $X_{\wp}(\overline{\rho})$, i.e. there exist $x^{\wp}\in (\Spf R_{\infty}^{\wp})^{\rig}$, $\chi\in \widehat{T}$ such that
$(x^{\wp}, \rho_L, \chi)\in X_{\wp}(\overline{r})$.
Let $\fm_y$ be the maximal ideal of $R_{\infty}[1/p]$ corresponding to the point $y:=(x^{\wp}, \rho_L)$ of $(\Spf R_{\infty})^{\rig}$ (e.g. if $\rho_L$ is attached to an automorphic representation of $\widetilde{G}$ with non-zero $U^{\wp}$-fixed vectors). We see that (where the object on the right hand side denotes the subspace of vectors annihilated by $\fm_y$)
\begin{equation*}
  \widehat{\Pi}(\rho_L):=\Pi_{\infty}[\fm_y]
\end{equation*}
is an admissible unitary Banach representation of $\GL_n(L)$, which one might expect to be the right representation (up to multiplicities) corresponding to $\rho_L$ in the $p$-adic local Langlands program. Suppose
\begin{enumerate}
\item $D:=D_{\rig}(\rho_L)$ is trianguline of parameter $(\delta_1, \cdots, \delta_n)$ such that $(D, \delta_i)$ is special non-critical;
\item $\rho_L$ is semi-stable, thus there exists $\alpha\in E^{\times}$ such that $\delta_i \unr(\alpha^{-1}q_L^{1-i})$ is an algebraic character;
\item the monodromy operator $N$ on $D_{\st}(\rho_L)$ satisfies $N^{n-1}\neq 0$.
\end{enumerate}
By \cite[Prop. 4.3]{Colm2}, the assumption 3 implies that the trianguline parameter of $D$ is unique. By Lemma \ref{monofull}, we also have that  $D_i^{i+1}$ is non-crystalline for all $i=1, \cdots, n-1$.
\\

\noindent As in  \S~\ref{sec: lgln-ll}, we put
$\ul{\lambda}:=(\lambda_{1,\sigma}, \cdots, \lambda_{n,\sigma})_{\sigma\in \Sigma_L}\in X_{\Delta}^+$ with $\lambda_{i,\sigma}:=\wt(\delta_i)_{\sigma}+i-1$. Put $\delta:=\delta_1\otimes \cdots \otimes \delta_n$, we have (see (\ref{equ: zeta}) for $\zeta$)
\begin{equation}\label{equ: lgln-chi}
 \delta\zeta =\delta_B \chi_{\ul{\lambda}} \unr(\alpha)\circ \dett
\end{equation}
where $\delta_B=\unr(q_L^{1-n}) \otimes \cdots \otimes \unr(q_L^{2i-n-1}) \otimes \cdots \otimes \unr(q_L^{n-1})$ denotes the modulus character of $B(L)$. By assumption (in the beginning of this section), there exist $\chi: T(L) \ra E^{\times}$ and $x^{\wp}\in (\Spf R_{\infty}^{\wp})^{\rig}$ such that $(x^{\wp}, \rho_L,\chi)\in X_{\wp}(\overline{\rho})$, or equivalently, such that the eigenspace $J_B(\Pi_{\infty}^{R_{\infty}-\an}[\fm_y])[T(L)=\chi]$ is non-zero, with $y=(x^{\wp}, \rho_L)$.
Since $\rho_L$ is non-critical with the unique triangulation, the following lemma follows directly from the global triangulation theory \cite{KPX}\cite{Liu}.
\begin{lemma}\label{lem: lgln-nonc}
For a continuous character $\chi: T(L)\ra E^{\times}$, the eigenspace
\begin{equation*}
  J_B(\Pi_{\infty}^{R_{\infty}-\an}[\fm_y])[T(L)=\chi]\neq 0
\end{equation*}
if and only if $\chi=\delta\zeta$. Equivalently, $(y,\chi)\in X_{\wp}(\overline{\rho})$ if and only if $\chi=\delta\zeta$.
\end{lemma}
\noindent In the following, we let $\chi:=\delta\zeta$, $\chi_{\infty}:=\delta_B \unr(\alpha)\circ \dett$ (thus $\chi=\chi_{\infty}\chi_{\ul{\lambda}}$) and $x:=(y,\chi)$.
\begin{proposition}\label{lg-ind}
The natural map (induced by applying the Jacquet-Emerton functor, recall $\chi\delta_B^{-1}=\chi_{\ul{\lambda}}\unr(\alpha)\circ \dett$ and $\chi \hookrightarrow J_B(i_{\overline{B}}^G(\alpha, \ul{\lambda}))\hookrightarrow J_B(\bI_{\overline{B}}^G(\alpha, \ul{\lambda}))$)
\begin{equation}\label{equ: lgln-bij0}
     \Hom_{\GL_n(L)}\big(\bI_{\overline{B}}^G(\alpha, \ul{\lambda}), \Pi_{\infty}^{R_{\infty}-\an}[\fm_y]\big)\lra \Hom_{T(L)}\big(\chi, J_B\big(\Pi_{\infty}^{R_{\infty}-\an}[\fm_y]\big)\big)
\end{equation}
is bijective.
\end{proposition}
\begin{proof}
(a) By Lemma \ref{lem: lgln-nonc}, \cite[Thm. 4.3]{Br13II} and \cite[Cor. 3.4]{Br13I}, one can deduce any non-zero map $f$ in the left hand side set factors through $\St_n^{\an}(\alpha, \ul{\lambda})$, and induces a non-zero map $\St_n^{\infty}(\alpha, \lambda)\hookrightarrow \Pi_{\infty}^{R_{\infty}-\an}[\fm_y]$. Hence (\ref{equ: lgln-bij0}) is injective (noting $J_B(\St_n^{\infty}(\alpha, \lambda))\cong \chi$).
\\

\noindent
(b) By \cite[Cor. 3.9]{BHS1} and the fact that $M_{\infty}$ is finite projective over $S_{\infty}[[\GL_n(\co_L)]]$, we see $\Pi_{\infty}^{R_{\infty}-\an}$ is a direct summand of $\cC^{\Q_p-\an}\big(\Z_p^s \times \GL_n(\co_L),E\big)$ as $\GL_n(\co_L)$-representations where $s=n^2(|S_p|+1)+q$. Let $\fm \subset S_{\infty}$ be the preimage of $\fm_y$ via the morphism $S_{\infty} \ra R_{\infty}$. Then $V:=\Pi_{\infty}^{R_{\infty}-\an}[\fm]$ is an admissible Banach representation of $\GL_n(L)$ equipped with a continuous action of $R_{\infty}$, satisfying that $V|_{H}\cong \cC^{\Q_p-\an}(H,E)^{\oplus r}$ for certain $r\in \Z_{\geq 1}$, where $H$ is a pro-$p$ uniform compact open subgroup of $\GL_n(\co_L)$. By the same argument as in the proof of \cite[Prop. 6.3.3]{Br16}, we have
\begin{equation}\label{equ: lgln-vanb}
\hH^i\big(\overline{\ub}_{\Sigma_L}, V \otimes_E \eta\big)=0,
\end{equation}
for all $i\in \Z_{\geq 1}$, and any character $\eta: \overline{\ub}_{\Sigma_L} \twoheadrightarrow \ft_{\Sigma_L} \ra E$.
\\

\noindent
(c) We modify the proof of \cite[Thm. 4.8]{BCho} (see also \cite[\S~5.6]{BE}) to obtain the surjectivity in (\ref{equ: lgln-bij0}). Indeed,  our $V$ satisfies the first hypothesis in \cite[Thm. 4.8]{BCho}. Moreover, since $\rho_L$ is non-critical we have that, $(\chi_{\ul{\lambda}},\chi_{\infty})$ is \emph{non-critical with respect to} $V[\fm_y]$ by Lemma \ref{lem: lgln-nonc} in the terminology of \emph{loc. cit.}. We have a similar commutative diagram as in the proof of \cite[Thm. 4.8]{BCho} with $V$ of \emph{loc. cit.} being our $V[\fm_y]$. And one reduces to show the following natural maps are bijective
\begin{multline}\label{equ: lgln-ad0}
  \Hom_{(\ug_{\Sigma_L},B(L))}\Big(M(\ul{\lambda})^{\vee} \otimes_E \cC_c^{\infty}\big(N(L), \chi_{\infty}\delta_B^{-1}\big), V[\fm_y]\Big)\\ \xlongrightarrow{\eta_1} \Hom_{(\ug_{\Sigma_L},B(L))}\Big(L(\ul{\lambda}) \otimes_E \cC_c^{\infty}\big(N(L), \chi_{\infty}\delta_B^{-1}\big), V[\fm_y]\Big)\\ \xlongrightarrow{\eta_2}
  \Hom_{(\ug_{\Sigma_L},B(L))}\Big(M(\ul{\lambda}) \otimes_E \cC_c^{\infty}\big(N(L), \chi_{\infty}\delta_B^{-1}\big), V[\fm_y]\Big)
\end{multline}
where  $M(\ul{\lambda})^{\vee}$ is the dual of $M(\ul{\lambda})$ in the BGG category $\co^{\ub_{\Sigma_L}}$ (thus  $L(\ul{\lambda})$ the unique semi-simple sub of $M(\ul{\lambda})^{\vee}$), the sequence is induced by the natural composition $M(\ul{\lambda})\twoheadrightarrow L(\ul{\lambda}) \hookrightarrow M(\ul{\lambda})^{\vee}$ and where $\cC_c^{\infty}\big(N(L), \chi_{\infty}\delta_B^{-1}\big)$ denotes the space of locally constant $(\chi_{\infty}\delta_B^{-1})$-valued functions on $N(L)$ with compact support. We discuss a little on the topology. For $M\in \co^{\ub_{\Sigma_L}}$, we equip $M$ with the finest locally convex topology, which realizes $M$ as an $E$-vector space of compact type since $M$ is of countable dimension. And $\cC_c^{\infty}\big(N(L), \chi_{\infty}\delta_B^{-1}\big)$, with the topology defined in \cite[\S~3.5]{Em11}, is also an $E$-vector space of compact type. One can check in our case that this topology coincides with the finest locally convex topology on $\cC_c^{\infty}\big(N(L), \chi_{\infty}\delta_B^{-1}\big)$. Finally, we equip $M\otimes_E \cC_c^{\infty}\big(N(L), \chi_{\infty}\delta_B^{-1}\big)$ with the inductive (or equivalently, projective) tensor product topology, which is of compact type and coincides in fact with the finest locally convex topology. Note that all the maps in the sets in (\ref{equ: lgln-ad0}) are continuous. Since $(U,\chi_{\infty})$ is non-critical with respect to $V[\fm_y]$, one can show as in \cite[Prop. 4.9]{BCho} that $\eta_2$ is bijective and $\eta_1$ is injective.
\\

\noindent
(d) Recall that all the irreducible constituents of $M(\ul{\lambda})$ are of form $L(s\cdot \ul{\lambda})$ with $s$ an element in the Weyl group $(S_{n})^{\oplus d_L}$ of $\Res^L_{\Q_p} \GL_n$, and $L(\ul{\lambda})$ has multiplicity one in $M(\ul{\lambda})$. Thus to prove that $\eta_1$ is surjective (and hence to prove the proposition), it is sufficient to prove the following claim.

\noindent \textbf{Claim: } Let $M, M'\in \co^{\ub_{\Sigma_L}}$ such that $M\subset M'$ and $M'/M\cong L(s\cdot \ul{\lambda})$ with $s\neq 1$, then the restriction map
\begin{equation*}
  \Hom_{(\ug_{\Sigma_L},B(L))}\big(M' \otimes_E \cC_c^{\infty}\big(N(L), \chi_{\infty}\delta_B^{-1}\big), V[\fm_y]\big) \lra \Hom_{(\ug_{\Sigma_L},B(L))}\big(M \otimes_E \cC_c^{\infty}\big(N(L), \chi_{\infty}\delta_B^{-1}\big), V[\fm_y]\big)
\end{equation*}
is surjective.

\noindent We  prove the claim. We equip $M' \otimes_E \cC_c^{\infty}\big(N(L), \chi_{\infty}\delta_B^{-1}\big)$ with an $R_{\infty}[1/p]$-action via $R_{\infty}[1/p]\twoheadrightarrow R_{\infty}[1/p]/\fm_y\cong E$.  Given a $(\ug_{\Sigma_L}, B(L))$-equivariant morphism \big(which is also equivariant under the $R_{\infty}[1/p]$-action, where the left object is equipped with an $R_{\infty}[1/p]$-action via $R_{\infty}[1/p]\twoheadrightarrow R_{\infty}[1/p]/\fm_y\cong E$\big)
\begin{equation*}f: M \otimes_E \cC_c^{\infty}\big(N(L), \chi_{\infty}\delta_B^{-1}\big)\longrightarrow  V[\fm_y]\hooklongrightarrow V,\end{equation*} let  $\widetilde{V}$ denote the push-forward of $M' \otimes_E \cC_c^{\infty}\big(N(L), \chi_{\infty}\delta_B^{-1}\big)$ via $f$,
which sits in an exact sequence of $(\ug_{\Sigma_L}, B(L))$-modules
\begin{equation}\label{equ: lgln-adex}
  0 \ra V \ra \widetilde{V} \xrightarrow{\kappa} L(s\cdot \ul{\lambda})\otimes_E \cC_c^{\infty}\big(N(L), \chi_{\infty}\delta_B^{-1}\big) \ra 0,
\end{equation}
and is equipped with a natural locally convex topology and a natural continuous $R_{\infty}[1/p]$-action (satisfying that (\ref{equ: lgln-adex}) is equivariant under the $R_{\infty}[1/p]$-action).
Since $L(s\cdot \ul{\lambda})\otimes_E \cC_c^{\infty}\big(N(L), \chi_{\infty}\delta_B^{-1}\big)$ is equipped with the finest locally convex topology, the extension (\ref{equ: lgln-adex}) is split as topological $E$-vector spaces. In particular, $\widetilde{V}$ is also an $E$-vector space of compact type.

\noindent Using the same argument as in the proof of \cite[Prop. 4.1]{BH2}, we see that (\ref{equ: lgln-adex}) induces an $R_{\infty}[1/p]$-equivariant exact sequence of finite dimensional $E$-vector spaces \footnote{Actually, the only argument of \emph{loc. cit.} that does not directly apply to our case is for the finiteness of the dimension of $J_B(\widetilde{V})[\ft=s\cdot \lambda]\{T(\Q_p)=\chi_{s\cdot \lambda} \chi_{\infty}\}$. However, this also follows from the left  exactness of the Jacquet-Emerton functor, and the fact that both the left and right $E$-vector spaces have finite dimensions. }
\begin{multline*}
0 \lra J_B(V)[\ft=s\cdot \lambda]\{T(\Q_p)=\chi_{s\cdot \lambda} \chi_{\infty}\} \lra J_B(\widetilde{V}) [\ft=s\cdot \lambda]\{T(\Q_p)=\chi_{s\cdot \lambda} \chi_{\infty}\} \\ \lra J_B(\chi_{s\cdot \lambda}\otimes_E \cC_c^{\infty}(N(L), \chi_{\infty}\delta_{B}^{-1})[\ft=s\cdot \lambda]\{T(\Q_p)=\chi_{s\cdot \lambda} \chi_{\infty}\} \lra 0.
\end{multline*}
The sequence is still exact after taking the $\fm_y$-generalized eigenspaces for $R_{\infty}[1/p]$. By Lemma \ref{lem: lgln-nonc}, we easily deduce
\begin{equation*}
J_B(V)[\ft=s\cdot \lambda]\{T(\Q_p)=\chi_{s\cdot \lambda} \chi_{\infty}, \fm_y\}=0.
\end{equation*}
Hence we obtain an isomorphism (see \cite[Lem. 3.5.2]{Em11} for the last isomorphism)
\begin{multline}\label{equ: splitJac}
J_B(\widetilde{V}) [\ft=s\cdot \lambda]\{T(\Q_p)=\chi_{s\cdot \lambda} \chi_{\infty}, \fm_y\}\\ \xlongrightarrow{\sim} J_B(\chi_{s\cdot \lambda}\otimes_E \cC_c^{\infty}(N(L), \chi_{\infty}\delta_{B}^{-1})[\ft=s\cdot \lambda]\{T(\Q_p)=\chi_{s\cdot \lambda} \chi_{\infty}, \fm_y\} \cong \chi_{s\cdot \ul{\lambda}} \chi_{\infty},
\end{multline}
and we see the left hand side is actually annihilated by $\fm_y$.
By \cite[Thm. 3.5.6]{Em11}, the inverse of (\ref{equ: splitJac}) induces a $B(L)$-invariant map
\begin{equation*}
  \chi_{s\cdot \ul{\lambda}} \otimes_E \cC^{\infty}_c(N(L), \chi_{\infty}\delta_B^{-1}) \lra \widetilde{V}[\fm_y],
\end{equation*}
and hence a $(\ug_{\Sigma_L}, B(L))$-invariant map $ M(s\cdot \ul{\lambda})\otimes_E \cC^{\infty}_c(N(L), \chi_{\infty}\delta_B^{-1}) \ra \widetilde{V}[\fm_y]$.
However, this map has to factor through
\begin{equation}\label{equ: lgln-adsec}L(s\cdot \ul{\lambda})\otimes_E  \cC^{\infty}_c(N(L), \chi_{\infty}\delta_B^{-1}) \lra \widetilde{V}[\fm_y].\end{equation}
Indeed, otherwise, there exists $\lambda'\neq s\cdot \lambda$ and a non-zero $(\ug_{\Sigma_L}, B(L))$-equivariant map
\begin{equation}\label{equ: lambda'}
  L(\lambda') \otimes_E \cC^{\infty}_c(N(L), \chi_{\infty}\delta_B^{-1}) \lra \widetilde{V}[\fm_y],
\end{equation}
which has image in $V[\fm_y]$ by the fact (which follows easily from \cite[Lem. 3.5.2, Thm. 3.5.6]{Em11})
\begin{equation*}
  \Hom_{(\ug_{\Sigma_L}, B(L))}\big(L(\lambda')\otimes_E \cC^{\infty}_c(N(L),\chi_{\infty}\delta_B^{-1}), L(s\cdot \lambda) \otimes_E \cC^{\infty}_c(N(L), \chi_{\infty}\delta_B^{-1})\big)=0.
\end{equation*}
However, applying the Jacquet-Emerton functor to a non-zero map as in (\ref{equ: lambda'}) with image in $V[\fm_y]$ and using \cite[Lem. 3.5.2]{Em11}, we easily obtain a contradiction against Lemma \ref{lem: lgln-nonc}.

\noindent One can check (e.g. by \cite[Thm. 3.5.6]{Em11}) that (\ref{equ: lgln-adsec}) gives a section of $\kappa$ in (\ref{equ: lgln-adex}) and induces a projection $\widetilde{V}[\fm_y] \twoheadrightarrow V[\fm_y]$. The composition
\begin{equation*}
  M'\otimes_E \cC^{\infty}_c(N(L), \chi_{\infty}\delta_B^{-1}) \lra \widetilde{V}[\fm_y] \twoheadlongrightarrow V[\fm_y]
\end{equation*}
gives thus the desired lifting of $f$. This concludes the proof of the claim (hence of the proposition).
\end{proof}
\begin{remark}\label{rem: Ressubq}
By (a) of the proof, we see the set on the left hand side of (\ref{equ: lgln-bij0}) stays unchanged if  $\bI_{\overline{B}}^G(\alpha, \ul{\lambda})$ is replaced by $\St_n^{\an}(\alpha, \ul{\lambda})$ or any subrepresentation of $\St_n^{\an}(\alpha, \ul{\lambda})$ containing $\St_n^{\infty}(\alpha, \ul{\lambda})$.
\end{remark}
\noindent Let $\fI_y\subseteq \fm_y$ be a closed ideal of $R_{\infty}[1/p]$ such that $\dim_E (R_{\infty}[1/p]/\fI_y) < +\infty$ and that $\fm_y$ is the unique closed maximal ideal containing $\fI_y$ (e.g. $\fI_y=\fm_y^k$).
\begin{corollary}\label{cor: lgln-bij3}
  Let $\St_n^{\infty}(\alpha, \ul{\lambda})\subset W\subset \St_n^{\an}(\alpha, \ul{\lambda})$, and suppose we have a morphism $f: W\ra \Pi_{\infty}^{R_{\infty}-\an}[\fI_y]$. If the restriction of $f$ on  $\St_n^{\infty}(\alpha, \ul{\lambda})$ has image in $\Pi_{\infty}^{R_{\infty}-\an}[\fm_y]$, then $f$ has image in $\Pi_{\infty}^{R_{\infty}-\an}[\fm_y]$.
\end{corollary}
\begin{proof} Applying the Jacquet-Emerton functor to $f|_{\St_n^{\infty}(\alpha, \ul{\lambda})}$, we obtain a morphism $\chi \ra J_B(\Pi_{\infty}^{R_{\infty}-\an}[\fm_y])$ which, by (\ref{equ: lgln-bij0}) and Remark \ref{rem: Ressubq}, induces a morphism $f': W \ra \Pi_{\infty}^{R_{\infty}-\an}[\fm_y]$. Moreover $f'|_{\St_n^{\infty}(\alpha, \ul{\lambda})}=f|_{\St_n^{\infty}(\alpha, \ul{\lambda})}$. Consider $f'-f: W\ra \Pi_{\infty}^{R_{\infty}-\an}[\fI_y]$. By Lemma \ref{lem: lgln-nonc} (which also holds if the $\fm_y$-eigenspace is replaced by the $\fm_y$-generalized eigenspace) and \cite[Cor. 3.4]{Br13I}, we deduce $f'-f=0$. The corollary follows.
\end{proof}
\noindent The following theorem is the main result of the paper.
\begin{theorem}\label{thm: lgln-main}
(1) The following restriction map is bijective
\begin{equation}\label{equ: lgln-main}
    \Hom_{\GL_n(L)}\big(\Sigma(\alpha,\ul{\lambda},\cL(\rho_L)), \Pi_{\infty}^{R_{\infty}-\an}[\fm_y]\big) \lra \Hom_{\GL_n(L)}\big(\St_n^{\infty}(\alpha, \ul{\lambda}), \Pi_{\infty}^{R_{\infty}-\an}[\fm_y]\big).
\end{equation}

\noindent
(2) Let $0\neq \psi\in \Hom(L^{\times},E)$ and $i\in \Delta$, an injection
\begin{equation*}
f: \St_n^{\infty}(\alpha, \ul{\lambda})\hooklongrightarrow \Pi_{\infty}^{R_{\infty}-\an}[\fm_y]   \end{equation*}
can extend to an injection $\Sigma_{i}(\alpha, \ul{\lambda}, \psi)\hookrightarrow \Pi_{\infty}^{R_{\infty}-\an}[\fm_y]$ if and only if $\psi\in \cL(\rho_L)_i$.
\end{theorem}
\noindent The rest of the section is to prove Theorem \ref{thm: lgln-main}, and we use the strategy of \cite{Ding3}. First note that the injectivity of (\ref{equ: lgln-main}) follows from Lemma \ref{lem: lgln-nonc} and \cite[Cor. 3.4]{Br13I} (as in (a) of the proof of Proposition \ref{lg-ind}).
\\

\noindent Theorem \ref{thm: lgln-main} (2) follows from (1). Indeed, the ``if" part is a direct consequence of (1). Conversely, suppose there exists $\Sigma_{i}(\alpha, \ul{\lambda}, \psi)\hookrightarrow \Pi_{\infty}^{R_{\infty}-\an}[\fm_y]$ with $\psi\notin \cL(\rho_L)_i$. By (1), we know $\Sigma_i(\alpha, \ul{\lambda}, \cL(\rho_L)_i)\hookrightarrow \Pi_{\infty}^{R_{\infty}-\an}[\fm_y]$, hence we obtain an injection
\begin{equation*}
\Sigma_i(\alpha, \ul{\lambda}, \cL(\rho_L)_i)\oplus_{\Sigma_i(\alpha, \ul{\lambda})} \Sigma_{i}(\alpha, \ul{\lambda}, \psi) \hooklongrightarrow \Pi_{\infty}^{R_{\infty}-\an}[\fm_y].
\end{equation*}
However, $\psi \notin \cL(\rho_L)_i$ implies $\cL(\rho_L)_i+ E\psi =\Hom(L^{\times}, E)$. In particular, let $0\neq \psi_{\infty}\in \Hom_{\infty}(L^{\times},E)$, we see $\psi_{\infty}\in \cL(\rho_L)_i+E\psi$ and hence
\begin{equation*}
  \Sigma_i(\alpha, \ul{\lambda}, \psi_{\infty}) \hooklongrightarrow\Sigma_i(\alpha, \ul{\lambda}, \cL(\rho_L)_i)\oplus_{\Sigma_i(\alpha, \ul{\lambda})} \Sigma_{i}(\alpha, \ul{\lambda}, \psi) \hooklongrightarrow \Pi_{\infty}^{R_{\infty}-\an}[\fm_y].
\end{equation*}
However, by Remark \ref{rem: lgln-ext3} (ii), $\Sigma_i(\alpha, \ul{\lambda}, \psi_{\infty})$ contains $V\otimes_E L(\ul{\lambda})$ where $V$ is a smooth extension of $v_{\overline{P}_i}^{\infty}$ by $\St_n^{\infty}$. By applying the (left exact) Jacquet-Emerton functor to $V\otimes_E L(\ul{\lambda})\hookrightarrow \Pi_{\infty}^{R_{\infty}-\an}[\fm_y]$ and using \cite[Prop. 4.3.6]{Em11}, we get a contradiction with Lemma \ref{lem: lgln-nonc}. (2) follows.
 \\

\noindent It rests to  show (\ref{equ: lgln-main}) is surjective. By definition (cf. \S ~\ref{sec: lgln-ll}), it suffices to show that for any $i\in \Delta$, $\psi\in \cL(\rho_L)_i$, the following restriction map is surjective
\begin{equation}\label{equ: lgln-rest}
  \Hom_{\GL_n(L)}\big(\Sigma_i(\alpha, \ul{\lambda}, \psi), \Pi_{\infty}^{R_{\infty}-\an}[\fm_y]\big) \lra \Hom_{\GL_n(L)}\big(\St_n^{\infty}(\alpha, \ul{\lambda}), \Pi_{\infty}^{R_{\infty}-\an}[\fm_y]\big).
\end{equation}
Let $t: \Spec E[\epsilon]/\epsilon^2 \ra X_{\wp}(\overline{\rho})$ be an element in $T_{X_{\wp}(\overline{\rho}),x}$, such that the $i$-th factor of the image of $t$ under (\ref{equ: lgln-paev}) equals $\psi$, and the $j$-th factors for all $j\neq i$ are zero (where the existence of $t$ follows from Corollary \ref{cor: lgln-patch}). Let $\Psi\in \Hom(T(L),E)$ be the image of $t$ via the first map in (\ref{equ: lgln-paev}) \big(thus $\Psi|_{Z_{j}(L)}=0$ for $j\neq i$\big).
Since $X_{\wp}(\overline{\rho})$ is smooth at $x$, by Theorem \ref{thm: patev} (2) and \cite[Cor. 17.3.5 (i)]{EGAiv1}, we deduce $\cM_{\infty}$ is locally free in a certain neighborhood of $x$.
Let $r=\dim_E x^* \cM_{\infty}$. We have thus the following facts:
\begin{enumerate}
\item $(t^* \cM_{\infty})^{\vee}\cong \big(\chi(1+\Psi \epsilon)\big)^{\oplus r}=:\widetilde{\chi}^{\oplus r}$ \big(resp. $(x^* \cM_{\infty})^{\vee}\cong \chi^{\oplus r}$\big)as $T(L)$-representations (recall $\chi=\delta\zeta$, see (\ref{equ: lgln-chi})),
\item  $(t^* \cM_{\infty})^{\vee}\subseteq J_B\big(\Pi_{\infty}^{R_{\infty}-\an}[\fI_t]\big)$ where $\fI_t$ denotes the kernel of the morphism $R_{\infty}[1/p] \ra E[\epsilon]/\epsilon^2$ induced by $t$,
\item there are natural $T(L)$-equivariant injections
\begin{equation}\label{equ: lgln-injf}(x^* \cM_{\infty})^{\vee}\hooklongrightarrow (t^* \cM_{\infty})^{\vee} \hooklongrightarrow J_B(\Pi_{\infty}^{R_{\infty}-\an}[\fI_t]).\end{equation}
\end{enumerate}

\begin{lemma}
  The morphisms of $T(L)$-representations: $(x^*\cM_{\infty})^{\vee} \hookrightarrow J_B(\Pi_{\infty}^{R_{\infty}-\an}[\fm_y])$ and $(t^* \cM_{\infty})^{\vee} \hookrightarrow
  J_B(\Pi_{\infty}^{R_{\infty}-\an}[\fI_t])$ are balanced (see \cite[Def. 0.8]{Em2} \cite[Def. 5.17]{Em07}).
\end{lemma}
\begin{proof}
By \cite[Def. 5.17]{Em07}, it is sufficient to prove the kernel of the $(\text{U}(\ug_{\Sigma_L}),B(L))$-equivariant map (induced by (\ref{equ: lgln-injf}))
\begin{equation}\label{equ: bal1a}
\text{U}(\ug_{\Sigma_L}) \otimes_{\text{U}(\ub_{\Sigma_L})}  \cC^{\sm}_c\big(N(L),(x^* \cM_{\infty})^{\vee}\big) \lra \Pi_{\infty}^{R_{\infty}-\an}[\fm_y]
\end{equation}
\begin{equation}\label{equ: bal2a}
  \text{\Big(resp. }\text{U}(\ug_{\Sigma_L}) \otimes_{\text{U}(\ub_{\Sigma_L})} \cC^{\sm}_c\big(N(L),(t^* \cM_{\infty})^{\vee}\big) \lra \Pi_{\infty}^{R_{\infty}-\an}[\fI_t]\text{\Big)},
\end{equation}
contains the kernel of the $(\text{U}(\ug_{\Sigma_L}),B(L))$-equivariant map
\begin{equation}\label{equ: bal1}
\text{U}(\ug_{\Sigma_L}) \otimes_{\text{U}(\ub_{\Sigma_L})}\cC^{\sm}_c\big(N(L),(x^*   \cM_{\infty})^{\vee}\big) \lra  \cC^{\lp}_c\big(N(L),(x^* \cM_{\infty})^{\vee}\big)
\end{equation}
\begin{equation}\label{equ: bal2}
\text{\Big(resp. } \text{U}(\ug_{\Sigma_L}) \otimes_{\text{U}(\ub_{\Sigma_L})} \cC^{\sm}_c\big(N(L),(t^* \cM_{\infty})^{\vee}\big) \lra  \cC^{\lp}_c\big(N(L),(t^* \cM_{\infty})^{\vee}\big)\text{\Big)},
\end{equation}
where we refer to \cite[Def. 2.5.21]{Em2} for $\cC^{\lp}_c(N(L),-)$. Moreover, as in the discussion above \cite[Lem. 2.8.8]{Em2}, the map (\ref{equ: bal1}) (resp. (\ref{equ: bal2})) can be induced from
\begin{equation*}
\text{U}(\ug_{\Sigma_L}) \otimes_{\text{U}(\ub_{\Sigma_L})}(x^*   \cM_{\infty})^{\vee} \lra  \cC^{\Q_p-\pol}\big(N(L),(x^* \cM_{\infty})^{\vee}\big)
\end{equation*}
\begin{equation*}
\text{\Big(resp. } \text{U}(\ug_{\Sigma_L}) \otimes_{\text{U}(\ub_{\Sigma_L})} (t^* \cM_{\infty})^{\vee} \lra  \cC^{\Q_p-\pol}\big(N(L),(t^* \cM_{\infty})^{\vee}\text{\Big)},
\end{equation*}
by tensoring with $\cC^{\infty}_c(N(L),E)$, where $\cC^{\Q_p-\pol}(N(L),U):=\cC^{\Q_p-\pol}(N(L),E)\otimes_E U$ with $\cC^{\Q_p-\pol}(N(L),E)$ the ring of the algebraic $E$-valued functions on the affine algebraic group $\Res^L_{\Q_p} N$, and $U$ a locally analytic representation of $T(L)$ (cf. \cite[\S~2.5]{Em2}).
The natural exact sequence $0 \ra \chi \ra \widetilde{\chi}\ra \chi\ra 0$ induces a $(\ug_{\Sigma_L}, B(L))$-equivariant commutative diagram
\begin{equation*}
  \begin{CD}
  0@>>>  \text{U}(\ug_{\Sigma_L}) \otimes_{\text{U}(\ub_{\Sigma_L})} \chi @>>>\text{U}(\ug_{\Sigma_L}) \otimes_{\text{U}(\ub_{\Sigma_L})} \widetilde{\chi} @>>>  \text{U}(\ug_{\Sigma_L}) \otimes_{\text{U}(\ub_{\Sigma_L})} \chi @>>> 0\\
  @. @V j_1 VV @V j_2 VV @VVV @. \\
  0 @>>> \cC^{\Q_p-\pol}(N(L),\chi)@>>>\cC^{\Q_p-\pol}\big(N(L),\widetilde{\chi}\big) @>>> \cC^{\Q_p-\pol}(N(L),\chi) @>>> 0
  \end{CD}.
\end{equation*}
By \cite[Lem. 2.5.8]{Em2} and the proof of \cite[Thm. 4.3]{Br13II}, we have $\cC^{\Q_p-\pol}(N(L),\chi)\cong M(\ul{\lambda})^{\vee}\otimes_E \chi_{\infty}$. By the structure of (the dual of) Verma module, we see that any irreducible constituent of $\Ker j_1$ and $\Ker j_2$ is isomorphic to $L(s\cdot \ul{\lambda})\otimes_E \chi_{\infty}$ for $s\neq 1$. Thus the kernel $W$ of (\ref{equ: bal1}) (resp.  of (\ref{equ: bal2})) admits a $(\ug_{\Sigma_L},B(L))$-equivariant filtration $0=\cF^0W \subsetneq \cF^1 W \subsetneq \cdots \subsetneq \cF^m W=W$ (for certain $m$) such that $\cF^i W/\cF^{i+1} W \cong \cC_c^{\infty}(N(L),E) \otimes_E L(s\cdot \ul{\lambda})\otimes_E \chi_{\infty}$ for certain $s\neq 1$.  Using the same argument as in the proof of \cite[Prop. 4.9]{BCho}, we have by Lemma \ref{lem: lgln-nonc}:
\begin{equation*}
  \Hom_{(\ug_{\Sigma_L},B(L))}\big(L(s'\cdot \ul{\lambda})\otimes_E \cC^{\infty}_c\big(N(L),\chi_{\infty}\delta_B^{-1}\big), \Pi_{\infty}^{R_{\infty}-\an}[\fI_y]\big)=0
\end{equation*}
for any $s'\neq 1$,  where $\fI_y=\fm_y$ or $\fI_y=\fI_t$. Together with the above filtration on $W$, we deduce (inductively) that  (\ref{equ: bal1a}) (resp. (\ref{equ: bal2a})) restricts to $0$ on $W$. The lemma follows.
\end{proof}
\noindent By \cite[Thm. 0.13]{Em2}, the injections in (\ref{equ: lgln-injf}) induce thus
\begin{equation}\label{equ: lgln-keymap}
 I_{\overline{B}}^G \big(\chi\delta_B^{-1}\big)^{\oplus r} \hooklongrightarrow I_{\overline{B}}^G \big(\widetilde{\chi}\delta_B^{-1}\big)^{\oplus r} \longrightarrow \Pi_{\infty}[\fI_t].
\end{equation}
where as in \emph{loc. cit.}, for a locally $\Q_p$-analytic representation $V$ of $T(L)$, $I_{\overline{B}}^G(V)$ denotes the closed $G$-subrepresentation of $(\Ind_{\overline{B}(L)}^G V)^{\Q_p-\an}$ generated by $V$ via the composition
\begin{equation*}
  V \lra J_B\big((\Ind_{\overline{B}(L)}^G V)^{\Q_p-\an}\big) \lra (\Ind_{\overline{B}(L)}^G V)^{\Q_p-\an},
\end{equation*}
where the second map denotes the canonical lifting with respect to a compact open subgroup $N_o$ of $N(L)$ (cf.  \cite[(0.10)]{Em11}) (note that by \cite[Lem. 2.8.3]{Em2}, $I_{\overline{B}}^G(V)$ is independent of the choice of $N_o$). We have the following easy lemma.
\begin{lemma}\label{lem: lgln-ibg}(1) $I_{\overline{B}}^G (\chi \delta_B^{-1})\cong i_{\overline{B}}^G(\alpha, \ul{\lambda})$.

\noindent (2) The natural exact sequence $0\ra \chi \ra \widetilde{\chi} \ra \chi \ra 0$ induces a sequence (which is not exact in general)
\begin{equation*}
  I_{\overline{B}}^{G}(\chi\delta_B^{-1}) \hooklongrightarrow I_{\overline{B}}^G(\widetilde{\chi}\delta_B^{-1}) \twoheadlongrightarrow I_{\overline{B}}^G(\chi \delta_B^{-1}).
\end{equation*}
\end{lemma}
\begin{proof}
(1) follows easily from \cite[Prop. 2.8.10]{Em2}. (2) follows by definition.
\end{proof}
\noindent In particular, we see $v_{\overline{P}_i}^{\infty}(\alpha, \ul{\lambda})$ has multiplicity $2$ as irreducible constituent in $I_{\overline{B}}^G(\widetilde{\chi}\delta_B^{-1})$. Put
\begin{eqnarray*}U&:=&I_{\overline{B}}^G(\widetilde{\chi}\delta_B^{-1}) \cap \big(\sum_{\overline{P}\supsetneq \overline{B}} \bI_{\overline{P}}^G(\alpha, \ul{\lambda})\big), \\
W&:=&I_{\overline{B}}^G(\widetilde{\chi}\delta_B^{-1})\cap \bI_{\overline{B}}^G(\alpha, \ul{\lambda}),
\end{eqnarray*}where the intersection is taken inside $(\Ind_{\overline{B}(L)}^{G} \widetilde{\chi}\delta_B^{-1})^{\Q_p-\an}$ (noting $\bI_{\overline{B}}^G(\alpha, \ul{\lambda}) \hookrightarrow (\Ind_{\overline{B}(L)}^{G} \widetilde{\chi}\delta_B^{-1})^{\Q_p-\an}$). Thus $\tilde{\Sigma}(\alpha, \ul{\lambda})' :=W/U$ is a subrepresentation of $\St_n^{\an}(\alpha, \ul{\lambda})$, and $I_{\overline{B}}^G(\widetilde{\chi}\delta_B^{-1})/U$ is an extension of $i_{\overline{B}}^G(\alpha, \ul{\lambda})$ by $\tilde{\Sigma}(\alpha, \ul{\lambda})'$. By Lemma \ref{lem: lgln-nonc} (as in (a) of the proof of Proposition \ref{lg-ind}), we can deduce that any irreducible constituent of $U$ can not appear in the socle of $\Pi_{\infty}^{R_{\infty}-\an}[\fI_t]$. So (\ref{equ: lgln-keymap}) induces
\begin{equation}\label{equ: lgln-keymap2}
  \St_n^{\infty}(\alpha, \ul{\lambda})^{\oplus r}\hooklongrightarrow \big(I_{\overline{B}}^G (\widetilde{\chi}\delta_B^{-1})/U\big)^{\oplus r} \lra \Pi_{\infty}^{R_{\infty}-\an}[\fI_t].
\end{equation}The composition is injective and has image in $\Pi_{\infty}^{R_{\infty}-\an}[\fm_y]$ (since $(x^* \cM_{\infty})^{\vee}$ has image in $J_B(\Pi_{\infty}^{R_{\infty}-\an}[\fm_y])$ via (\ref{equ: lgln-injf})). Moreover, any map in the right hand set of (\ref{equ: lgln-main}) is in fact induced by this composition (e.g. by the bijection in (\ref{equ: lgln-bij0}) with $\bI_{\overline{B}}^G(\alpha,\ul{\lambda})$ replaced by $\St_n^{\infty}(\alpha, \ul{\lambda})$).
\\

\noindent
Denote by $\Sigma_i(\alpha, \ul{\lambda})':=\tilde{\Sigma}(\alpha, \ul{\lambda})' \cap \Sigma_i(\alpha, \ul{\lambda})$. By Corollary \ref{cor: lgln-bij3}, the  composition
\begin{equation}\label{equ: compW}
  (\Sigma_i(\alpha, \ul{\lambda})')^{\oplus r} \hooklongrightarrow \big(I_{\overline{B}}^G (\widetilde{\chi}\delta_B^{-1})/U\big)^{\oplus r} \lra \Pi_{\infty}^{R_{\infty}-\an}[\fI_t]
\end{equation}
also has image in $\Pi_{\infty}^{R_{\infty}-\an}[\fm_y]$, and is also injective (since $\soc_{G} \Sigma_i(\alpha, \ul{\lambda})'\cong \St_n^{\infty}(\alpha, \ul{\lambda})$, by Proposition \ref{soc}). By Remark \ref{rem: Ressubq}, we deduce (\ref{equ: compW}) can uniquely extend to an injection
\begin{equation}\label{equ: compnoL}\Sigma_i(\alpha, \ul{\lambda})^{\oplus r} \lra \Pi_{\infty}^{R_{\infty}-\an}[\fm_y] \hooklongrightarrow \Pi_{\infty}^{R_{\infty}-\an}[\fI_t].\end{equation} Putting (\ref{equ: compW}) (\ref{equ: compnoL}) together, we obtain \big(where the composition has image in $\Pi_{\infty}^{R_{\infty}-\an}[\fm_y]$\big)
\begin{equation}\label{equ: lgL1}
  \Sigma_i(\alpha, \ul{\lambda})^{\oplus r}\hooklongrightarrow \Big( \big(I_{\overline{B}}^G (\widetilde{\chi}\delta_B^{-1})/U\big)\oplus_{\Sigma_i(\alpha, \ul{\lambda})'} \Sigma_i(\alpha, \ul{\lambda})\Big)^{\oplus r} \lra \Pi_{\infty}^{R_{\infty}-\an}[\fI_t].
\end{equation}
We prove that $\Sigma_i(\alpha, \ul{\lambda}, \psi)$ is a subrepresentation of $\big(I_{\overline{B}}^G (\widetilde{\chi}\delta_B^{-1})/U\big)\oplus_{\Sigma_i(\alpha, \ul{\lambda})'} \Sigma_i(\alpha, \ul{\lambda})$. Let $V$ be the pull-back of $I_{\overline{B}}^G (\widetilde{\chi}\delta_B^{-1})/U$ via the injection $i_{\overline{P}_i}^G(\alpha, \ul{\lambda}) \hookrightarrow i_{\overline{B}}^G(\alpha, \ul{\lambda})$ which is thus isomorphic to an extension of $i_{\overline{P}_i}^{G}(\alpha,\ul{\lambda})$ by $\tilde{\Sigma}(\alpha, \ul{\lambda})'$. We have a commutative diagram
\begin{equation}\label{commrevi}
  \begin{CD}
 0 @>>> \tilde{\Sigma}(\alpha, \ul{\lambda})' @>>> V @>>> i_{\overline{P}_i}^G(\alpha, \ul{\lambda}) @>>> 0 \\
 @. @VVV @VVV @| @. \\
    0 @>>> \St_n^{\an}(\alpha, \ul{\lambda}) @>>> \sE_{i}^{\emptyset}(\alpha, \ul{\lambda})/\tilde{U} @>>> i_{\overline{P}_i}^G(\alpha, \ul{\lambda}) @>>> 0
  \end{CD}
\end{equation}
where $\tilde{U}:=\sum_{\overline{P}\supsetneq \overline{B}} \bI_{\overline{P}}^G(\alpha, \ul{\lambda})$. By Theorem \ref{cor: lgln-key} (see also Remark \ref{rem: lgln-ext3} (ii)), the pull-back of the bottom exact sequence via
\begin{equation}\label{injabc}w_{\overline{P}_i}^{\infty}(\alpha,\ul{\lambda}):=w_{\overline{P}_i}^{\infty}(\ul{\lambda})\otimes_E \unr(\alpha)\circ \dett\hooklongrightarrow i_{\overline{P}_i}^G(\alpha, \ul{\lambda})\end{equation}
is split. Using
\begin{equation*}\Ext^1_G(w_{\overline{P}_i}^{\infty}(\alpha,\ul{\lambda}), \tilde{\Sigma}(\alpha, \ul{\lambda})')\hooklongrightarrow \Ext^1_G(w_{\overline{P}_i}^{\infty}(\alpha,\ul{\lambda}),\St_n^{\an}(\alpha,\ul{\lambda}))\end{equation*} \big(by $\Hom_G(w_{\overline{P}_i}^{\infty}(\alpha,\ul{\lambda}), \St_n^{\an}(\alpha,\ul{\lambda})/\tilde{\Sigma}(\alpha, \ul{\lambda})')=0$\big), we deduce that the pull-back of the top exact sequence of (\ref{commrevi}) via (\ref{injabc}) is also split. Thus $I_{\overline{B}}^G (\widetilde{\chi}\delta_B^{-1})/U$ has a subrepresentaiton $\tilde{\Sigma}_i(\alpha, \ul{\lambda}, \psi)'$, which is isomorphic to an extension of $v_{\overline{P}_i}^{\infty}(\alpha,\ul{\lambda})$ by $\tilde{\Sigma}(\alpha, \ul{\lambda})'$.
Moreover, by similar arguments as in Proposition \ref{prop: iso-simp}, Lemma \ref{lem: ext1} (2), we have an isomorphism
\begin{equation*}
  \Ext^1_G(v_{\overline{P}_i}^{\infty}(\alpha,\ul{\lambda}), \Sigma_i(\alpha, \ul{\lambda})') \xlongrightarrow{\sim}  \Ext^1_G(v_{\overline{P}_i}^{\infty}(\alpha,\ul{\lambda}), \tilde{\Sigma}(\alpha, \ul{\lambda})').
\end{equation*}
We deduce then $\tilde{\Sigma}_i(\alpha,\ul{\lambda}, \psi)'$ comes by push-forward from an extension $\Sigma_i(\alpha,\ul{\lambda}, \psi)'$ of $v_{\overline{P}_i}^{\infty}(\alpha, \ul{\lambda})$ by $\Sigma_i(\alpha,\ul{\lambda})'$. It is clear that push-forward of $\Sigma_i(\alpha, \ul{\lambda}, \psi)'$ via $\Sigma_i(\alpha,\ul{\lambda})'\hookrightarrow \Sigma_i(\alpha,\ul{\lambda})$ is isomorphic to $\Sigma_i(\alpha,\ul{\lambda}, \psi)$ (e.g. using the fact that the both are subrepresentations of $\tilde{\Sigma}_i(\alpha,\ul{\lambda}, \psi)$), and is a subrepresentation of $ \big(I_{\overline{B}}^G (\widetilde{\chi}\delta_B^{-1})/U\big)\oplus_{\Sigma_i(\alpha, \ul{\lambda})'} \Sigma_i(\alpha, \ul{\lambda})$.
\\

\noindent The composition in (\ref{equ: lgL1}) induces then
\begin{equation*}
   \Sigma_i(\alpha, \ul{\lambda})^{\oplus r}\hooklongrightarrow \Sigma_i(\alpha,\ul{\lambda}, \psi)^{\oplus r} \lra \Pi_{\infty}^{R_{\infty}-\an}[\fI_t].
\end{equation*}
We know the composition has image in $\Pi_{\infty}^{R_{\infty}-\an}[\fm_y]$. We show the image of the second morphism is also contained in $ \Pi_{\infty}^{R_{\infty}-\an}[\fm_y] $. For any $a\in \fm_y$, the composition (which is $\GL_n(L)$-equivariant)
\begin{equation}\label{equ: lgln-keymap4}
\Sigma_i(\alpha,\ul{\lambda}, \psi)^{\oplus r} \lra \Pi_{\infty}^{R_{\infty}-\an}[\fI_t] \xlongrightarrow{v\mapsto av} \Pi_{\infty}^{R_{\infty}-\an}[\fI_t],
\end{equation}
factors through $\Sigma_i(\alpha,\ul{\lambda}, \psi)^{\oplus r}/ \Sigma_i(\alpha, \ul{\lambda})^{\oplus r}\cong v_{\overline{P}_i}^{\infty}(\alpha, \ul{\lambda})^{\oplus r}$. So (\ref{equ: lgln-keymap4}) has to be zero, since by Lemma \ref{lem: lgln-nonc}, $v_{\overline{P}_{i}}^{\infty}(\alpha, \ul{\lambda})$ is not a subrepresentation of $\Pi_{\infty}^{R_{\infty}-\an}[\fm_y]$ (hence of $\Pi_{\infty}^{R_{\infty}-\an}[\fI_t]$).

\noindent In summary, we obtain morphisms
\begin{equation}\label{equ: lgln-keymap1}
  \St_n^{\infty}(\alpha, \ul{\lambda})^{\oplus r}\hooklongrightarrow \Sigma_{i}(\alpha, \ul{\lambda}, \psi)^{\oplus r} \lra  \Pi_{\infty}^{R_{\infty}-\an}[\fm_y] .
\end{equation}
From which the surjectivity of (\ref{equ: lgln-rest}) follows (note that all the maps in the right hand set of (\ref{equ: lgln-rest}) are induced from the composition (\ref{equ: lgln-keymap1})). This concludes the proof of Theorem \ref{thm: lgln-main}.
\begin{remark}\label{rem: lgln-lg}
  (i) Using the fact that $\tilde{\Sigma}(\alpha,\ul{\lambda}, \cL(\rho_L))\cong \Sigma(\alpha, \ul{\lambda},\cL(\rho_L))\oplus_{\Sigma(\alpha,\ul{\lambda})} \St_n^{\an}(\alpha, \ul{\lambda})$, and (\ref{equ: lgln-bij0}) (see also Remark \ref{rem: Ressubq}), one can actually deduce from (\ref{equ: lgln-main}) an similar bijection with $\Sigma(\alpha, \ul{\lambda},\cL(\rho_L))$ replaced by $\tilde{\Sigma}(\alpha,\ul{\lambda}, \cL(\rho_L))$. And the statement in Theorem (2) also holds if the part ``an injection $\Sigma_i(\alpha,\ul{\lambda}, \psi)\hookrightarrow  \Pi_{\infty}^{R_{\infty}-\an}[\fm_y]$" is replaced by ``a non-zero morphism  $\tilde{\Sigma}(\alpha,\ul{\lambda}, \cL(\rho_L))\ra  \Pi_{\infty}^{R_{\infty}-\an}[\fm_y]$". However, the author does not know if $\soc_G \St_n^{\an}(\alpha, \ul{\lambda})\cong \St_n^{\infty}(\alpha,\ul{\lambda})$, and hence does not know if any non-zero $\tilde{\Sigma}(\alpha,\ul{\lambda}, \cL(\rho_L))\ra \Pi_{\infty}^{R_{\infty}-\an}[\fm_y]$ is injective.
\\

\noindent (ii) The global context that we are working in is not important for the above arguments. Actually, a key step is the existence of the character $\widetilde{\chi}$, which is a consequence of the surjectivity of (\ref{equ: tangL}); we also use the smoothness of the patched eigenvariety at the point $x$. But as in the discussion in Remark \ref{rem: eig}, one can expect that these hold in more general setting. The rest of the arguments (assuming Lemma \ref{lem: lgln-nonc}) is irrelevant with the global context.
\end{remark}

\end{document}